\documentclass[a4paper, 11pt, final]{article}
\usepackage[left=2cm,right=2cm,top=2cm,bottom=2.5cm]{geometry}
\usepackage[utf8]{inputenc}
\usepackage[T1]{fontenc}
\usepackage[english]{babel}	
\usepackage{amsmath,amsthm,amsfonts,amssymb,amscd,mathtools}
\usepackage{array, booktabs, cite, float, footmisc, kvoptions,multicol, nag,
newtxmath, newtxtext, pdf14, pdftexcmds, ragged2e, upref, url, xcolor, xpatch, zref-base}
\usepackage{microtype}
\usepackage{tikz}
\usepackage{svg}
\usepackage{adjustbox}
\usepackage{upgreek}

\usepackage[backref=page,colorlinks=false]{hyperref}
\renewcommand*{\backref}[1]{}
\renewcommand*{\backrefalt}[4]{\quad\small
  \ifcase #1 (\textbf{NOT CITED.})%
  \or (Cited on page~#2.)%
  \else (Cited on pages~#2.)%
  \fi}
\usepackage{faktor}
\usepackage{xfrac}
\usepackage{cleveref}
\usepackage{graphics}
\usepackage{mathrsfs}
\usepackage[normalem]{ulem}
\usepackage{pgfplots}
\pgfplotsset{compat=1.18}
\usepackage{mathalfa}

\usetikzlibrary{cd}
\usepackage{libertine}
\def\Hol{\operatorname{Hol}}

\def\D{\mathcal{D} }

\def\Im{\mathfrak{Im} \ }
\def\Re{\mathfrak{Re} \ }
\def\R{\mathbb R}
\def\id{\mathrm{id}}

\def\T{\mathbb T}
\def\Q{\mathbb{Q}}
\def\Z{\mathbb Z}
\def\V{\mathbb{V}} 
\def\K{\mathbb{K}}
\def\S{\mathbb S}
\def\A{\mathbb{A}}
\def\N{\mathbb N}
\def\C{\mathbb C}

\def\E{\mathbb{E}}

\def\qand{\quad \text{and} \quad}

\def\e{\operatorname{Exp}}

\def\ev{\mathrm{even}}

\def\Green{\mathcal{G}}
\DeclareRobustCommand{\stirling}{\genfrac\{\}{0pt}{}}
\def\KK{\mathcal{K}_{\gamma,\tau}}
\def\AA{\mathcal{A}_{\gamma,\tau}}
\def\DC{\mathcal{D}_{\gamma,\tau}}
\def\Leb{\operatorname{Leb}}
\def\rmid{\rho^{\mathrm{mid}}}

\def\sig{\mathfrak{s}} 
\def\Vect{\operatorname{Vect}}
\renewcommand{\chi}{\upchi}
\theoremstyle{plain}
\newtheorem{thm}{Theorem}[section]
\newtheorem{theorem}[thm]{Theorem}
\newtheorem{proposition}[thm]{Proposition}
\newtheorem{fact}[thm]{Fact}
\newtheorem{lemma}[thm]{Lemma}
\newtheorem{corollary}[thm]{Corollary}
\newtheorem{definition}[thm]{Definition}
\newtheorem{remark}[thm]{Remark}
\newtheorem{example}[thm]{Example}

\newtheorem{question}[thm]{Question}
\newtheorem{theo}{Theorem}
\newtheorem*{theo*}{Theorem}
\newtheorem{coro}{Corollary}

\title{Jet rigidity of Mather's $\beta$-function and complexified KAM curves holomorphic in potential for analytic standard maps}
\author{Mathieu Helfter\footnote{Institute of Science and Technology Austria (ISTA); Am Campus 1, 3400 Klosterneuburg, Austria. Partially supported by ERC SPERIG $\#885707$. \ {\em Email address}: mathieu.helfter@ist.ac.at
}}
\date{\today}
\begin{document}
\selectlanguage{english}
\maketitle

\begin{abstract} 
Mather's $\beta$-function associates with each rotation number the least
average action carried by an orbit of a twist map. For analytic standard
maps in the KAM regime, we prove the following. For any family of $d \ge 1$ directions of perturbation satisfying a natural symmetry condition, the jet of the $\beta$-function  at a single algebraic Diophantine
rotation number locally determines the potential for an
open and prevalent, hence dense, set of base potentials in the KAM domain. In particular, we obtain that for a residual and prevalent set of even potentials in the KAM domain, every real analytic deformation  with finite Fourier support that preserves the full jet of the $\beta$-function at a fixed algebraic Diophantine rotation number is constant. These results are non perturbative within the KAM domain. 
  As an intermediate result of independent interest, we establish a KAM theorem giving joint $C^\infty$--holomorphic dependence of the invariant curves, and hence of
the $\beta$-function, on complex holomorphic potentials and on a domain of complexified rotation numbers whose boundary contains real Diophantine numbers. 
\end{abstract}

{ \small \tableofcontents}

\section{Introduction and results}\label{lieu:introduction}


For families of dynamical systems defined by a parameter, for instance a potential, it is natural to ask whether information about the orbits of the underlying dynamics, such as their length or action, allows one to recover the underlying parameter.
This kind of inverse problem was popularized by Kac's famous question ``Can one hear the shape of a drum?''. 
This precise question asks whether the spectrum of the Laplacian determines the domain inside which the sound propagates. While the answer turned out to be negative, thanks to the counterexample constructed by Gordon, Webb and Wolpert~\cite{GWW}, Kac's question nevertheless sowed the seeds of inverse problems in the dynamical community. 

Notably spectral-rigidity results have been obtained in symmetry-restricted classes.  Wave-trace invariants and the Laplace spectrum are closely related to the lengths of periodic trajectories.
 For instance, Zelditch in~\cite{Z2} proved rigidity for domains with Neumann or Dirichlet boundary conditions within the class of 
simply connected analytic $\mathbb Z_2$-symmetric planar domains. This was then extended to any dimension by Hezari and Zelditch~\cite{HezariHamid}. A breakthrough was then performed by the same authors in~\cite{HZ22}, essentially using the technology of~\cite{DKW}, where they have shown that an ellipse with small enough eccentricity is, up to isometries, uniquely spectrally determined among all smooth domains without any symmetry assumption. 

Such spectral information is encoded by the length of periodic trajectories. 
For discrete symplectic systems, the role of length of trajectories is replaced by their action. 
This article proposes to bring new techniques to address inverse problems in the discrete setting of area preserving maps displaying persistent invariant curves: the so-called KAM curves, named after Kolmogorov, Arnold and Moser. We consider a specific class of twist maps that are \emph{generalized standard maps}. Generalized standard maps describe the motion of a pendulum subjected to periodic impulses in discrete time. The impulse only depends on the angle and is given by a potential. Precisely, the corresponding map is 
\begin{equation}
F \colon (\theta ,r)  \mapsto   ( \theta + r +  V ' (\theta)  ,  r + V'(\theta) ) 
\end{equation}
 where $ \theta$ is an angle, $r$ the momentum and $ V$ the potential. This is a classical twist map of the cylinder. It admits a generating function in configuration space defined as 
 \begin{equation}
 \mathcal{L} (\theta , \theta ' ) = \frac{1}{2} (\theta' - \theta)^2 + V (\theta)  \;. 
\end{equation}   
Orbits correspond then to sequences of angles and momenta $ (\theta_n, r_n)_{ n \in \Z} $  such that $(\theta_n)_{n\in\Z}$ is a critical point of the formal action obtained by summing $\mathcal L(\theta_n,\theta_{n+1})$ over $n$.
An orbit is moreover said to be minimal, or to have minimal action, if it minimizes the average of $ \mathcal{L}$. 
Trajectories with minimal action are often physically realized and hence observable. Aubry--Mather theory organizes them according to their rotation number. Given a rotation number $\omega \in \mathbb{R}$, Mather's $\beta$-function is defined as
\begin{equation} \label{eq:def-beta-intro}
\beta(\omega) = \inf_{\{\theta_k\} \in \mathcal{C}_\omega} \liminf_{N \to \infty} \frac{1}{2N+1} \sum_{k=-N}^{+N}  \mathcal{L}(\theta_k, \theta_{k+1}),
\end{equation}
where $\mathcal{C}_\omega$ is the set of all successive angles $\{\theta_k\}_{k \in \mathbb{Z}}$ of an orbit of rotation number $ \omega$, that is, such that $$\lim_{|k| \to \infty} \frac{\theta_k}{k} = \omega . $$ 
 See for instance ~\cite{Bangert,mather2006action,siburg2004principle,sorrentino2015action} for introductions to Aubry--Mather theory and the $\beta$-function. More details and another definition based on invariant measures are also given in Section~\ref{subsec:aubry-mather}.
 
The general question this article addresses is whether the mapping
\begin{equation*}
V \longmapsto \beta
\end{equation*}
is one-to-one, or at least locally invertible, after removing trivial symmetries of the system, such as translations of the potential. 
We give the following answers.  For a fixed $\rho>0$, by analytic potentials we mean real-valued zero-mean functions on the circle $\T$ that extend holomorphically to the complex strip $ \{z\in\C/\Z:|\operatorname{Im}z|<\rho\}$. They form a Fréchet space denoted by $\Hol_\rho^\R(\T)$. See Section~\ref{subsec:potential}. 
In Theorem~\ref{thm:prevalent-affine} we show the existence of a universal neighbourhood $\mathcal U^\R$ of the zero potential such that an open and prevalent subset of $\mathcal U^\R\times(\Hol_\rho^\R(\T))^d$ consists of a base potential and perturbation directions for which the map assigning a potential to its $\beta$-function is locally injective on the corresponding affine space.
A stronger answer is given under a general symmetry assumption that includes for instance even potentials. 
We fix a set of analytic potentials $\mathscr{V}$ that is \emph{Fourier separated} and \emph{stable under Fourier truncation}. See Definitions~\ref{def:fourier-separated} and \ref{def:truncation}.
Then in Theorem~\ref{thm:main-rig}, we  show that for a fixed set of directions of perturbations in $\mathscr{V}$ of finite dimension the set of base potentials in $ \mathcal{U}^\R\cap \mathscr{V}$  for which the $ \beta$-function is locally functionally injective on the corresponding affine space is open and dense in $\mathcal{U}^\R \cap \mathscr{V}$. 
 Even more, we are actually able to observe this rigidity phenomenon locally at a given algebraic number verifying some usual Diophantine condition. 
Indeed the proof relies on a \emph{``fully holomorphic''} KAM
theorem (Theorem~\ref{thm:hol-KAM}). The neighborhood $ \mathcal{U}^\R$
mentioned above is actually provided by the same theorem without any
additional shrinking.
 
In the case of an area-preserving twist map of the annulus, the classical
KAM theorem takes a particularly simple form. Let us first recall that such
a system is said to be integrable when the annulus is foliated by invariant
circles, on each of which the dynamics is a rigid rotation whose rotation
number varies from one circle to another. The classical KAM theorem states
that, for a fixed set $\mathcal D$ of rotation numbers satisfying a
classical Diophantine condition, see~\eqref{eq:dc}, under a sufficiently
small analytic exact symplectic perturbation of the system, the circles
with rotation number in $\mathcal D$ persist as nearby analytic invariant
circles. Then a Cantor subset of the annulus remains integrable in the
sense that it is foliated by invariant curves with these rotation numbers.
The parametrizations of these curves depend $C^\infty$-Whitney-smoothly
on the rotation number. See~\cite{Poschel,Poschel2001}.

While the destruction of rational invariant curves cannot generically be
avoided, a stronger regularity of the surviving curves is provided, in the
case of standard maps at least, by considering a closed complex extension
$\mathcal A$ of $\mathcal D$ such that
$\mathcal D\subset\partial\mathcal A$. We prove that, in the interior of
$\mathcal A$, the extended parametrizations are holomorphic in the usual
sense and satisfy the invariance equation of the complexified standard
map. Moreover, all their derivatives with respect to the rotation number
extend continuously to the boundary and satisfy the Whitney conditions.
This complex extension is also essential to our proof of a second
holomorphic dependence, namely holomorphic dependence on complex
potentials in the usual Fréchet sense.

As a direct application, Theorem~\ref{thm:kam-main-beta} shows that the
normalized $\beta$-function admits an extension which is tame
holomorphic in the potential $V$ and $C^\infty$--Whitney holomorphic in
the complexified rotation variable $q = \exp (2i \pi \omega) $ where  $ \omega \in \mathcal{A}$. In particular, its jet
at every Diophantine rotation number is well defined.

This regularity presents a natural connection with \emph{linear response} as introduced by Ruelle~\cite{Ruelle}. Linear response concerns the
differentiability, under perturbations of the dynamics, of averages of fixed observables against the associated invariant measures.
See also the surveys~\cite{Ruellereview,Baladi} for more details.
To make this relation precise in our context, fix
$V\in\mathcal U\cap\Hol_\rho^\R(\T)$ and
$\omega\in\mathcal D\cap[0,1)$, and put 
$q:=\exp(2\pi i\omega)$. Let $h_{V,q}$ be the angular parametrization
provided by the KAM theorem, chosen so that the corresponding
parametrization of the invariant curve conjugates the restriction of
$F_V$ to the rotation $\tau_\omega:\theta\mapsto\theta+\omega$ on $\T$.

Let $\widetilde h_{V,q}$ be any lift of $h_{V,q}$. On the cylinder of
configurations $  \mathcal C :=
\mathbb R^2/
\left\{
(\theta,\theta')\sim(\theta+k,\theta'+k):k\in\mathbb Z
\right\}$,
consider the map
\begin{equation*}
\begin{array}{rcl}
\iota_{V,q}\colon\T &\longrightarrow& \mathcal C\\[2pt]
\theta &\longmapsto&
\left[
\widetilde h_{V,q}(\theta),
\widetilde h_{V,q}(\theta+\omega)
\right].
\end{array}
\end{equation*}
If $\widetilde F_V$ denotes the dynamics on $\mathcal C$ induced by
$F_V$, the invariance equation gives
\begin{equation*}
\widetilde F_V\circ\iota_{V,q}
=
\iota_{V,q}\circ\tau_\omega.
\end{equation*}
Consequently, the pushforward measure
$\mu_{V,q}:=(\iota_{V,q})_*\Leb_\T$ is a
$\widetilde F_V$-invariant ergodic probability measure, since $\omega$
is irrational. Note that this measure does not depend on the chosen lift.
We then write
\begin{equation*}
\Phi(V)(q)
:=
\beta_V(\omega)-\frac{\omega^2}{2}
=
\int_{\mathcal C}
\left(\mathcal L_V-\frac{\omega^2}{2}\right)\,d\mu_{V,q}.
\end{equation*}

Holomorphy, thus differentiability, in the potential provides a direct
interpretation in the sense of linear response. Indeed, for every perturbation direction
$W\in\Hol_\rho^\R(\T)$, a simple computation  provides that
\begin{equation*}
D\Phi(V)[W](q)
=
\int_{\mathcal C}W(\theta_0)\,d\mu_{V,q}.
\end{equation*}
Consequently, for another direction $H$,
\begin{equation*}
D^2\Phi(V)[H,W](q)
=
\left.\frac{d}{dt}\right|_{t=0}
\int_{\mathcal C}W(\theta_0)\,d\mu_{V+tH,q},
\end{equation*}
which is the linear response of the average of the fixed observable
$W $  in the direction $H$. The tame holomorphic dependence on
$V$ therefore provides response coefficients of every order, together
with their $C^\infty$--Whitney-holomorphic dependence on $q$.
A related  result  on  linear response in the KAM setting for invariant measures of circle diffeomorphisms was obtained by Galatolo and
Sorrentino~\cite{GS}.
It is however important to note that, for non real potentials or non real rotation number, there is no interpretation in terms of invariant measures. 

This precise regularity enables our rigidity statements. 
Theorem~\ref{thm:main-rig} provides that for every $d \ge 1$ and for an open and dense choice of base potentials in $\mathcal U^\R$, the  $(d-1)$-jet of the $ \beta$-function at a single Diophantine and algebraic rotation number is locally injective on a fixed subspace of $ \mathscr{V}$ of directions of perturbations. Moreover on this affine space, mapping a potential to the corresponding jet of the $\beta$-function is 
a local real analytic diffeomorphism outside its critical set. In particular level sets on this space of perturbations have zero measure. A third answer is obtained by a direct application in
Example~\ref{ex:def-rig}. It provides that for a residual and prevalent set of even potentials in $ \mathcal{U}^\R$, every real analytic deformation  with finite Fourier support and which preserves the full jet of the $\beta$-function at a fixed algebraic Diophantine rotation number is constant.
In Corollary~\ref{cor:def-rig-gen} we provide a more general statement that directly provides such examples. 

A last independent result in Theorem~\ref{thm:one-frequency-rigidity} gives that on the KAM domain $ \mathcal{U}^\R$, two potentials that share the same $ \beta$--function simply evaluated at any fixed Diophantine rotation number, for all nearby perturbations, must be equal.  

These rigidity results echo a long standing tradition of rigidity problems in symplectic dynamics. Systems on surfaces such as billiard maps, geodesic flows or mechanical systems share common features with the standard map and present strong motivation for this work and future outcomes of the developed techniques. 
In short, a billiard map describes the dynamics of a point moving freely inside a convex planar domain and being reflected  elastically. It is structurally closer to Kac's initial problem than standard maps.
The geometry of the domain replaces the potential of generalized standard maps. A related question to the inverse problem of Kac is Birkhoff's famous conjecture. It claims that ellipses are the only integrable strictly convex planar billiards. This conjecture is at the heart of spectral invariance questions.  For generalized standard maps, the corresponding conjecture is false. Suris  in~\cite{suris} exhibited non trivial potentials, now known as Suris potentials, for which the corresponding dynamics are integrable. However, the rigidity question near such potentials has a positive answer. Fierobe and Tsodikovich in~\cite[Corollary 1]{FT} have shown that, for Suris potentials of sufficiently small eccentricity, sufficiently small deformations preserving the $\beta$-function must stay in the class of Suris potentials.

The present study for generalized standard maps leads naturally to the question of whether the $\beta$-function of a billiard map determines the boundary of a convex domain.  In that direction let us mention De Simoi, Kaloshin and Wei~\cite{DKW} who proved dynamical spectral rigidity near the circle in the $\mathbb Z_2$-symmetric class. Then Huang, Kaloshin and Sorrentino~\cite{HuangKaloshinSorrentino2018} have shown that, for a generic strictly convex billiard domain, the maximal marked length spectrum determines the Lyapunov exponent of Aubry--Mather sets for periodic orbits. Building on~\cite{DKW}, Fierobe, Kaloshin and Sorrentino~\cite{fierobe2025billiard} extended deformational spectral rigidity from nearly circular domains to nearly elliptical domains under the stronger assumption of dihedral symmetry.
 Koval~\cite{koval2026local} considers  a stronger notion of
integrability allowing one to prove a local version of
Birkhoff's conjecture without any symmetry condition. 
 
 For billiard maps, the relation between the $\beta$-function and the maximal marked length spectrum is developed further in~\cite{Sor14,Bialy2022}. In the very recent work of Fierobe-Kaloshin-Trujillo ~\cite{FKT} the first steps of rigidity of the Mather $\beta$-function for billiard maps are addressed.

For geodesic flows on closed surfaces, the standard rigidity concerns whether the lengths of closed geodesics, marked by their free homotopy classes, determine the metric of a surface.
In this setting, Mather's $\beta$-function corresponds to minimal action as a function of the homological rotation vector, whereas the marked length spectrum corresponds to minimal geodesic length in each free homotopy class. The relation between the $\beta$-function and integrability of the flow was studied by Massart and Sorrentino~\cite{MS}; see also~\cite{SVeselov}.
Guillarmou, Lefeuvre and Paternain in~\cite{guillarmou2025marked}, following a previous result of Guillarmou and Lefeuvre~\cite{GL}, have shown that, when the genus is at least $2$, two metrics with the same marked length spectrum and Anosov geodesic flows are isometric via an isometry isotopic to the identity.
It is then natural to ask what happens outside of the Anosov regime: near a nondegenerate elliptic closed geodesic, is the germ of the metric determined by the $\beta$-function of a transverse Poincar\'e return map? In that direction Abbondandolo and Mazzucchelli in~\cite{AbbondandoloMazzucchelli} proved, under strong symmetry assumptions, rigidity for a notion of marked length spectrum that they introduce.

Back on generalized standard maps, the regularity of the maps involved must be made precise. Two types of regularity are of particular interest: the dependence on the rotation number and the dependence on the potential $V$. Under the twist condition and near the integrable regime, for a rotation number satisfying a Diophantine condition, the corresponding invariant curve supporting orbits with that prescribed number persists under small perturbations. These curves are the so-called KAM curves, named after Kolmogorov, Arnold and Moser. The appropriate regularity of their parametrization is usually Whitney regularity, as shown by Lazutkin~\cite{lazutkin1973existence} and then Pöschel~\cite{Poschel}.

In another direction, and as an inspiration for the present work, Carminati, Marmi and Sauzin~\cite{CMS} have shown that, for sufficiently small analytic generalized standard maps, these KAM curves depend $C^1$-holomorphically on the rotation number on an appropriate complex extension of a fixed set of Diophantine numbers. This was then used by Carminati, Marmi, Sauzin and Sorrentino~\cite{CMSS} to show that the $\beta$-function is $C^1$-holomorphic on the same complex extension of rotation numbers. This $C^1$-holomorphy is unfortunately insufficient to solve the precise inverse problem in which we are interested. However, their techniques, based on the Lagrangian formulation of the KAM theorem given by Levi--Moser~\cite{levi2001lagrangian}, provide the foundation for obtaining a $C^\infty$-holomorphic extension of the $\beta$-function. This means that it is holomorphic in the interior of the complex Diophantine set and that all its derivatives extend to the boundary as a $C^\infty$-Whitney jet.

Such regularity was conjectured in their paper and reconciles their result with Lazutkin and Pöschel's point of view on the regularity of KAM curves. It comes with an additional and necessary improvement: tame holomorphic dependence on the potential, with holomorphy understood in the usual sense for maps between Fréchet spaces. We also stress that the density of resonances obstructs ordinary holomorphy. Carminati, Marmi, Sauzin and Sorrentino have also shown the quasianalyticity of this extension of the $\beta$-function. That is, its values on a set of positive measure within the set of Diophantine numbers lying on the boundary of the complex extension determine the whole $\beta$-function.

This precise regularity enables arguments based on infinite jets of the $\beta$-function at a given rotation number and on the shyness of exceptional potentials. Generalizing this KAM theorem to other settings, such as billiard maps or suitable Poincaré maps, is natural and would enable similar rigidity results. In this direction, let us mention the very recent work of Fierobe--Kaloshin--Trujillo~\cite{FKT}, where, for billiard maps, they prove a KAM theorem similar to that of~\cite{CMS}, but with $C^\infty$-holomorphic dependence on the rotation number over a similar complex extension. Holomorphic dependence on the analytic billiard domain remains open.

\paragraph{Generalized standard maps}

Let us recall that a \emph{twist map} of the cylinder $\A := \T \times \R$ is an exact
symplectic diffeomorphism
$$ F : (\theta_0, r_0) \mapsto (\theta_1, r_1)$$ 
which admits  a \emph{generating function} $\mathcal L:\R^2\to\R$. This means that $F (\theta_0 , r_0 ) = (\theta_1, r_1) $ if and only if 
\begin{equation*}
  r_0 = -\partial_{\theta_0} \mathcal{L}(\theta_0,\theta_1) \qand 
  r_1 = \partial_{\theta_1} \mathcal{L}(\theta_0,\theta_1),
\end{equation*}
with \emph{twist condition} $\partial_{\theta_0}\partial_{\theta_1}
\mathcal L<0$.
 We assume furthermore that the generating function, and thus the dynamics in configuration space,  is invariant under this diagonal $\Z$-action, and therefore descends to $\mathcal{C} :=(\R\times\R)/\Z$ under the $\Z$-action $ k\cdot(\theta_0,\theta_1):=(\theta_0+k,\theta_1+k)$ for $ k\in\Z$ and $ (\theta_0,\theta_1) \in \R \times \R$.
 
In this setting, orbits of $F$ correspond to critical points of the formal action
$\sum_n \mathcal{L}(\theta_n, \theta_{n+1})$.

A \emph{generalized standard map} $F_V$ given by a periodic potential $ V : \T \to \R$ that is at least differentiable is then given by
\begin{equation}\label{def:standard-map} 
 F_V (  \theta , r) =  ( \theta + r +  V ' (\theta)  ,  r + V'(\theta) ) 
\end{equation}
Then, its generating function takes the form 
\begin{equation*} \mathcal{L}_V(\theta_0,\theta_1) = \tfrac{1}{2}(\theta_1 - \theta_0)^2 + V(\theta_0),
\end{equation*}

Then, recall that a \emph{KAM curve} of rotation number $\omega  \in  \R \setminus \Q$ is an
$F$-invariant embedded circle $\Gamma \subset \A$ on which the dynamics is conjugate to
the rigid rotation $T_\omega : \theta \mapsto \theta + \omega$.  
The existence and destruction of such curves as $V$ is perturbed are well known. Their persistence was first studied by Chirikov with the Chirikov  standard  map (or simply standard map) where the potential is simply $ K \cos (2\pi\theta) $ with $ K$ as a free parameter. This model was first introduced by Chirikov in \cite{chirikov1971research}. We have known, since the work of  Moser~\cite{moser1962invariant}, that, for not too large parameters, invariant curves with Diophantine rotation numbers persist. 
\begin{figure}[H]
    \centering
     \includegraphics[width=1\linewidth]{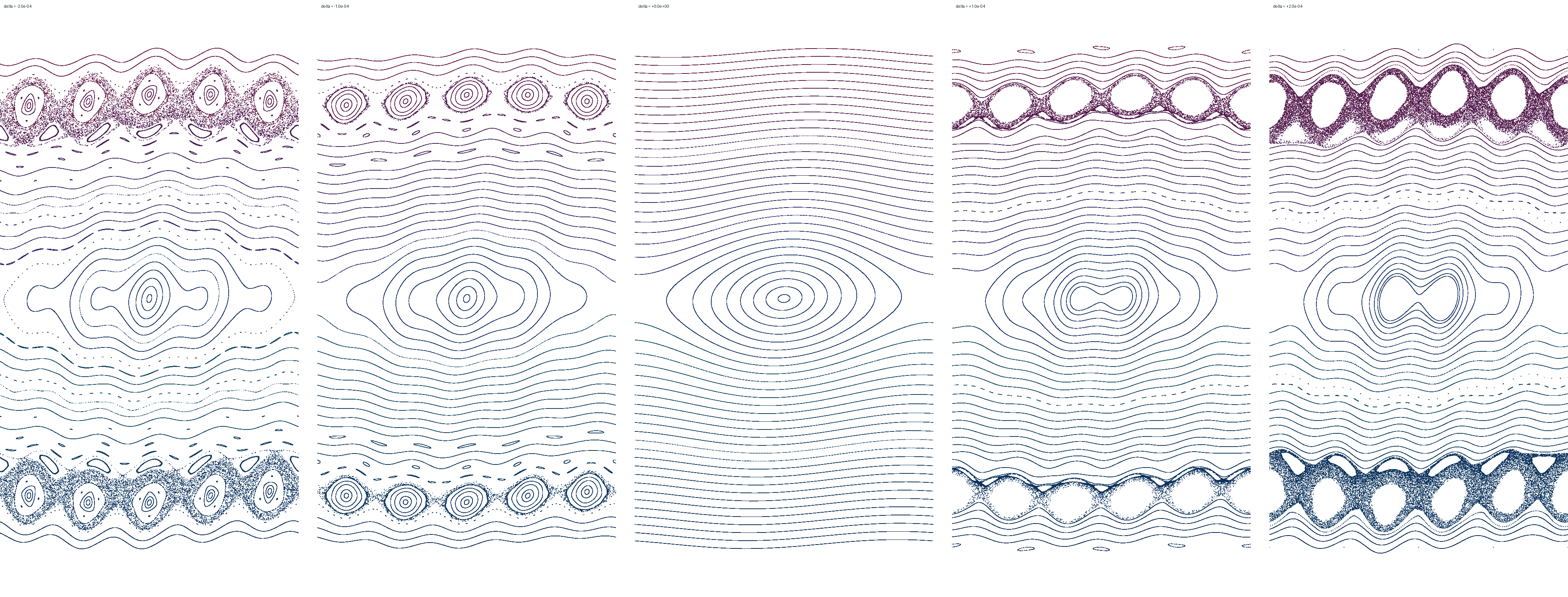}
\caption{Phase portraits of the generalized standard maps associated with the even, zero-mean potentials
\(
V (\theta)
= \frac{1}{2}\log\!\left(
1 - 2\cdot 10^{-3}\cos(2\pi\theta) + 10^{-6}
\right)
+ \delta\cos(10\pi\theta), 
\)
from left to right, with parameter values \(\delta \) taking successive values of $ (-2,-1,0,1,2) \cdot 10^{-4} \).
}
  \label{fig:kam}
\end{figure}

Figure~\ref{fig:kam} illustrates the deformation of curves with small perturbations. The potential is chosen to be even, with zero mean and all its Fourier coefficients nonzero while being not too far from the zero potential. The perturbation is even with  a single Fourier mode. 
 Let us just observe how the geometry visually shifts as the parameter $\delta$ varies.

We now reduce geometric information of invariant curves with prescribed rotation number of a twist map $F$  by considering the $ \beta$-function as given in \eqref{eq:def-beta-intro}.  
Mather proved that, for every exact twist map, the $\beta$-function is
differentiable at each irrational rotation number, whereas for a generic
such map it is not differentiable at rational rotation numbers
\cite{mather1990differentiability}. This apparent difficulty in obtaining a holomorphic $\beta$-function is overcome in the present KAM theorem.

\paragraph{Space of zero-mean analytic potentials} \label{subsec:potential}
Let us introduce a few notations and specify the set of potentials we consider. 
For $  \rho >0 $, set
\begin{equation*}
    \T_\rho := \left\{z \in \C / \Z : \lvert \Im z \rvert < \rho \right\}.
\end{equation*} 
This is a cylinder around the circle $  \T$ in the complex set $ \C / \Z $ as displayed in Figure~\ref{fig:hol-ext}. 
 
For a function $V$ that is holomorphic on $\T_\rho$ and for $0<\rho'<\rho$, set
\begin{equation*}
    |V|_{\rho'} := \sup_{z\in\T_{\rho'}} |V(z)|.
\end{equation*}
\begin{figure}[H]
    \centering
 \includegraphics[width=0.25\linewidth]{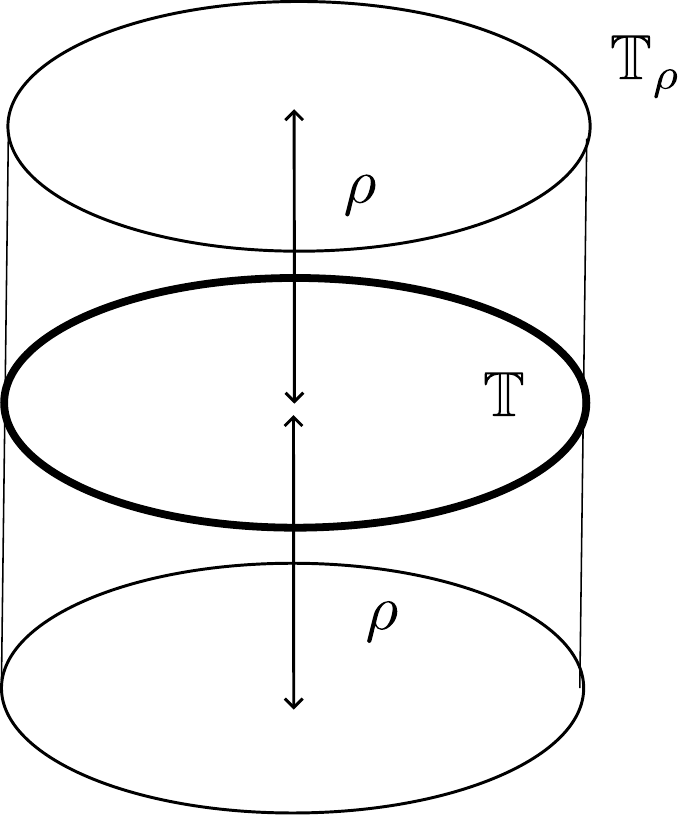}
    \caption{Representation of the cylinder $\T_\rho$ in $\C/\Z$.}
    \label{fig:hol-ext}
\end{figure}

We normalize the Haar measure of $\T=\R/\Z$ to have total mass one. For
every integrable function $f$ on $\T$, the notation
$\int_\T f$ therefore denotes its corresponding mean for that normalized Haar measure.
 For aesthetic purposes, depending on the
context, we interchangeably use the notation 
\begin{equation*}
    \langle f\rangle:=\int_\T f
\end{equation*}
for the very same quantity.

 We denote by $\Hol_\rho(\T)$ the space of \emph{zero-mean
analytic potentials} admitting an extension to $ \T_\rho$. These are potentials $ V : \T \to \C$ satisfying the following conditions. 
\begin{enumerate}
    \item[(a)] \textsc{Zero mean. }$\langle V\rangle=0$, 
    \item[(b)] \textsc{Holomorphicity. } $V$ admits a holomorphic extension to $\T_\rho$.
\end{enumerate}
While the real-valued potentials are the only ones that are dynamically meaningful, it is necessary for running the holomorphic KAM theorem to consider complex potentials in full generality.  
The set $ \Hol_\rho (\T)$ is a Fréchet space  with topology generated by the family of seminorms $( |\cdot|_{\rho'})_{  0 < \rho' < \rho} $. 

The corresponding subspace of real potentials is then denoted by 
\begin{equation*}
    \Hol_{\rho}^{\R}(\T)
    :=\left\{V\in\Hol_\rho(\T):V(\theta)\in\R
    \text{ for every }\theta\in\T\right\}
\end{equation*}
which is also a Fréchet space with the induced canonical topology. 

We propose the following notion of rigidity of the $\beta$-function in the precise setting of generalized standard maps. 
\begin{definition}[$\beta$-rigidity]
A subset $S\subset\Hol_\rho^\R(\T)$ is \emph{$\beta$-rigid} if the map
$V\mapsto\beta_V$ is injective on $S$. It is \emph{locally $\beta$-rigid
at $V_0\in S$} if this map is injective on $S\cap\mathcal O$ for some
neighbourhood $\mathcal O$ of $V_0$.
\end{definition}
Obviously one could replace the potential $V$ for instance by the generating function of a more general symplectic map or a Hamiltonian in continuous time settings. 
 To our knowledge, no general nontrivial local recovery theorem from Mather's $\beta$-function to the underlying parameter of the system has yet been proved.  
  The present paper shows that such locally rigid families
are abundant within the class of generalized standard maps. 
This must be expressed in the infinite-dimensional space of potentials. To replace the notion of almost everywhere in finite dimension, we use the notion of \emph{prevalence} as introduced by Hunt--Sauer--Yorke in~\cite{hunt1992prevalence}.
 Recall that for an ambient topological vector space $X$ (e.g. a Fréchet space), a Borel subset $A\subset X$ is said to be \emph{shy} if
there exists a compactly supported Borel probability measure $\mu$ on $X$ such that $\mu(x+A)=0$ for every $x \in X$.  Conversely, a subset $ \mathcal{P}$ is said to be prevalent in $X$ if its complement is shy in $X$. 
If $\mathcal O\subset X$ is open, we say that $\mathcal P\subset\mathcal O$ is prevalent in $\mathcal O$ when $\mathcal O\setminus\mathcal P$, regarded as a subset of $X$ and thus empty outside $\mathcal O$, is shy in $X$. We also recall that shy Borel subsets of a separable Fréchet space form a $\sigma$-ideal.
It was proven in \cite{hunt1992prevalence} that a  shy set has empty interior. Hence any (open) prevalent set is (open and) dense.

\begin{theo}
\label{thm:prevalent-beta-rigidity}
There exists a neighbourhood $\mathcal U^\R$ of the zero potential such
that, for every fixed $d \geq 1$, there is an open and prevalent subset of  $\mathcal U^\R\times(\Hol_\rho^\R(\T))^d$, consisting of tuples $(V_0,h^{(1)},\ldots,h^{(d)})$ such that the affine subspace $V_0+\Vect_\R(h^{(1)},\ldots,h^{(d)})$ is locally $\beta$-rigid at $V_0$.
\end{theo}

In particular, within each such affine space, the Mather $\beta$-functions of two potentials sufficiently close to the base potential $V_0$  coincide only if
the potentials are equal. The exceptional subset, regarded as a subset of $(\Hol_\rho^\R(\T))^{d+1}$, is shy by Theorem~\ref{thm:shyness}, since it is the zero set of a real-analytic map on a connected open subset of this Fréchet space.

The neighbourhood $\mathcal U^\R$ is the real part of the neighbourhood provided by the KAM Theorem~\ref{thm:hol-KAM}. In particular, the rigidity argument requires no further shrinking of the KAM domain.
Taking an infinite sequence of spaces of perturbations allows one to obtain the following corollary.

\begin{coro}
\label{cor:simultaneous-affine}
There is a residual and prevalent subset of  $\mathcal U^\R\times(\Hol_\rho^\R(\T))^{\N^*}$, consisting of choices of a base potential
$V_0\in\mathcal U^\R$ together with an infinite sequence of analytic 
perturbation directions, such that every finite initial collection of
these directions generates a locally $\beta$-rigid affine family at
$V_0$.
\end{coro}

Here the product is endowed with its product Fréchet topology, and prevalence is understood in  the ambient space $(\Hol_\rho^\R(\T))^\N$. 
Its proof is given after Theorem~\ref{thm:prevalent-affine}.

As mentioned earlier, rigidity need not be detected from $\beta$ as an entire function of the rotation number. It surprisingly appears that it can be detected locally
at a single rotation number. This leads us to the following paragraph.

\subsection{Diophantine sets and their complex extensions}
We fix $\tau>0$ and
$0<\gamma $, and consider the usual set of
$(\gamma,\tau)$-Diophantine numbers
\begin{equation} \label{eq:dc}
    \DC:=\left\{\omega\in\R:
    \operatorname{dist}(m\omega,\Z)
    \geq \frac{\gamma}{m^{1+\tau}}
    \text{ for every }m\in\N^*\right\}.
\end{equation} 
This set is invariant under translation by integers. 
 It is well known that
$\bigcup_{\gamma>0}\DC$ has full Lebesgue measure modulo $\Z$. Its
complement has measure zero and contains all Liouville numbers. 
Under the condition that $ \gamma < \gamma^\ast ( \tau) := \frac{1}{2 \zeta (1 + \tau) }$, where $ \zeta$ denotes the usual Zeta function, it is well known that $\DC$  has positive measure. 
A much celebrated KAM result, due to Lazutkin and then Pöschel~\cite{lazutkin1973existence,Poschel,Poschel2001},  states that, for a twist map $F$  near the integrable regime, a smooth invariant curve corresponding to a rotation number  $\omega \in \DC $ persists under sufficiently small $C^\infty$--perturbations of $F$. Moreover, the smooth parametrization of the curve depends smoothly on $\omega$ in the Whitney sense.
Now we choose to  follow~\cite{CMS} to prove holomorphy instead of $C^\infty$ dependence. 
This goes by extending $\DC$ to the subset of $\C$
\begin{equation*}
    \AA := \left\{ z \in \C \colon |\Im z| \ge d(\Re z,\DC) \right\},
\end{equation*}
where $d(\cdot,X)$ denotes the distance to a subset $X\subset\R$.

We then consider its exponential compactification on the Riemann sphere,
\begin{equation*}
\KK := \e(\AA) \cup \left\{0,\infty\right\},
\end{equation*}
where $\e(z):=e^{2\pi iz}$. See Figure~\ref{fig:diamonds}.
Note that $\KK$ is compact and pathwise connected. 
We then denote by $\Hol^\infty(\KK)$ the space of functions holomorphic in the
interior of $\KK$  and admitting a $C^\infty$--Whitney extension to the boundary.  The definition of this space and its topology are detailed in Section~\ref{lieu:top}.

\begin{figure}[H]
    \centering     
    \includegraphics[width=0.35\linewidth]{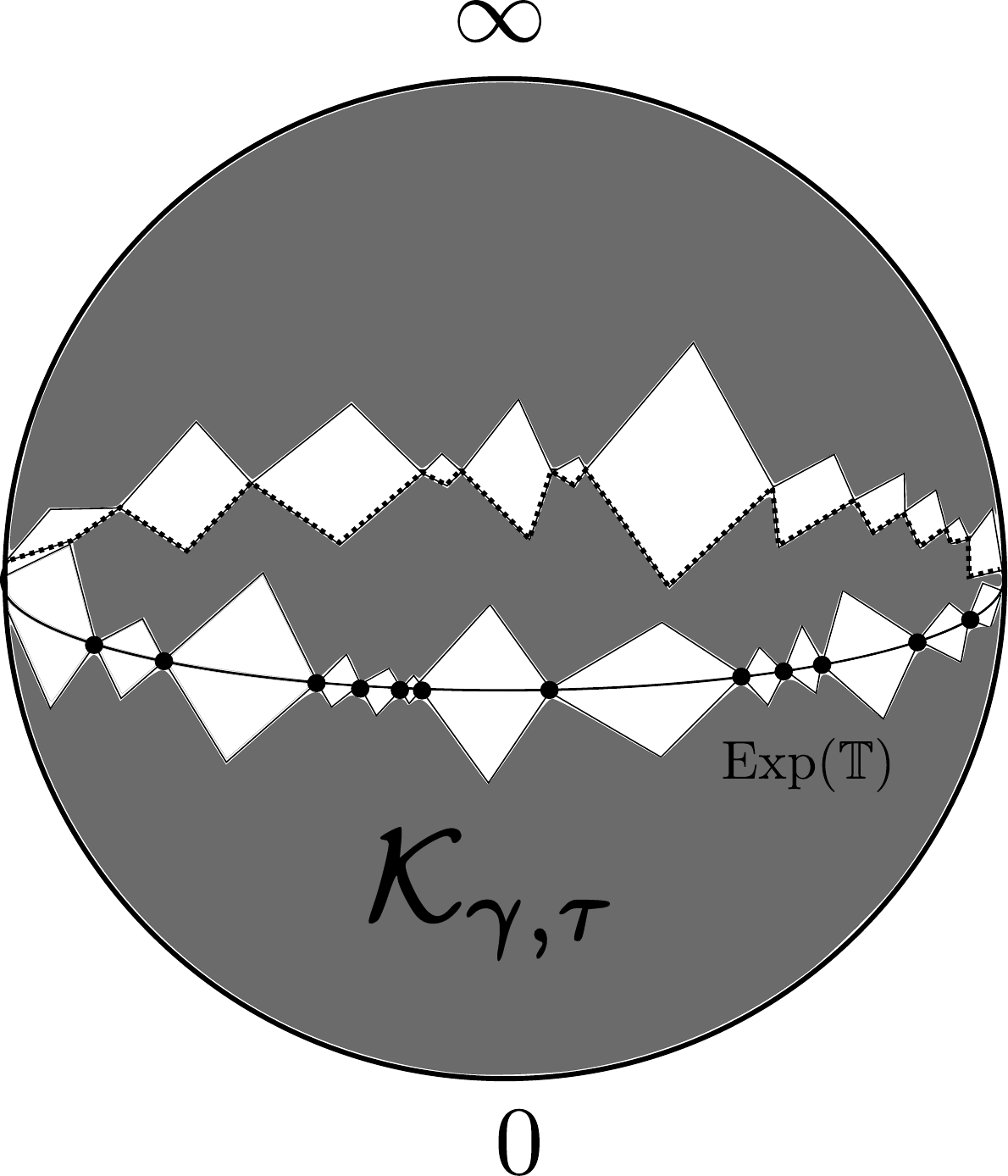}
    \caption{The set $\KK$, obtained from the complex sphere by removing
    ``diamonds'' near resonances.}
    \label{fig:diamonds}
\end{figure}

\subsection{The KAM theorem}


For generalized standard maps near the zero potential, Carminati--Marmi--Sauzin~\cite{CMS} proved that the parametrization of KAM curves admits a $C^1$-holomorphic extension on the set $ \KK$. This work was actually motivated by a question of  Kolmogorov~\cite{kolmogorov1954general} on the existence of an extension of the parametrization of KAM curves to a complex domain of rotation numbers which is monogenic in the sense of Borel.
Then, Carminati--Marmi--Sauzin--Sorrentino~\cite{CMSS}  used this regularity property of KAM curves to construct a complex extension of Mather's $\beta$-function to $\KK$ that is $C^1$-holomorphic in the complex exponentially lifted rotation number in $ \KK$.
An intermediate result of this article, which is also of independent
interest, gives the $C^\infty$--Whitney-holomorphic version of the result in \cite{CMS}. 
\begin{theo*}[Informal statement of Theorem~\ref{thm:hol-KAM}]
Fix $\tau>0$, $0<\gamma<\gamma^\ast(\tau)$, and
$\rho>\rho_\infty>0$. There exists a
neighbourhood $\mathcal{U} \subset \mathrm{Hol}_\rho(\mathbb{T})$ of the
zero potential and a map 
\begin{equation*}
  U  : \mathcal{U} \times \KK \longrightarrow
  \mathrm{Hol}_{\rho_\infty}(\mathbb{T})
\end{equation*}
which depends $C^\infty$--Whitney-holomorphically on the ``complexified rotation number'' 
$q\in\KK$ and is tame
holomorphic in the potential $V\in\mathcal U$ in the Fréchet sense. 
For
$V\in\mathcal U\cap\Hol_\rho^\R(\T)$ and $\omega\in\DC$, the map
$ \id_\T+U(V,e^{2\pi i\omega})$ is an
orientation-preserving analytic circle diffeomorphism, and the corresponding invariant graph is a KAM curve of rotation number $\omega$.
\end{theo*}

Relative to
\cite{CMS}, this KAM theorem has two separate gains:
\begin{enumerate}
  \item it upgrades $C^1$-holomorphic frequency dependence to
    $C^\infty$--Whitney-holomorphic dependence;
  \item we establish tame holomorphic dependence on the full infinite-dimensional
    space of analytic potentials, rather than only the scalar families
    $\epsilon V$.
\end{enumerate}
Let us, however, comment that, for the rigidity statements, only
Whitney-smooth dependence on the Diophantine rotation numbers would be
required. Working with $\KK$ instead not only provides a much stronger
result but also simplifies the proof, since its quasiconvexity allows us
to work with derivative seminorms instead of the usual Whitney
seminorms and yields natural definitions of the Fréchet spaces in which
the parametrization must lie.

As a corollary, in that
KAM neighbourhood of the zero potential, we improve the result in \cite{CMSS} to the same setting as follows. 
\begin{theo}
\label{thm:kam-main-beta}
Fix $\tau,\rho>0$ and $0<\gamma<\gamma^\ast(\tau)$. Then there exist a
neighbourhood $\mathcal U\subset\Hol_\rho(\T)$ of the zero potential and
a functional
\begin{equation*}
  \Phi:\mathcal U\longrightarrow\Hol^\infty(\KK)
\end{equation*}
that is tame holomorphic. For every
$V\in\mathcal U\cap\Hol_\rho^\R(\T)$ and every $\omega\in\DC$,
\begin{equation*}
  \beta_V(\omega)=\frac{\omega^2}{2}+\Phi(V)\left(e^{2\pi i\omega}\right).
\end{equation*}
Thus $\Phi$ is defined and holomorphic on the complex KAM domain, while
its restriction to real potentials extends the normalized function
$\beta_V(\omega)-\omega^2/2$.
\end{theo}

Throughout the paper, $\mathcal U$ denotes the neighbourhood in
Theorem~\ref{thm:kam-main-beta} (actually constructed in the precise statement of the KAM
Theorem~\ref{thm:hol-KAM}). 
Here $\Phi$ is a tame holomorphic map between Fréchet spaces. We retain this notation all along the text.  See Section~\ref{lieu:top}. 

We shall also mention, as shown in \cite{CMS}, that even in the $C^1$-holomorphic setting, the circle acts as a natural boundary for the domain of analyticity in the complexified rotation number $ q \in \KK$  of the parametrization of the KAM curve.

\subsection{Jet-rigidity}
Let $\mathcal U$ be the neighbourhood of the zero potential in Theorem~\ref{thm:kam-main-beta}, and set
\begin{equation*}
    \mathcal{U}^\R :=\mathcal U \cap \Hol_\rho^\R (\T) .
\end{equation*}

For local injectivity of the $\beta$-function
to hold on  an affine subspace of perturbations, it appears to be necessary to impose some nondegeneracy condition. This leads to introducing the following.

Before stating the main result, Theorem~\ref{thm:main-rig}, we shall introduce and recall a few notations and definitions. 
Denote by $(e_n)_{n\in\Z}$ the canonical Fourier basis on $\T$ and for a map $ h \in \Hol_\rho (\T)$, write its Fourier decomposition as  $$ h = \sum_{n \in \Z} \hat{h}_n e_n . $$ 
\begin{definition}[Fourier-separated subspace]
\label{def:fourier-separated}
A closed real linear subspace
$\mathscr{V}\subset\Hol_\rho^\R(\T)$ is said to be 
\emph{Fourier-separated} if for every integer $d \ge 1$, for every
subspace $\V\subset\mathscr{V}$ of dimension $d$, there exist a
basis $(h^{(1)},\ldots,h^{(d)})$ of $\V$, pairwise distinct positive
integers $n_1,\ldots,n_d$, and phases
$\xi_1,\ldots,\xi_d\in\mathbb S^1$ such that
\begin{equation}
\label{eq:fourier-separation}
\det\left[
\Re \!\left(
\xi_k \,\widehat{h^{(j)}}_{n_k}
\right)
\right]_{1\le  k,j\le  d}\neq0.
\end{equation}
\end{definition}
Note that condition~\eqref{eq:fourier-separation} is independent of the chosen basis 
of $\V$. Also, if $\mathscr{V}$ has finite dimension, then it is sufficient to check the condition for $\V=\mathscr{V}$. The closedness assumption ensures that
$\mathscr{V}$, with the induced topology, is again a Fréchet space.

\begin{definition}[Stability under Fourier truncation] \label{def:truncation}
We say that a closed real vector space $\mathscr{V} \subset \Hol^\R_\rho (\T) $ is \emph{stable under Fourier truncations} if 
\begin{equation*}
    \mathcal T_N h:= \sum_{|n|\le  N}\widehat h_n e_n\in\mathscr{V}
    \qquad\text{for every }h\in\mathscr{V}\text{ and }N \geq 0. 
\end{equation*} 
\end{definition}

\begin{remark}
Evenness gives an important special case. The space
\begin{equation*}
    \Hol_{\rho}^\ev(\T)
    :=\left\{V\in\Hol_\rho^{\R}(\T):V(\theta)=V(-\theta)
    \text{ for every }\theta\in\T\right\}
\end{equation*}
is Fourier-separated and stable under Fourier truncations.
 We write
\begin{equation*}
    \mathcal U^\ev:=\mathcal U\cap\Hol_\rho^\ev(\T).
\end{equation*}
The same holds for odd potentials, for instance. 

More generally, one can show that a closed vector space $ \mathscr{V}$ is Fourier separated and stable under Fourier truncations if and only if there exist $ I \subset\N^*$ and  a sequence  $ (\xi_n)_{n \in I} \in (\S^1)^I $ such that 
\begin{equation}
\mathscr{V} = \left\{h\in\Hol_\rho^\R(\T):
\xi_n\widehat h_n\in\R\ \text{for every }n\in I \qand
\widehat h_n=0\ \text{for every }n\in\N^*\setminus I
\right\}.
\end{equation}
Dropping the stability under Fourier truncation allows other natural spaces of deformations such as for instance 
\begin{equation*}
\left\{h\in\Hol_\rho^\R(\T):
\widehat h_{2k-1}=\widehat h_{2k}
\text{ for every }k\geq1\right\},
\end{equation*}
which is Fourier separated but is not stable under Fourier truncations.

\end{remark}

In the main results,  deformations of potentials are considered inside a Fourier separated space $ \mathscr{V}$. The base potential will also be in $ \mathscr{V}$ when the latter is stable under Fourier truncation. 
Observe that the case of even potentials replaces the usual $ \Z_2 $ symmetries that are required in the proof of inverse problems such as \cite{Z2,HezariHamid} for Kac's question, \cite{GM,CD, DKW} for billiards or also \cite{SKL} for billiards with obstacles. Note that $\Z_2$-symmetry can be replaced, as in the work of Fierobe--Kaloshin--Sorrentino~\cite{fierobe2025billiard}, by dihedral symmetries. Our definition of being Fourier separated contains them as well as more general families of symmetries.

Now recall the usual notion of jet in this setting. Given $d \ge 1 $, for a  function $f : \AA \to \C $ that is $C^\infty$--Whitney smooth at a point  $ \omega_0 \in \AA$, we write
\begin{equation*}
 J_{\omega_0}^{d-1}f
 :=\left(f(\omega_0),\partial_\omega f(\omega_0),\ldots,
 \partial_\omega^{d-1}f(\omega_0)\right) \in\C^d.
\end{equation*}

Finally, recall the following definition. 

\begin{definition}[Proper real analytic subset]
A  subset $ Z $ of a Fréchet space $X$ is said to be a  \emph{proper real analytic subset} if it is the zero set 
\begin{equation*}
    \left\{x\in\mathcal O \colon f(x)=0 \right\} ,
\end{equation*}
of a real analytic nonzero map $f:\mathcal{O} \to \R$ defined on a connected open subset $ \mathcal{O}$ of $X$. 
\end{definition}
In Theorem~\ref{thm:shyness} we show that  proper real analytic subsets are shy. 

The main result of the present work is the following.

\begin{theo}[Main]
\label{thm:main-rig}
Let $\omega_0\in\DC$ be an algebraic rotation number. Let $\mathscr{V} \subset \Hol_\rho^\R(\T)$ be a Fourier-separated linear
subspace.
 For  a linear subspace
$\V \subset\mathscr{V}$ of dimension $d \ge 1 $, there exists a proper real-analytic subset
$\mathcal Z_{\V} \subset \mathcal{U}^\R$ such that, for every $ V \in \mathcal{U}^\R \setminus \mathcal Z_{\V}$, the locally defined jet map
\begin{equation*}
\begin{array}{c@{\quad}c@{\quad}c}
\V\cap\left(-V+\mathcal U^\R\right)
&\longrightarrow&
\R^d
\\[4pt]
W
&\longmapsto&
J_{\omega_0}^{d-1}\left(\beta_{V+W}\right).
\end{array}
\end{equation*}
is real analytic and a local diffeomorphism at $0$. 

\smallskip\noindent
In particular, $\mathcal Z_{\V}$ is shy in $\Hol_\rho^\R(\T)$ and relatively closed in $\mathcal U^\R$.

\smallskip\noindent
 If, in addition, $\mathscr{V}$ is
stable under Fourier truncations, then $\mathcal Z_\V\cap\mathscr{V}$ is shy
in $\mathscr{V}$. 
\end{theo}
It is known that, for every $\tau>0$, one may choose $0<\gamma<\gamma^\ast(\tau)$ so that the Diophantine set $\DC$ contains infinitely many algebraic numbers, for instance quadratic irrationals. See e.g.~\cite[Theorem~1.1]{mcmullen2009uniformly}.

As mentioned before, this applies when $ \mathscr{V}$ is, for instance, the space of even potentials or the space of odd potentials.

Let us provide three examples of applications of the latter theorem. The first ones are trivial applications and need no proof.

\begin{coro}
\label{cor:nonlocal}
Under the hypotheses of Theorem~\ref{thm:main-rig}, let
$V_0 \in \mathcal{U}^\R \setminus \mathcal Z_\V$. Then level sets of 
$$  V \in (  V_0 + \V ) \cap \mathcal{U}^\R \mapsto J^{d-1}_{\omega_0} \beta_V  $$ 
have zero Lebesgue measure. 
\end{coro} 
While being trivial by real analyticity, it allows us to observe that the genericity of finite-dimensional rigidity is not a localized property. 

Then the following prohibits, in some sense, the existence of flexible deformations. 

\begin{coro}
\label{cor:defor-1-d}
Under the hypotheses of Theorem~\ref{thm:main-rig}, let
$V_0\in\mathcal U^\R\setminus\mathcal Z_\V$, let $I\subset\R$ be a connected
open interval containing $0$, and let
\begin{equation*}
    t\in I\longmapsto V_t\in(V_0+\V)\cap\mathcal U^\R
\end{equation*}
be a real-analytic deformation of $V_0$. If
$J_{\omega_0}^{d-1}(\beta_{V_t})$ is independent of $t$, then
$V_t\equiv V_0$ on $I$.
\end{coro}

The next consequence is not confined to one prescribed space of deformations.

\begin{coro}
\label{cor:def-rig-gen}
Fix an algebraic $\omega_0\in\DC $. Let $\mathscr{V}$ be Fourier separated and stable under Fourier truncation  and fix   a countable collection  $ (\V_n)_{n \in \N^*} $ of finite-dimensional subspaces of $\mathscr{V}$. 
There exists a subset $\mathcal R\subset\mathscr V\cap\mathcal U^\R$, residual and prevalent in $\mathscr V\cap\mathcal U^\R$, with the following property.
 Let $V_0\in\mathcal R$, let $I\subset\R$ be a connected open
interval containing $0$, and let
\begin{equation*}
    t\in I\longmapsto V_t\in \mathscr{V}  \cap \mathcal{U}^\R
\end{equation*}
be a real-analytic deformation of $V_0$. Assume that, for every
$k\geq1$, there exists $n_k \in \N^* $ such that 
\begin{equation*}
    \partial_t^k V_t |_{t=0} \in \V_{n_k} \;. 
\end{equation*}
 If $$ J^\infty_{\omega_0} \beta_{V_t}= J^\infty_{\omega_0} \beta_{V_0}$$
  for every $t\in I$, then $V_t\equiv V_0$ on $I$.
\end{coro}
Let us give an example of the above corollary.

\begin{example}
\label{ex:def-rig}
We restrict to perturbations with finitely many Fourier coefficients and even potentials. This is an arbitrary choice to make it more explicit. 

Fix an algebraic $\omega_0\in\DC$. There exists a  residual and prevalent subset $\mathcal R\subset\mathcal U^\ev$  with the following property. Let $V_0\in\mathcal R$, let $I\subset\R$ be a connected open
interval containing $0$, and let
\begin{equation*}
    t\in I\longmapsto V_t\in\mathcal U^\ev
\end{equation*}
be a real-analytic deformation of $V_0$. Assume that, for every
$k\geq1$, the derivative 
\begin{equation*}
    \left.\partial_t^kV_t\right|_{t=0}
\end{equation*}
has finitely many nonzero Fourier coefficients. If $ J^\infty_{\omega_0} \beta_{V_t}= J^\infty_{\omega_0} \beta_{V_0}$ for every $t\in I$,
then $V_t\equiv V_0$ on $I$.
\end{example}

If the base potential and the deformation directions are allowed to vary, the Fourier-separation condition can be dropped. Using the notion of prevalence introduced before
Theorem~\ref{thm:prevalent-beta-rigidity}, set
\begin{equation*}
    \mathscr P_d
    :=\mathcal U^\R\times\left(\Hol_\rho^{\R}(\T)\right)^d.
\end{equation*}

\begin{theo}
\label{thm:prevalent-affine}
Fix $d\ge1$ and an algebraic $\omega_0\in\DC$. There is an
open and  prevalent subset $\mathscr R_d\subset\mathscr P_d$ such
that for every
\begin{equation*}
    (V_0,h^{(1)},\ldots,h^{(d)})\in\mathscr R_d,
\end{equation*}
the locally defined map
\begin{equation*}
    \begin{aligned}
    (a_1,\ldots,a_d) \in \R^d &\longmapsto
    J_{\omega_0}^{d-1}\!\left(\beta_{V_a}\right),
    \qquad \text{ where  }  V_a:=V_0+\sum_{j=1}^d a_jh^{(j)},
    \end{aligned}
\end{equation*}
is a local diffeomorphism at $ 0$. 
\end{theo}

\begin{remark}
Actually, for a fixed tuple $(V_0,h^{(1)},\ldots,h^{(d)})\in\mathscr R_d$, the domain of this map is the open neighbourhood of $0$ given by
\begin{equation*}
\left\{a\in\R^d:V_0+\sum_{j=1}^d a_jh^{(j)}\in\mathcal U^\R\right\}.
\end{equation*}
\end{remark}

Note that this is a stronger version  of Theorem~\ref{thm:prevalent-beta-rigidity}.

The above results all use finitely many prescribed deformation directions and the jet of the $\beta$-function.
If instead all sufficiently small perturbations are available, then the sole
value of $\beta$ at a single Diophantine frequency already determines the
potential. 
No Fourier-separation condition or algebraicity assumption is needed.

\begin{theo}
\label{thm:one-frequency-rigidity}
Let $\mathcal B$ be an open, convex, and balanced neighbourhood of $0$ in
$\Hol_\rho^\R(\T)$ such that
\begin{equation*}
\mathcal B+\mathcal B\subset\mathcal U^\R.
\end{equation*}
Let $\omega_0\in\DC$ and let $V_1,V_2\in\mathcal B$. Assume that there
exists a neighbourhood $\mathcal W$ of $0$ in $\Hol_\rho^\R(\T)$ such that
\begin{equation*}
\beta_{V_1+W}(\omega_0)=\beta_{V_2+W}(\omega_0)
\qquad\text{for every }W\in\mathcal W.
\end{equation*}
Then $V_1=V_2$.
\end{theo}

Note that such a neighbourhood  $\mathcal{B}$ exists  as we will prove that $\mathcal U^\R$ is a locally convex
open neighbourhood of $0$ in
$\Hol_\rho^\R(\T)$.

It happens that all these results are genuinely nonlinear near the zero potential. Indeed, we show within the proof (see Corollary~\ref{cor:D1-at-zero}) that the Fréchet derivative of the $\beta$-function at the zero potential vanishes. 

\subsection{Open problems and perspectives}
 We now discuss structural limitations of the results and possible outcomes.  

\paragraph{Flexibility and synchronicity.}
The present results leave open the complementary
question of \emph{flexibility}.
The sole flexibility result for generalized standard maps was recently
obtained by Y.~Li~\cite{li2025deformations}. He constructs nontrivial
smooth deformations that preserve, respectively, the symplectic actions
or the Lyapunov exponents of infinitely many periodic orbits accumulating
on an invariant curve with Liouville rotation number. It appears within his proof that the accumulation curve cannot correspond to a Diophantine rotation number.

Note that there exist trivial symmetries that leave $ \beta$ unchanged. Indeed, denote $(T_aV)(\theta):=V(\theta+a)$ for fixed $a\in\T$. Then translating every configuration gives $\beta_{T_aV}=\beta_V$, since 
\begin{equation*}
    \mathcal L_{T_aV}(\theta_0,\theta_1)
    =\mathcal L_V(\theta_0+a,\theta_1+a).
\end{equation*}
However, a natural question is the following. 

\begin{question}
Modulo the action $V\mapsto T_aV$, do there exist nonconstant real analytic
deformations in the KAM domain that preserve the full Mather $\beta$-function or at least its infinite jet at a single algebraic Diophantine number ? 
\end{question} 
Furthermore, the present results and the involved techniques of the proofs suggest that injectivity of the jet map on every large enough finite dimensional subspace of a Fourier separated space stable under Fourier truncation is independent of the choice of the algebraic rotation number. 
\begin{question}[Synchronicity of rigidity]
\label{quest:two-frequencies}
Let $\omega_1,\omega_2\in\DC$ be algebraic rotation numbers.
Let $\mathscr V\subset\Hol_\rho^\R(\T)$ be an infinite dimensional Fourier separated space which is stable under Fourier truncations. Set
\begin{equation*}
\V_N:=\mathcal T_N\mathscr V,
\qquad
d_N:=\dim\V_N.
\end{equation*}
Is it true that, for every $V\in\mathcal U^\R\cap\mathscr V$, there exists
$N_0\geq1$ such that, for every $N\geq N_0$, the two locally defined maps
\begin{equation*}
W\in\V_N
\longmapsto
J_{\omega_i}^{d_N-1}\left(\beta_{V+W}\right)\in\R^{d_N},
\qquad i=1,2,
\end{equation*}
are either both local diffeomorphisms at $0$ or both fail to be local
diffeomorphisms there?
\end{question}

Note that, by Corollary~\ref{cor:def-rig-gen}, this is true for a residual set of potentials with $N_0 =1$, and any algebraic $\omega_1,\omega_2$ (as there are countably many of them). A positive answer to that question would present some property of synchronisation for algebraic Diophantine rotation numbers of local finite dimensional invertibility of the jet of $\beta$-function.

\paragraph{Billiard maps.}
The billiard map of a strictly convex domain $\Omega\subset\R^2$ is another
natural exact twist map, and its $\beta$-function determines the maximal marked length spectrum, see e.g.~\cite{Sor14,Bialy2022}.
Very recently, Fierobe, Kaloshin and
Trujillo~\cite{FKT} prove that if the $\beta$-functions of two analytic billiard-like maps coincide on a subset of positive measure of a fixed KAM Diophantine set near zero, then they coincide on that whole KAM Diophantine set.
 For billiard domains, this implies equality of the corresponding Marvizi--Melrose invariants.
While they do not study the inverse problem, they bring the first step toward establishing a holomorphic theorem in the flavour of the present work. 
Their proof starts from Lazutkin's complex extension of the set of rotation numbers, which is very similar to the extension considered here. Then, on this set but at a fixed domain, they provide a
$C^\infty$--Whitney-holomorphic extension of the parametrization of the KAM curves. 
A natural next step would be to prove holomorphic dependence on the parametrization of the boundary. 
In view of the present work, it is natural to renew the question raised
in~\cite{FKT}: whether, and in which setting, the jet of the
$\beta$-function of a billiard map at a single rotation number determines its domain. 
They actually deal with the more general billiard-like-maps including systems such as symplectic billiards or outer billiards. Let us mention that  Baracco, Bernardi and Nardi in \cite{Nardi} computed the first derivatives in the rotation number of $ \beta$ at zero for  these two systems. The similar inverse problem for symplectic or outer billiards could be legitimately set. 

Also one may ask whether for such systems Laplace rigidity could be proved using $ \beta$-function.  

\paragraph{Mechanical systems.}
Techniques developed here could apply to flows, and in particular to the class of \emph{mechanical systems} on the torus.
For instance, consider flows generated by a Hamiltonian of the form
\begin{equation}
H ( \theta , p ) = \frac{1}{2}  | p|^2   - V (\theta) 
\end{equation}
where $ V : \T \to \R$ is once again a potential. 
 For such systems, the $\beta$-function is the minimal average action of an orbit of the flow among orbits with corresponding rotation number. 
 Extending Theorem~\ref{thm:hol-KAM} to this continuous-time setting would provide a common ground and tools to address rigidity questions of such systems.

\paragraph{Geodesic flows} 
Beyond the natural continuous-time rigidity questions based either on
minimal action as a function of the rotation vector or on geodesic
lengths indexed by free homotopy classes, geodesic flows on surfaces
also admit a natural discrete-time description. Near a nondegenerate
elliptic periodic orbit in the unit tangent bundle, one may consider a
transverse Poincaré section. The corresponding local first return map is exact
symplectic and, generically, satisfies the twist condition.
Close to the flat metric, for instance, one may then ask to what extent the $\beta$-function locally determines the metric.

\subsection*{Acknowledgements}
This work is partially supported by ERC SPERIG $\#885707$. 
I am deeply grateful to Vadim Kaloshin and Corentin Fierobe for many
thoughtful and enlightening discussions.

\subsection{Organization of the paper}

The paper has two main parts.

\smallskip
 We first provide  the setting in which the  KAM theorem will be proven.  Then rigidity arguments are given by studying the local behaviour of the second Fréchet derivative in the potential of the jet in the rotation number of the $\beta$-function near the zero potential. Holomorphy provided by the KAM theorem allows us to pass from local to nonperturbative analysis.
 
Before entering these two parts, Section~\ref{sec:preliminaries} presents all the material used throughout the paper. We first recall the variational setting for generalized standard maps and the precise definition of Mather's $\beta$-function. Then we introduce the Fréchet and $C^\infty$--Whitney-holomorphic spaces needed to study
the dependence on the potential and on the complexified rotation number.
This goes by providing the quasiconvexity property of $\KK$ which allows us to control the full Whitney topology through explicit derivative estimates that are easier to handle than the defining Whitney family of seminorms. 

\smallskip
Section~\ref{sec:KAM} consists of stating rigorously the KAM theorem and providing its proof. It happens to be necessary to define some adapted spaces for solving the Euler--Lagrange equation. 
 After defining these and studying the
discrete operators involved, in Section~\ref{subsec:cohom}, we solve the Levi--Moser cohomological equation, presented in Section~\ref{subsec:levy-moser}, with explicit tame estimates. 
Then, following~\cite{CMS}, the iteration is first run
at order zero on $ \KK$ in Section~\ref{subsec:zero-scheme}.
In Section~\ref{subsec:convergence-every-order}, an induction on the order of the derivatives then proves convergence in the full $C^\infty$--Whitney topology without shrinking the KAM 
neighbourhood. This gives Theorem~\ref{thm:hol-KAM}, including
the tame holomorphic dependence on the potential, and then the corresponding extension of Mather's $\beta$-function in
Theorem~\ref{thm:kam-main-beta}. The latter is proven in Section~\ref{subsec:beta-holo}. 

Section~\ref{sec:perturbative} turns this analytic construction into
rigidity results. We first compute in Section~\ref{subsec:frechet-deriv} the first variations of the extended
$\beta$-function and obtain, in Section~\ref{subsec:linearization}, an explicit Fourier representation of its second derivative at the zero potential. This representation leads to a
Wronskian criterion for the local invertibility of the jet map that is studied in Section~\ref{subsec:wronskian}. 
 We verify  this criterion at explicit arbitrarily small potentials, and holomorphy then shows that its failure is confined to a proper real analytic subset of the KAM domain.
Finally in Section~\ref{subsec:wronsk-to-rig}, we deduce the proofs of Theorems~\ref{thm:main-rig}, \ref{thm:prevalent-affine} and \ref{thm:one-frequency-rigidity}. 

\section{Setting and topology}
\label{sec:preliminaries}
This preliminary section provides proper definitions involved in this article together with auxiliary results needed for the proofs of the KAM and rigidity theorems. 

\subsection{Some elements of Aubry--Mather theory for generalized standard maps} \label{subsec:aubry-mather}
We recall in greater detail the notion of twist maps, their Lagrangian formulation and link to Aubry--Mather theory.

\subsubsection{Rotation number and the \texorpdfstring{$\beta$}{beta}-function} \label{subsec:beta}

Let $\A := \T \times \R$ denote the cylinder.
A \emph{twist map} of $\A$ is an exact symplectic diffeomorphism
\begin{equation*}
  F : (\theta_0, r_0) \mapsto (\theta_1, r_1)
\end{equation*}
preserving $d\theta \wedge dr$ and admitting a generating function
$\mathcal{L} : \R \times \R \to \R$ of class $C^k$, for $k \ge 2$, satisfying the
periodicity condition
$\mathcal{L}(\theta_0 + 1, \theta_1 + 1) = \mathcal{L}(\theta_0, \theta_1)$
together with
\begin{equation*}
  r_0 = -\partial_{\theta_0} \mathcal{L}(\theta_0,\theta_1),
  \qquad
  r_1 = \partial_{\theta_1} \mathcal{L}(\theta_0,\theta_1),
\end{equation*}
and the twist condition
$\partial_{\theta_0}\partial_{\theta_1}\mathcal{L}(\theta_0,\theta_1)
\le -\varepsilon  $ for some $ \varepsilon > 0$. 

Orbits of $F$ are in bijection with configurations
$(\theta_n)_{n \in \Z} \in \R^{\Z}$ that are critical points of the formal
action functional
\begin{equation*}
  \mathcal{A}\left((\theta_n)\right)
  := \sum_{n \in \Z} \mathcal{L}(\theta_n, \theta_{n+1}),
\end{equation*}
two configurations being identified when they differ by a simultaneous
integer translation $\theta_n \mapsto \theta_n + k$, with the same
$k\in\Z$ for every $n$.
We focus throughout on the family of \emph{generalized standard maps},
for which the generating function is
\begin{equation*}
  \mathcal{L}_V (\theta_0, \theta_1)
  = \tfrac{1}{2}(\theta_1-\theta_0)^2 + V(\theta_0),
  \qquad  \text{ for } V \in C^2(\T).
\end{equation*}
That is, the corresponding map $F_V$ is defined as
\begin{equation}
\begin{array}{rcl}
 F_V\; \;   \colon   \; \;   \; \;   \mathbb{A}  \; \;  &\longrightarrow&  \; \;  \mathbb{A} \\
(\theta_0,r_0) &\longmapsto& (\theta_1,r_1),
\end{array}
\qquad \text{with }
\begin{dcases}
\theta_1 = \theta_0 + r_0 + V'(\theta_0)\\
r_1 = r_0  +  V'(\theta_0) \; .
\end{dcases}
\end{equation}
We now proceed to define the rotation number of an invariant measure for a general twist map $F$ of the above form. Fix a lift $\widehat F  \colon \R^2 \to \R^2$ of $F$  commuting with
the translation $(x, r) \mapsto (x+1, r)$.
Write $ \widehat{F} =: (X, P) $ for its coordinates. 
The \emph{displacement} $f := X - \mathrm{id}$ is $\Z$-periodic in $x$ and thus defines a smooth
function on the cylinder $\A$. 

For an $F$-invariant probability measure $\nu$ on $\A$ such that $f \in L^1(\nu)$, the \emph{rotation number} of $\nu$ is defined as the \emph{average angular displacement per iterate}:
\begin{equation*}
  \rho(\nu) := \int_{\A} f \, d\nu \; \in \R \; .
\end{equation*} 

We then  transport invariant measures to the space $ \mathcal{C} := \left.(\R \times \R)\right/ \Z$,  where the quotient is taken under the action $ k \cdot (\theta, \theta') := (\theta + k, \theta' + k) $ for $ k \in\Z $ and $ \theta, \theta' \in \R$. 
Both $\theta'-\theta$ and $\mathcal L(\theta,\theta')$ are invariant
under this action and therefore define functions on
$\mathcal C$.
 The twist condition implies that
\begin{equation*}
\begin{aligned}
    T\colon\mathbb A&\longrightarrow\mathcal C,\\
    (\theta,r)&\longmapsto[\theta,X(\theta,r)]
\end{aligned}
\end{equation*}
is a well-defined smooth diffeomorphism. For generalized standard maps, note that we have the explicit formula 
\begin{equation*}
    X(\theta,r)=\theta+r+V'(\theta).
\end{equation*}
Introduce the conjugacy of $F$ by $ T$ as 
\begin{equation*}
    \widetilde F:=T\circ F\circ T^{-1}
    \colon\mathcal C\longrightarrow\mathcal C.
\end{equation*}
If $(\theta_0,\theta_1) \in \mathcal{C} $, then  
\begin{equation*}
    \widetilde F ( \theta_0,\theta_1) = (\theta_1,\theta_2) ,
\end{equation*}
where $\theta_2$ is uniquely determined by the discrete Euler--Lagrange
equation
\begin{equation*}
    \partial_{\theta_1}\mathcal L(\theta_0,\theta_1)
    +\partial_{\theta_0}\mathcal L(\theta_1,\theta_2)=0.
\end{equation*}
Thus $\widetilde F$ simply shifts two consecutive entries of an orbit
configuration.

In particular, for $\nu$ the $F$-invariant probability measure on $\mathbb{A} $ such
that the displacement $ f$ is integrable, we consider its pushforward 
\begin{equation*}
    \mu:=T_\ast\nu.
\end{equation*}
This is a  $\widetilde F$-invariant probability measure on
$\mathcal C$. 
Consequently, the rotation number of $ \nu$ can equivalently be written as that of $ \mu$:
\begin{equation*}
    \rho(\mu)
    :=\int_{\mathcal C}(\theta'-\theta)\,d\mu(\theta,\theta')
    =\rho(\nu).
\end{equation*}

Mather's $\beta$-function is therefore defined by
\begin{equation*}
    \beta(\omega)
    :=
    \min_{\substack{
        \mu\ \widetilde F\text{-invariant probability measure}\\
        \rho(\mu)=\omega}}
    \int_{\mathcal C}\mathcal L\,d\mu,
    \qquad \omega\in\mathbb R.
\end{equation*}
The Aubry--Mather theorem recalled below ensures that this minimum is
attained.
The  function $\beta$ is convex. 
For generalized standard maps, the function
$\omega \mapsto \beta(\omega) - \tfrac{\omega^2}{2}$
is $1$-periodic and even.   

For each $\omega \in \R$,
the \emph{Mather set} of rotation number $\omega$ is the closed invariant set
\begin{equation*}
  \mathcal{M}_\omega
  := \overline{\bigcup \left\lbrace T^{-1}\left(\operatorname{supp} \mu\right)
     \;:\; \begin{array}{l}
     \mu\text{ is a $\widetilde F$-invariant probability measure},\\
     \rho(\mu)=\omega\text{ and }\displaystyle\int_{\mathcal C}\mathcal L\,d\mu=\beta(\omega)
     \end{array}
     \right\rbrace} \; \subset \A \; .
\end{equation*}
The following fundamental theorem provides the starting point of the study of the geometry of minimal orbits. See, for instance~\cite{Bangert} for more details. 
\begin{theorem}[Birkhoff]
  Every Mather set is contained in a Lipschitz graph over its projection
  to the angle coordinate $\theta\in\T$.
\end{theorem}

Furthermore, the following theorem, due independently to Aubry and Mather,
guarantees that minimizing objects exist for every rotation number.

\begin{theorem}[Aubry--Mather {\cite{Aubry,Mather82}}]
  For every $\omega \in \R$, the twist map $F$ possesses an invariant set
  $\mathcal{M}_\omega \subset \A$ of rotation number $\omega$ consisting
  entirely of action-minimizing orbits.
  If $\omega \in \Q$, then $\mathcal{M}_\omega$ contains an
  action-minimizing periodic orbit.
  If $\omega \notin \Q$, then $\mathcal{M}_\omega$ is either an invariant
  circle or a Cantor set.
\end{theorem}
This guarantees that $ \beta$ is actually well defined.

As is usual in KAM theory, rotation numbers satisfying classical Diophantine
conditions enjoy better regularity than that provided by the above theorems.
This is the subject of the next paragraph.

\subsubsection{KAM curves and regularity of their parametrization}

A (smooth) \emph{KAM curve} of rotation number $\omega \in \R \setminus \Q$
is an $F$-invariant  embedded circle
$\Gamma \subset \A$ on which the dynamics is smoothly conjugate to the
rotation by $\omega$. Precisely, there exists a $C^\infty$ embedding
$\gamma \colon \T \to \A$ with $\Gamma = \gamma(\T)$ such that
\begin{equation*}
  F \circ \gamma = \gamma \circ \tau_\omega \; ,
  \qquad
  \tau_\omega \colon \T \to \T, \quad \theta \mapsto \theta + \omega \bmod 1 \; .
\end{equation*}
Equivalently, there exists an orientation-preserving circle diffeomorphism
$h \colon \T \to \T$, with lift $\widetilde h$, solving the discrete
Euler--Lagrange equation
\begin{equation}\label{eq:EL}
  \partial_{\theta_0} \mathcal{L}\left(\widetilde h(\theta),\,
  \widetilde h(\theta+\omega)\right)
  + \partial_{\theta_1} \mathcal{L}\left(\widetilde h(\theta-\omega),\,
  \widetilde h(\theta)\right)
  = 0 \; ,
  \qquad \forall\, \theta \in \R \; . 
\end{equation}
The corresponding curve is the  graph of the map 
\begin{equation*}
  \gamma(\theta)
  = \left( h(\theta),\,
    -\partial_{\theta_0} \mathcal{L}\left(\widetilde h(\theta),\,
    \widetilde h(\theta+\omega)\right) \right) \; .
\end{equation*}
We know from Aubry and Mather's theory~\cite{Aubry,Mather82} that the configuration
$\left(\widetilde h(\theta + n\omega)\right)_{n \in \Z}$ is a global minimizer of the action among orbits with the same prescribed rotation number. Moreover, $\Gamma$ is a smooth graph
over the angle coordinate $\theta \in \T$.
Writing
\begin{equation*}
  \widetilde h(\theta) = \theta + u(\theta)
\end{equation*}
with $u \in C^\infty(\T)$, we normalize $u$ by requiring
$\int_\T u = 0$. This normalization fixes the freedom of precomposition by rotations
and determines the homeomorphism $h$ uniquely.

\smallskip
The dynamics on a KAM curve of irrational rotation number is conjugate to
an irrational rotation and thus uniquely ergodic. The pushforward of Lebesgue measure under the parametrization $\gamma$ is consequently a minimizing invariant measure. For general background on action-minimizing sets of twist maps, see Bangert~\cite{Bangert}. 

Consequently, the $\beta$-function takes the form 
\begin{equation}\label{eq:beta-general}
  \beta(\omega)
  =  \int_\T  \mathcal{L} \left(\widetilde h(\theta),\,\widetilde h(\theta+\omega)\right) \,d\theta \;. 
\end{equation} 
For the generalized standard map with generating function
$\mathcal{L}_V(\theta_0,\theta_1)
=\tfrac12(\theta_1-\theta_0)^2+V(\theta_0)$,
equation~\eqref{eq:beta-general} becomes
\begin{equation*}
  \beta(\omega)
  = \frac{\omega^2}{2} +   \int_\T \left[
      \tfrac{1}{2}\left(u(\theta+\omega)-u(\theta)\right)^2
      + V\left(\theta + u(\theta)\right)
    \right]\,d\theta,
\end{equation*}
and the discrete Euler--Lagrange equation~\eqref{eq:EL} reduces to
\begin{equation*}
  u(\theta+\omega) - 2u(\theta) + u(\theta-\omega)
  = V'\left(\theta + u(\theta)\right),
  \qquad \forall\,\theta \in \T.
\end{equation*}
We write $\id_\T$ for the identity map on $\T$ and, by an abuse of
notation, also for the constant map $a\in A\mapsto\id_\T$, for any
parameter space $A$. For a function $f(x,\theta)$ with $\theta\in\T$
and $x$ in some other space, we write
\begin{equation*}
  \int_\T f := x\longmapsto\int_\T f(x,\theta)\,d\theta
\end{equation*}
whenever this makes sense. 

We now turn to the regularity of the map
$(V ,\omega)\mapsto u$ among such parametrizations.

Although KAM curves may exist only for a Cantor set of Diophantine
rotation numbers, their parametrizations vary with $\omega$ with the
strongest regularity available in this setting. 
Recall that $\omega  \in  \R$ is \emph{$(\gamma,\tau)$-Diophantine} if
\begin{equation*}
  \left|\omega - \frac{p}{q}\right| \ge \frac{\gamma}{q^{2+\tau}}
  \qquad \forall\, p \in \Z,\ q \in \N^*,
\end{equation*}
for some constants $\gamma > 0$ and $\tau >  0 $.
This set is invariant under integer translation.
For a fixed $\tau$, the measure of $\DC$ modulo $\Z$ tends to one as $\gamma $ goes to $ 0$.
A classical, yet seminal, result due to Lazutkin~\cite{lazutkin1973existence}, and later,
independently, to Pöschel~\cite{Poschel} states that, near the integrable case in a fixed twist region, the parametrizations of KAM curves within $\DC$ with fixed $ \gamma, \tau$ depend $C^\infty$-smoothly on the Diophantine rotation
number in the Whitney sense.

On the other hand, Carminati--Marmi--Sauzin
\cite{CMS} fix an analytic potential $V$ and treat the scalar family
$\epsilon V$  with a free small parameter $\epsilon$. 
They show $C^1$-holomorphy in the rotation number on $\AA$ (more precisely, in $q := \e (\omega) \in \KK$) and usual holomorphy in the parameter $ \epsilon$.  
Theorem~\ref{thm:hol-KAM} provides a stronger  statement. 
The parametrization belongs to $\Hol^\infty(\KK , \Hol_{\rho_\infty} (\T) )$ for $ \rho_\infty < \rho$, hence is holomorphic in the interior of $\KK$ with a $C^\infty$-Whitney jet on its boundary and moreover the dependence on the potential is holomorphic near the zero potential.

\subsection{Topologies and Whitney
\texorpdfstring{$C^\infty$}{C-infinity}-holomorphicity}
\label{lieu:top}
While the potential is considered under usual holomorphy, the rotation cannot be. This  leads us to consider the notion of $C^\infty$-Whitney holomorphy. For the sake of clarity, we treat them separately. We will then observe that notations are consistent.

\subsubsection{Holomorphy on open Fréchet domains}

First, we quickly recall that a Fréchet space $A$ over $\K\in\{\R,\C\}$ is a complete
Hausdorff locally convex space whose topology is generated by a countable
family of seminorms. We can always assume that such a family $( p_m)_{m \ge 0} $ is taken to be increasing. 
Then for every  $m\geq0$, let $A_m$ be the Banach completion of
$A/\ker p_m$ for the norm induced by $p_m$.  Then
$A \simeq \varprojlim_{m\geq0} A_m $ is the projective limit of the Banach spaces  $ A_m$. 
In particular, a sequence $(a_n)$ converges to $a$ in $A$ if and only if
$p_m(a_n-a) $ converges to $0$ for every $m\geq0$.

For $\Hol_\rho(\T)$, we choose instead to work with the uncountable family $ (|\cdot |_{\rho'})_{ 0 < \rho' < \rho} $. Note that this is equivalent to taking 
$p_m (V) :=|V|_{\rho_m}$ where $  (\rho_m)_{m \in\N}$ is a positive sequence increasing to $\rho$. 

We now present the notion of holomorphy on open subspaces of  Fréchet spaces. 
The present notion goes back to Bastiani~\cite{Bastiani1964}. See 
 Schmeding~\cite{Schmeding2023} for a modern presentation of the topic. 
 
 A map  $F:\Omega\to B$ is  said to be holomorphic if its  directional derivatives
\begin{equation*}
 D^kF(x)[h_1,\ldots,h_k]
 :=\left.
 \partial_{t_1}\cdots\partial_{t_k}
 F(x+t_1h_1+\cdots+t_kh_k)
 \right|_{t_1=\cdots=t_k=0}
\end{equation*}
exist,
are complex multilinear and depend continuously
on the point $x \in \Omega$ and directions $h_1,\ldots,h_k \in A $,  for every $k\geq0$.
It is well known, for instance, that this definition is equivalent to $F$ being continuous and to its restriction to every complex affine line crossing $\Omega$ being holomorphic. Such equivalent descriptions of holomorphy on
Fréchet spaces are standard.  See \cite{Mujica1986,Dineen1999}.
We then denote by 
$\Hol^\infty(\Omega,B)$ the space of holomorphic maps from $ \Omega \subset A$ to $B$. This space is endowed with the topology generated by the seminorms
\begin{equation*}
    \|  F  \| _{K,m,k}
    :=
    \sup_{(x,h_1,\ldots,h_k)\in K}
  p_m^B \left(D^kF(x)[h_1,\ldots,h_k]\right),
\end{equation*}
where $m,k\geq0$ and  $K$ ranges over the compact subsets of $\Omega\times A^k$.

Following Hamilton~\cite[Part~II.1]{Hamilton1982}, a continuous
map $F:\Omega\subset  A \to  B $ between graded Fréchet spaces is called
\emph{tame} if every $x_0\in\Omega$ has a neighbourhood
$\mathcal O\subset\Omega$ for which there exist integers $m_0 , r\geq0$ such
that, for every $m \geq m_0 $, there exists a constant $C_m >0$ satisfying
\begin{equation*}
 p_m^B( F(x) )   \le C_m  \left( 1+ p_{m +r} ^A (x) \right) 
 \qquad\text{for every }x\in\mathcal O.
\end{equation*}
The essential point is  that the loss of derivatives represented by the shift $r$  is independent of
$m$.

For real holomorphy we choose the following usual definition throughout the paper. If $A$ is a real Fréchet space, a map on an open subset of $A$  is called \emph{real analytic} when it is locally the restriction of a holomorphic map on the complexification of $A$.

\subsubsection{Whitney jets on compact and perfect subsets}

The definition of holomorphy we adopt is the complex version of Whitney's definition \cite{whitney1934analytic} in the general setting where the target space is a Fréchet space and the domain is a compact and perfect subset of $ \C$. 

Let $K\subset\C$ be compact and perfect, i.e. without isolated points. Let $A$ be a complex Fréchet space with corresponding seminorms $(p_s)_{s\geq0}$. 
The same definition applies to a complete Hausdorff locally convex target by using all its continuous seminorms.

For a sequence of continuous maps $ \mathbf{F} = (F^{(k)})_{k \in \N}$ from $K$ to $A$ and integers $ k , \ell \ge 0$, we denote by
\begin{equation*}
R_{k,\ell}  \mathbf{F}  (x,y)
:=
F^{(k)}(y)
-
\sum_{j=0}^{\ell}
\frac{F^{(k+j)}(x)}{j!}(y-x)^j
\end{equation*}
the \emph{remainder}.

\begin{definition}[Whitney-holomorphic maps]
\label{def:holo}
A continuous map $F\colon K\to A$ is  said to be 
$C^\infty$--Whitney-holomorphic if there exists a sequence of continuous maps $ \mathbf{F} = (F^{(k)})_{k \in \N}$ from $K$ to $A$ with $F^{(0)}=F$ and such that, for every
$s,k,\ell\geq0$, the following holds:
\begin{equation}
\label{eq:whitney-remainder}
\lim_{\delta\to0^+}
\sup_{\substack{x,y\in K\\0<|x-y|\le \delta}}
\frac{p_s\left(R_{k,\ell} \mathbf{F} (x,y)\right)}
{|x-y|^\ell}
=0.
\end{equation}
We denote the resulting space by $\Hol^\infty(K,A)$.
\end{definition}

Since $K$ is perfect, the maps $F^{(k)}$ are uniquely determined by
$F$. 

For $m\geq0$, define the corresponding seminorms
\begin{equation}
\label{eq:whitney-seminorm}
\|F\|^{\mathrm W}_{K,s,m}
:=
\max_{0\le  k\le  m}\sup_{x\in K}p_s\left(F^{(k)}(x)\right)
+
\max_{\substack{k,\ell\geq0\\k+\ell\le  m}}
\sup_{\substack{x,y\in K\\x\neq y}}
\frac{p_s\left(R_{k,\ell}\mathbf{F}(x,y)\right)}
{|x-y|^\ell}.
\end{equation}
The seminorms \eqref{eq:whitney-seminorm} make
$\Hol^\infty(K,A)$ a Fréchet space. Indeed, a Cauchy sequence is  uniformly convergent if and only if elements of the sequence converge in every seminorm of $A$ and  
the normalized remainders converge uniformly as well. So
\eqref{eq:whitney-remainder} passes to the limit. Completeness of $A$ then gives completeness. 

For compact subsets of $\hat\C$, Definition~\ref{def:holo} is imposed in the two standard charts. The chartwise jets agree on the overlap by the
usual chain rule. This is the convention used for $K=\KK$ and will be investigated in greater depth in Subsection~\ref{subsec:charts}.

When $K$ is \emph{quasiconvex} it is possible to replace the family of seminorms in~\eqref{eq:whitney-seminorm} by the simpler family
\begin{equation}\label{eq:seminorms-simpler}
\|F\|_{K,s,m} := \max_{0\le  k\le  m} \sup_{x\in K}p_s\left(F^{(k)}(x) \right) 
\end{equation}
Recall that a subset $K\subset\C$ is called $c$-quasiconvex for some constant $ c>0$  if every pair of points
$x,y \in K$ can be joined in $K$ by a $1$-Lipschitz path 
\begin{equation*}
 \gamma:[0,L]\longrightarrow\C,
\end{equation*}
 satisfying
\[
L \le  c |x-y| \;.
\]
We will call such a path a rectifiable path between $x$ and $y$ and denote (abusively) its length $ \ell (\gamma) := L$. 
We will also say that $K$ is quasiconvex if it is $c$-quasiconvex for some $c$. 

\begin{lemma}
\label{lem:whitney-quasiconvex}
Let $K\subset\C$ be compact and $c$-quasiconvex. Let
$F \in \Hol^\infty(K,A)$. Then, for every
$k,\ell,s\geq0$ and $x,y\in K$,
\begin{equation}
\label{eq:whitney-path-bound}
 p_s\left(R_{k,\ell}\mathbf F(x,y)\right)
 \le  \frac{c^{\ell+1}}{(\ell+1)!}
 \sup_{z\in K}p_s(F^{(k+\ell+1)}(z))\,|x-y|^{\ell+1} 
\end{equation}
where $ \mathbf{F} = (F^{(k)})_{k \ge 0} $ denotes the derivatives of $F$. 
Consequently, the seminorms of \eqref{eq:seminorms-simpler}, defined only via the supremum of derivatives, and the full Whitney
seminorms \eqref{eq:whitney-seminorm} generate the same topology.
\end{lemma}

\begin{proof}
Choose a rectifiable path
$\gamma:[0, L ]\to K$ joining $x$ to $y$ in $K$. 
Since $\gamma$ is Lipschitz, it is differentiable almost everywhere and, for every $j\ge 0$, the first-order Whitney condition for almost every $t \in [0,L]$ implies 
\begin{equation*}
 \frac{d}{dt}F^{(j)}(\gamma(t))
 =F^{(j+1)}(\gamma(t)) \cdot \gamma'(t) \;. 
\end{equation*}
At $ \ell = 0$ a simple integration provides
\begin{equation*}
 R_{k,0}\mathbf F(x,y)
 =\int_0^{L}
 F^{(k+1)}(\gamma(t)) \cdot \gamma'(t)\,dt,
\end{equation*}
giving~\eqref{eq:whitney-path-bound} trivially in this case. For
$\ell\geq1$, a simple induction on $ \ell $ gives
\begin{equation*}
 R_{k,\ell}\mathbf F(x,y)
 =\frac{1}{(\ell-1)!}\int_0^{L}
 (y-\gamma(t))^{\ell-1}
 \left(F^{(k+\ell)}(\gamma(t))-F^{(k+\ell)}(x)\right)
 \gamma'(t)\,dt.
\end{equation*}
Hence, we have
\begin{equation*}
 p_s\left(F^{(k+\ell)}(\gamma(t))-F^{(k+\ell)}(x)\right)
 \le  t\sup_{z\in K}p_s(F^{(k+\ell+1)}(z)).
\end{equation*}
Since $\gamma$ is $1$-Lipschitz and
$\gamma(L)=y$,
\begin{equation*}
 |y-\gamma(t)|
 =|\gamma(\ell(\gamma))-\gamma(t)|
 \le  \ell(\gamma)-t.
\end{equation*}
It follows that
\begin{equation*}
 p_s\left(R_{k,\ell}\mathbf F(x,y)\right)
 \le  \frac{1}{(\ell-1)!}
 \int_0^{L}(\ell(\gamma)-t)^{\ell-1}t\,dt\,
 \sup_{z\in K}p_s(F^{(k+\ell+1)}(z))
 =\frac{L^{\ell+1}}{(\ell+1)!}
 \sup_{z\in K}p_s(F^{(k+\ell+1)}(z)).
\end{equation*}
Since $ L \le c |x - y | $, we obtain the desired inequality and conclude the proof. 
\end{proof}
\begin{remark}
In particular, when $K=\overline{K^\circ}$ is quasiconvex, a map is $C^\infty$--Whitney-holomorphic if it is holomorphic in the interior of $K$ and all its derivatives extend continuously to $K$. At interior points this notion is the usual holomorphy on an open set.
\end{remark} 
\subsubsection{Examples of continuous operations on Whitney-holomorphic maps}
We provide some examples of $C^\infty$--holomorphic maps that are used in the proofs.

\begin{proposition}
\label{fact:trace}
Let $X\subset K$ be compact and perfect. The restriction from $ \Hol^\infty(K,A) $ to $\Hol^\infty(X,A)$
\begin{equation*}
\begin{array}{rcl}
\Hol^\infty(K,A)
&\longrightarrow&
\Hol^\infty(X,A)
\\[4pt]
F
&\longmapsto&
F\vert_X
\end{array}
\end{equation*}
 is  continuous.
\end{proposition}

\begin{proof}
Whitney conditions, if they hold on $K$, also hold on $X$. 
\end{proof}
In other words, the trace topology is finer than the topology inherited from $\Hol^\infty(X,A)$.

\begin{lemma}
\label{lem:linear-postcomposition}
Let $L:A\to B$ be a continuous linear map between Fréchet spaces. If $\Omega$ is an open subset of a Fréchet space or a compact perfect subset of $\C$, then the map 
\begin{equation*}
 F \longmapsto L\circ F
\end{equation*}
is a continuous linear map from $\Hol^\infty(\Omega,A)$ to
$\Hol^\infty(\Omega,B)$. 
\end{lemma}

\begin{proof}
If $ \Omega $ is an open subset of a Fréchet space, this is trivial. In the compact and perfect case, note that applying $L$ to each Taylor remainder in~\eqref{eq:whitney-remainder} bounds every seminorm of $L\circ F$ by a finite number of seminorms on $F$.
\end{proof}

\begin{lemma}
\label{lem:fourier-extraction}
Let $\Omega$ be either an open Fréchet domain or a compact perfect subset
of $\C$. For every $k\in\Z$, Fourier extraction defines a  continuous
linear map
\begin{equation*}
 \mathcal F_k:\Hol^\infty\left(\Omega,\Hol_\rho(\T)\right)
 \longrightarrow\Hol^\infty(\Omega,\C),
 \qquad
 \mathcal F_k(f)(x):=\int_0^1f(x,\theta)e^{-2\pi i k\theta}\,d\theta.
\end{equation*}
We will denote $\widehat f_k := \mathcal F_k(f)$ when it improves readability. 
For every $0<\rho'<\rho$ and $V\in\Hol_\rho(\T)$,
\begin{equation*}
 |\widehat V_k|\le  e^{-2\pi\rho'|k|}|V|_{\rho'}.
\end{equation*}
The same applies without a mean constraint, including
for $k=0$.
\end{lemma}

\begin{proof}
The case $k=0$ is immediate.
For $k\neq0$, by  Cauchy's theorem and periodicity, the bound directly follows from: 
\begin{equation*}
\widehat V_k
=
e^{-2\pi\rho'|k|}
\int_0^1
V\left(\theta-i\operatorname{sgn}(k)\rho'\right)
e^{-2\pi i k\theta}\,d\theta.
\end{equation*}
Since this operation is linear, Lemma~\ref{lem:linear-postcomposition} therefore gives continuity and holomorphy of the corresponding map, while the displayed estimate applies in each defining seminorm. 
\end{proof}

\begin{lemma}
\label{lem:der-hol}
Let $0<\rho'<\rho$, and let $\Omega$ be an open  subset of a Fréchet space or a
compact perfect subset of $\C$. Then differentiation with respect to $ \theta \in \T$ is a continuous linear map
\begin{equation*}
 \partial_\theta:
 \Hol^\infty\left(\Omega,\Hol_\rho(\T)\right)
 \longrightarrow
 \Hol^\infty\left(\Omega,\Hol_{\rho'}(\T)\right).
\end{equation*}
More generally, for every $V\in\Hol_\rho(\T)$, every
$p\geq0$, and every $\rho''\in(0,\rho')$, we have the bound
\begin{equation}
\label{eq:der-theta}
|\partial_\theta^p V|_{\rho''}
\leq
\frac{p!}{(\rho'-\rho'')^p}|V|_{\rho'}.
\end{equation}
\end{lemma}

\begin{proof}
For $\theta\in\T_{\rho''}$, the higher-order Cauchy formula gives
\begin{equation*}
\partial_\theta^p V(\theta)
=
\frac{p!}{2\pi i}
\int_{|\zeta-\theta|=\rho'-\rho''}
\frac{V(\zeta)}{(\zeta-\theta)^{p+1}}\,d\zeta .
\end{equation*}
This gives~\eqref{eq:der-theta}. Applying the same estimate after differentiation in the auxiliary variable proves continuity in every seminorm.
\end{proof}

	
\subsection{Holomorphy on \texorpdfstring{$\KK$}{K}}
\label{subsec:charts}

We now specify the notion of holomorphy on $\KK$. Recall that $\KK$ is compact and perfect, see~\cite{CMS}.
We split $ \KK$ into the two compact perfect subsets
\begin{equation*}
\KK^{0}
:= \left\{ q  \in \KK  :  | q | \le e^\pi \right\} 
\qquad
\KK^\infty
:= \left\{ q  \in \KK  :  | q | \ge e^{-\pi} \right\} . 
\end{equation*}
This choice is tuned so that both of them contain an open neighborhood of the closure of  $   \e(  \T ) \cap \KK  $.   
 We use on them the restrictions of the standard
charts
\begin{equation*}
 q \in \KK^0   \mapsto q \in \C
\qand 
  q \in \KK^\infty   \mapsto q^{-1} \in \C 
\end{equation*}
Observe that, since $\AA=-\AA$, the two chart images coincide:
\begin{equation*}
\KK^0 = \left( \KK^\infty \right)^{-1} 
\subset\C.
\end{equation*}

We first prove its quasiconvexity. Combined with
Lemma~\ref{lem:whitney-quasiconvex}, this will allow us to work with the family of seminorms defined below.
\begin{lemma}
\label{lem:quasiconvex}
The set $\AA\subset\C$ is $2$-quasiconvex and  $\KK^0$ is quasiconvex. 
\end{lemma}

\begin{proof}
 
\emph{(i) Quasiconvexity of $\AA$.}
Since $ \DC $ is closed and $\R$ is separable, we can write its complement as a countable union of disjoint maximal open intervals 
\begin{equation*}
\R \setminus\DC=\bigcup_{I\in\mathcal I}I.
\end{equation*}
These intervals are exactly the connected components of $\R \setminus\DC$. 
For $I=(a,b)\in\mathcal I$, let us denote the open complex square with diagonal $I$ by 
\begin{equation*}
Q_I
:=
\left\{x+iy:a<x<b,\ |y|<\min(x-a,b-x)\right\}.
\end{equation*}
Since $I$ is maximal, $ \partial Q_I \subset \AA$. Since this holds for any such interval, the complement of $\AA$ is the union of the diamonds $Q_I$ for $I\in\mathcal{I}$. 
\begin{figure}[H]
    \centering
    \includegraphics[width=1\linewidth]{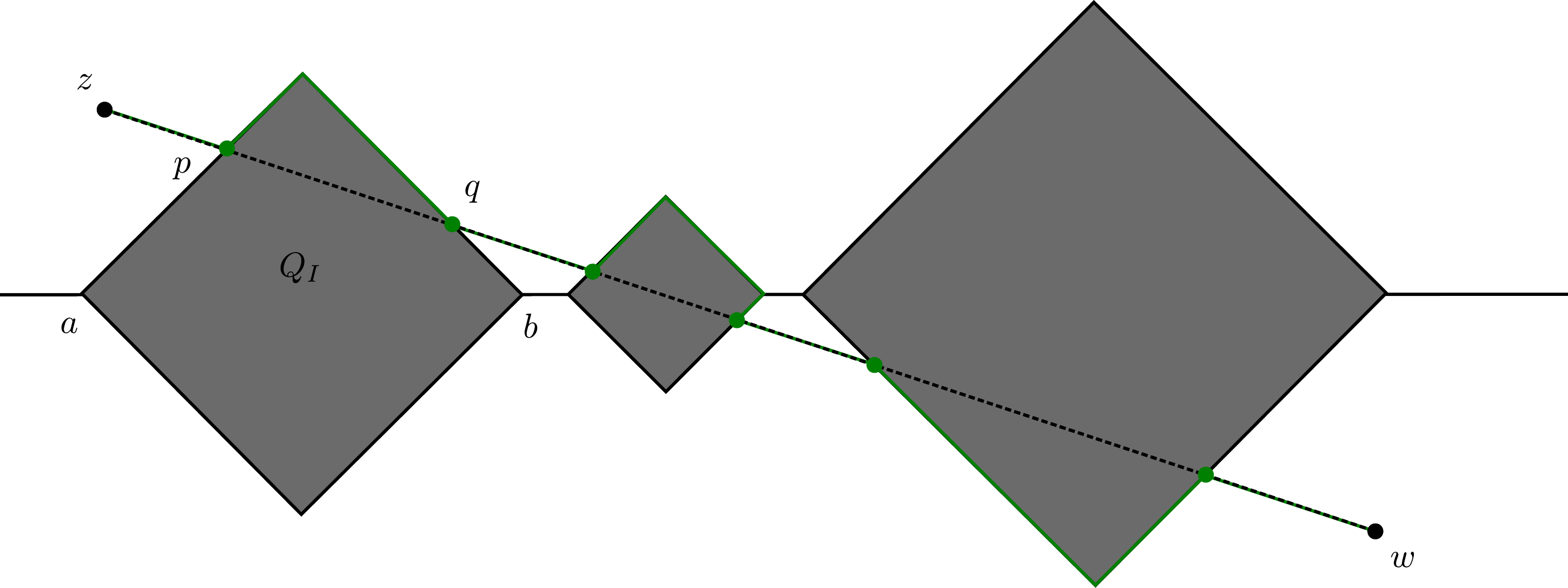}
    \caption{A diamond $Q_I$ for $I=(a,b)\subset\R\setminus\DC$ and a
    path, in green, joining points $z$ and $w$ inside $\AA$.}
    \label{fig:dia-q}
\end{figure}

Let $z,w\in\AA$, with $z\neq w$, and denote by $L:=|z-w|$ their distance. We now construct a $1$-Lipschitz parametrization $\Gamma$ with length $2L$ joining them in $\AA$. 
Start with the affine parametrization $  \Gamma_0 : t\in [0,2L] \mapsto z+\frac{t}{2L}(w-z) $, which is $1/2$-Lipschitz. 
For every $I\in\mathcal I$ such that $Q_I$ meets $[z,w]$,
we denote by $ (p_I,q_I) $ the interval  $ Q_I \cap [z,w]$. We also denote by $ a_I$ and $b_I$ the parameters such that $ \Gamma_0 (a_I) = p_I$ and $ \Gamma_0 ( b_I) = q_I$. 

Fix one of the shortest boundary arcs of $Q_I$ joining $p_I$ to $q_I$.  Since $Q_I$ is a square, this arc has length at most $2|p_I-q_I|=b_I-a_I$. 
Parametrize this path by arclength at unit speed (with adapted sign) and keep the parametrization
constant at $q_I$ for any remaining time. We denote the resulting
$1$-Lipschitz parametrization by $ \eta_I : [a_I,b_I]\longrightarrow\partial Q_I$. 
This surgery is illustrated in Figure~\ref{fig:dia-q}.
We are then ready to provide the parametrization $ \Gamma$ as 
\begin{equation*}
    \Gamma : t \in [0, 2L ] \mapsto 
    \begin{cases}
        \eta_I(t),
        &\text{if } t\in(a_I,b_I) 
        \text{ for some }I \in \mathcal{I} ,\\
        \Gamma_0(t),
        &\text{otherwise}.
    \end{cases}
\end{equation*}
Since the intervals are disjoint and countable, the path $ \Gamma$ joins $z$ to $w$, is well defined and $1$-Lipschitz, and has length $2L$. This ends the proof of the first assertion.

To prove the second assertion, note that, as a corollary of the proof, the set $\AA^{\mathrm{cyl}}:=\AA \cap \{\omega\in\C:|\Im \omega|\le 1/2\}$ is also $2$-quasiconvex. 
Indeed, since every $I$ has length at most $1$, it follows that $ \partial Q_I \subset \left\{ | \Im z | \le 1/2 \right\}$.  Hence  every pair of points in $\AA^{\mathrm{cyl}}$ is joined in $\AA^{\mathrm{cyl}}$ by the path constructed above. 

\smallskip
\noindent\emph{(ii) Quasiconvexity of $ \KK^0$.}
Observe that  $ \KK^0 = \e(\AA^{\mathrm{cyl}} ) \cup  \D $, where $\D$ is the closed disc of radius $e^{-\pi}$ centred at $0$.
On the closed strip $\{|\Im\omega|\le 1/2\}$, the exponential map, after quotienting by integer translations, induces the 
diffeomorphism 
\begin{equation*}
    \e:
    \{\omega\in\C:|\Im\omega|\le 1/2\} / \Z 
    \longrightarrow
    \{q\in\C:e^{-\pi}\le |q|\le  e^\pi\}.
\end{equation*}
This shows that $\e(\AA^{\mathrm{cyl}} ) $ is quasiconvex by quasiconvexity of $ \AA^{\mathrm{cyl}} $. The disc $\D$ is convex and its boundary circle is contained in $\e(\AA^{\mathrm{cyl}})$. Concatenating a segment in $\D$ with a quasiconvex path in $\e(\AA^{\mathrm{cyl}})$ therefore proves that their union $\KK^0$ is quasiconvex.
\end{proof}

 We say that a map  $f : \KK \to B$ from $ \KK$ to a Fréchet space $ B $ is holomorphic if the restriction of $ f $ to $ \KK^0$ is holomorphic and the map 
\begin{equation*}
 f^\# : q \in \KK^0  \mapsto f (  q^{-1}) 
\end{equation*}
is holomorphic as well. 
We denote by $ \Hol^\infty ( \KK , B)$ the corresponding space. Note that this is also a Fréchet space. 

 Actually, this article only deals with the cases of $ B = \Hol_\rho (\T)$ or $ B = \C$.  

By Lemmas~\ref{lem:quasiconvex} and~\ref{lem:whitney-quasiconvex}, the $C^\infty$-Whitney holomorphic topology on $\Hol^\infty ( \KK , \C )$ is given by  

\begin{equation}
\label{eq:def-scale-KK-scalar}
|f|_{m}
:=
\max_{0\le  k\le  m}
\sup_{q\in\KK^0}
\max\left\{  |\partial_q^k f(q)| , |\partial_q^k f^\#(q)|  \right\},
\end{equation}
indexed by $m\ge 0$. 

Similarly, the $C^\infty$-Whitney holomorphic topology on  $ \Hol^\infty ( \KK , \Hol_\rho (\T) ) $  is given by the family of seminorms 
\begin{equation}
\label{eq:def-scale-KK-vector}
|f|_{\rho',m}
:=
\max_{0\le  k\le  m}
\sup_{q\in\KK^0}
\max\left\{
|\partial_q^k f(q)|_{\rho'},
|\partial_q^k f^\#(q)|_{\rho'}
\right\},
\end{equation}
indexed by $m\geq0$ and $0<\rho'<\rho$.


\subsection{Combinatorial estimates}
We finish this preliminary section with auxiliary combinatorial estimates that are used in the proof of the KAM theorem.

\begin{lemma}
\label{lem:sharp-binomial-sum}
For all integers $p,m\ge0$,
\begin{equation*}
    p^m \le m!\binom{p+m}{m}.
\end{equation*}
Consequently, for every $r\in(0,1)$ and $m\ge0$,
\begin{equation*}
    \sum_{k=0}^\infty k^m r^k \le \frac{m!}{(1-r)^{m+1}}.
\end{equation*}
\end{lemma}

\begin{proof}
Since
\begin{equation*}
    \binom{p+m}{m}
    = \frac{(p+1)(p+2)\cdots(p+m)}{m!}
    \ge \frac{p^m}{m!},
\end{equation*}
the first assertion follows. Summing this inequality gives
\begin{equation*}
    \sum_{k\ge0} k^m r^k
    \le m!\sum_{k\ge0}\binom{k+m}{m}r^k
    = \frac{m!}{(1-r)^{m+1}},
\end{equation*}
where the last equality is the standard generating-function identity.
\end{proof}

\begin{lemma}
\label{lem:fubini-bound}
For every integer $\ell\ge0$,
\begin{equation*}
    \sum_{p=0}^{\ell} \stirling{\ell}{p}\,p!
    \le 2^{\ell}\,\ell!.
\end{equation*}
\end{lemma}

\begin{proof}
The identity
\begin{equation*}
    j^{\ell} = \sum_{p=0}^{\ell}
    \stirling{\ell}{p}\,p!\binom{j}{p}
\end{equation*}
and the equality
\begin{equation*}
    \sum_{j\ge0}\binom{j}{p}2^{-j-1}=1
\end{equation*}
give
\begin{equation*}
    \sum_{p=0}^{\ell}\stirling{\ell}{p}\,p!
    = \sum_{j\ge0}j^{\ell}2^{-j-1}.
\end{equation*}
Lemma~\ref{lem:sharp-binomial-sum}, applied with $r=1/2$, bounds the
right-hand side by $2^{\ell}\ell!$.
\end{proof}

\begin{lemma}
\label{lem:k-sum}
Let $\alpha\ge0$ and $0<D\le1$. Then
\begin{equation}
    \label{eq:k-sum-rigorous}
    \sum_{k\ne0}|k|^\alpha e^{-2\pi D|k|}
    \le \frac{K(\alpha)}{D^{\alpha+1}},
    \qquad
    K(\alpha):=
    \frac{2^{\alpha+1}\left((2\pi)^\alpha+\Gamma(\alpha+1)\right)}
    {(2\pi)^{\alpha+1}}.
\end{equation}
\end{lemma}

\begin{proof}
Put $\widetilde D:=2\pi D$. For $t\in[k-1,k]$ and $k\ge1$,
\begin{equation*}
    (t+1)^\alpha e^{-\widetilde Dt}
    \ge k^\alpha e^{-\widetilde Dk}.
\end{equation*}
Integrating over $t$ and summing over $k$ gives the bound
\begin{equation*} \int_0^\infty(t+1)^\alpha e^{-\widetilde Dt}\,dt \ge 
    \sum_{k\ge1} k^\alpha e^{-\widetilde Dk}
    .
\end{equation*}
Since $(t+1)^\alpha\le2^\alpha(1+t^\alpha)$ for $t\ge0$,
\begin{equation*}
    \int_0^\infty(t+1)^\alpha e^{-\widetilde Dt}\,dt
    \le2^\alpha\left(
        \frac1{\widetilde D}
        +\frac{\Gamma(\alpha+1)}{\widetilde D^{\alpha+1}}
    \right).
\end{equation*}
Using $1/\widetilde D\le(2\pi)^\alpha/
\widetilde D^{\alpha+1}$ and doubling for the negative integers gives
\eqref{eq:k-sum-rigorous}.
\end{proof}

\section{A fully holomorphic KAM theorem}\label{sec:KAM}

Recall that  the Euler--Lagrange equation for a parametrized invariant curve of a
generalized standard map with potential $V$ reduces to
\begin{equation*}
 V' ( \theta + u ( \theta) ) = u ( \theta + \omega) + u (\theta - \omega) - 2 u (\theta)  \;.
\end{equation*}
We seek, for fixed $0<\rho_\infty<\rho$ and an open neighbourhood
$\mathcal U$ of $0$ in $\Hol_\rho(\T)$, a map
\[
U\in\Hol^\infty\!\left(
\mathcal U,\Hol^\infty(\KK,\Hol_{\rho_\infty}(\T))
\right)
\]
such that, for every $V\in\mathcal U$ and $q\in\KK$,
$u=U(V)(q)$ solves the above equation.
In particular, if $q=\e(\omega)$ for some $\omega\in\DC$ and $ V$ is real, then $\id_\T+U(V)(q)\vert_\T$ is a parametrization of an invariant curve of rotation number $\omega$. 
We rewrite the Euler--Lagrange equation functionally as 
\begin{equation*}
     V ' ( \id_\T  + U(V) ) = \Delta U (V)
\end{equation*}
where
\begin{equation*}
    \Delta U(V,e^{2i\pi\omega},\theta)
    :=U(V,e^{2i\pi\omega},\theta+\omega)
      +U(V,e^{2i\pi\omega},\theta-\omega)
      -2U(V,e^{2i\pi\omega},\theta).
\end{equation*}

\subsection{Adapted spaces and the KAM theorem}
\label{subsec:Cpm}
Throughout, $\omega \in \AA$ denotes the \emph{complex rotation number} under
consideration, and we write $q:= \e  ( \omega)  \in \KK$. For
$f:\KK\times\T\to\C$, we formally write the forward
and backward \emph{shifts}
\begin{equation*}
    f^+(q,\theta) := f(q,\theta+\omega),
    \qquad
    f^-(q,\theta) := f(q,\theta-\omega).
\end{equation*}
These are involved in the Euler--Lagrange equation above.
Note that they act on 
Fourier coefficients by
\begin{equation*}
\widehat{f^+}_k(q)=q^k\widehat f_k(q),
\qquad
\widehat{f^-}_k(q)=q^{-k}\widehat f_k(q),
\qquad k\in\Z.
\end{equation*}
These definitions extend naturally at $ q= 0$ and $  q = \infty$ by taking limits when they exist. 
In general, neither $f^+$ nor $f^-$ is
holomorphic. In the spirit of~\cite{CMS}, we therefore
consider the following spaces.

We first denote holomorphic 
functions on $\T_\rho$ without the zero mean condition by
\begin{equation*}
  \mathscr H_\rho(\T):=\Hol_\rho(\T)\oplus\C.
\end{equation*}

We abbreviate the notation as follows:
\begin{equation*}
  \mathscr C_\rho
  :=\Hol^\infty\left(\KK,\mathscr H_\rho(\T)\right),
  \qquad
  C_\rho
  :=\Hol^\infty\left(\KK,\Hol_\rho(\T)\right).
\end{equation*}

More precisely, for $f\in\mathscr C_\rho$, we say that $f^+$ is
defined if there exists $g\in\mathscr C_\rho$ such that
\begin{equation*}
\widehat g_k(q)=q^k\widehat f_k(q)
\qquad
\text{for every }q\in\KK\cap\C^\ast\text{ and }k\in\Z.
\end{equation*}
Such a $g$ is unique and is denoted by $f^+$. The backward shift
$f^-$ is defined analogously using $q^{-k}$. Membership in
$\mathscr C_\rho$ provides, in particular, the required extensions at
$q=0$ and $q=\infty$.

\begin{definition}
\label{def:Cpm}
For $\rho>0$ and $\gamma,\tau>0$, and for a sign
$\varepsilon\in\{+,-\}$, set
\begin{align*}
\mathscr C_\rho^{\varepsilon}
&:=
\left\{f\in \mathscr C_\rho:
f^\varepsilon\in \mathscr C_\rho \right\},\\
C_\rho^{\varepsilon }
&:=
\left\{f\in C_\rho :
f^\varepsilon\in C_\rho \right\}.
\end{align*}
We also use $\pm$ to mean that both shifts have the required regularity and denote 
\begin{equation*}
\mathscr C_\rho^{\pm}
:=
\mathscr C_\rho^{+}\cap\mathscr C_\rho^{-},
\qquad
C_\rho^{\pm}
:=
C_\rho^{+}\cap C_\rho^{-}.
\end{equation*}

The spaces denoted by $C$ retain the zero-mean condition, whereas
those denoted by $\mathscr C$ allow constants.

For $0<\rho'<\rho$ and $m\ge0$, consider the seminorms 
\begin{equation*}
    |f|_{+,\rho',m} := \max\left( |f|_{\rho',m},|f^+|_{\rho',m}\right) ,
    \qquad
    |f|_{-,\rho',m} := \max\left( |f|_{\rho',m},|f^-|_{\rho',m}\right) .
\end{equation*}
For $0<\rho'<\rho$, also set
\begin{equation*}
  |\cdot|_{\pm,\rho',m}
  :=\max(|\cdot|_{+,\rho',m},|\cdot|_{-,\rho',m}).
\end{equation*}
Then for every $\varepsilon\in\lbrace+,-,\pm\rbrace$ the spaces $C^\varepsilon_\rho$ and $\mathscr C^\varepsilon_\rho$ are Fréchet spaces endowed with the family of seminorms $(|\cdot|_{\varepsilon,\rho',m})_{0<\rho'<\rho,\,m\ge 0}$. 

\end{definition}

We are now ready to state the KAM theorem. 
\begin{theo}
\label{thm:hol-KAM}
Fix $\tau,\rho>0$, take $0<\gamma<\gamma^\ast(\tau)$ and
$\rho_\infty\in(0,\rho)$.
There exists a neighbourhood $\mathcal{U} \subset \Hol_\rho(\T)$
of the zero potential and a map
\begin{equation*}
  U:\mathcal U\longrightarrow C^\pm_{\rho_\infty}
\end{equation*}
such that for every $V  \in  \mathcal{U}$,
\begin{equation}
  \label{eq:EL-U}
  V'\left( \id_\T + U(V)\right)
  =  \Delta U(V).
\end{equation}
The map $U$ is tame holomorphic.
Moreover, $U(0)\equiv0$, and
\begin{equation*}
 \int_{\T}U(V,q,\theta)\,d\theta=0
 \qquad\text{for every }V\in\mathcal U\text{ and }q\in\KK.
\end{equation*}

\smallskip
If $V\in\mathcal U\cap\Hol_\rho^\R(\T)$ and $\omega\in\DC$, then
\begin{equation*}
  \id_\T+U(V,e^{2\pi i\omega})
\end{equation*}
is a degree-one orientation-preserving analytic circle diffeomorphism.
The corresponding invariant graph is a KAM curve of rotation number
$\omega$.
\end{theo}

The adapted space $C^\pm_{\rho_\infty}$ ensures that both shifts $U(V)^+$ and $U(V)^-$ are holomorphic. This is needed to prove holomorphy of the $\beta$-function. The proof of Theorem~\ref{thm:hol-KAM} occupies the rest of this section. We first recall the construction of the Levi--Moser iteration scheme in the specific case of standard maps.

\subsection{Levi--Moser's iteration scheme}
\label{subsec:levy-moser}
To solve the Euler--Lagrange equation $V'(\id_\T+u)=\Delta u$, we first follow the proof of Carminati--Marmi--Sauzin, which uses the Levi--Moser iteration scheme, but one essential and important distinction is that we shall keep track of $C^\infty$-holomorphy along all the steps.
  
For $ u \in \Hol^\infty ( \KK , \Hol_\rho ( \T) ) $ and $ V \in \Hol_\rho (\T) $ for some $ \rho > 0$, we denote the \emph{error operator} of the Euler--Lagrange equation by
\begin{equation*}
    \mathcal{E} (u , V ) := V' ( \id_\T + u) - \Delta u \;.
\end{equation*}
 To solve $\mathcal{E} (u,V) = 0$, the first natural idea is to use the classical Newton iteration scheme. However, it appears that  $D_u\mathcal{E}(u,V)$ does
not admit a bounded inverse and Newton's method cannot be applied.  The idea of Levi--Moser
\cite{levi2001lagrangian}, used in \cite{CMS} for standard maps, is to make
\begin{equation*}
  \mathcal{E}(u_n,V)+D_u\mathcal{E}(u_n,V)[u_{n+1}-u_n]
\end{equation*}
quadratically small even though it is nonzero. Given an approximate
solution $u$, we find the new approximate solution $u+h$ by solving
\begin{equation}
    \label{eq:levi-moser}
    \mathcal{E}(u,V) + D_u\mathcal{E}(u,V)[h]
    = h \cdot
    \frac{D_u\mathcal{E}(u,V)[1+\partial_\theta u]}{1+\partial_\theta u}.
\end{equation}
This equation can be solved with smooth dependence on the parameters.
Moreover, a direct computation shows that such a solution $h$ satisfies
\begin{equation*}
    \mathcal{E}(u+h,V)
    = \frac{h}{1+\partial_\theta u}\,\partial_\theta\mathcal{E}(u,V)
    + \text{``quadratic term''}.
\end{equation*}
If $h$ has the size of $\mathcal E(u,V)$, it follows that the new
error $\mathcal E(u+h,V)$ is quadratic in the appropriate topology; see
Lemma~\ref{lem:exact-identity} for a precise statement and proof.

We write $[A,B]_\Delta := A\Delta B - B\Delta A$ for the commutator
induced by $\Delta$. Let us show precisely how, in the context of
generalized standard maps, equation~\eqref{eq:levi-moser} becomes an
equation of this commutator form. Write $A:=1+\partial_\theta u$ and
$E:=\mathcal E(u,V)$, so that $D_u\mathcal E(u,V)[k] =
V''(\id_\T+u)k-\Delta k$ for any $k$, and, by the chain rule,
\begin{equation*}
    D_u\mathcal E(u,V)[A] = A\,V''(\id_\T+u) - \Delta A
    = \partial_\theta\left[V'(\id_\T+u)\right] - \Delta A
    = \partial_\theta E,
\end{equation*}
the last equality using $\Delta A = \Delta(\partial_\theta u) =
\partial_\theta(\Delta u)$, since $\partial_\theta$ and $\Delta$ commute
and $E=V'(\id_\T+u)-\Delta u$. Multiplying~\eqref{eq:levi-moser} by
$A$ and substituting this identity on the right-hand side provides
\begin{equation*}
    AE + A\left[V''(\id_\T+u)h - \Delta h\right] = h\,\partial_\theta E \; . 
\end{equation*}
Since $AV''(\id_\T+u)=\partial_\theta[V'(\id_\T+u)]$ and, using
$E=V'(\id_\T+u)-\Delta u$ again, $\partial_\theta E -
\partial_\theta[V'(\id_\T+u)] = -\partial_\theta(\Delta u) = -\Delta A$,
this rearranges to
\begin{equation*}
    AE = A\Delta h - h\Delta A = [A,h]_\Delta.
\end{equation*}
Writing $h=:AS$, we
have shown that~\eqref{eq:levi-moser} becomes exactly
\begin{equation}
    \label{eq:cohom}
    [A,AS]_\Delta = AE.
\end{equation}

To solve this equation, we introduce appropriate solution spaces that also
ensure the required holomorphicity of
$u(\theta)=U(V,e^{2i\pi\omega},\theta)$.

\subsection{The discrete operators and their regularity}
\label{subsec:discrete}
We decompose the discrete Laplacian $\Delta f=f^++f^--2f$ into two intermediate discrete difference operators to solve the cohomological equation. 
For $f\in\mathscr C_\rho$, consider the \emph{forward discrete derivation} $\nabla$ and the \emph{backward
discrete derivation} $\nabla_-$
\begin{equation}
    \label{eq:nabla-def}
    \nabla f := f^+ - f,
    \qquad
    \nabla_- f := f - f^- \;. 
\end{equation}
Then the discrete Laplacian $ \Delta$ is equal to 
\begin{equation}
\Delta  = \nabla  - \nabla_-  = \nabla  \nabla_-  =  \nabla_- \nabla   .
\end{equation}

The discrete shifts act diagonally on the Fourier modes $e_k(\theta):=e^{2i\pi k\theta}$ for $k\in\Z$; so do the operators $\nabla$, $\nabla_-$ and $\Delta$. 
Precisely, for $ f \in \mathscr{C}_\rho $ and  $ k \in \Z $, we have
\begin{align*}
    \widehat{(\nabla f)}_k(q) &= (q^k-1)\,\hat f_k(q), \\
    \widehat{(\nabla_- f)}_k(q) &= (1-q^{-k})\,\hat f_k(q)
    = q^{-k}(q^k-1)\,\hat f_k(q), \\
    \widehat{(\Delta f)}_k(q) &= (q^k-2+q^{-k})\,\hat f_k(q)
    = q^{-k}(q^k-1)^2\hat f_k(q).
\end{align*}
If $f$ furthermore has zero mean, that is $f \in C_\rho$,  
we define formal inverses $\nabla^{-1}f$ and $\nabla_-^{-1}f$ by their Fourier coefficients
\begin{equation*}
    \widehat{(\nabla^{-1}f)}_k(q) := \lambda_k(q)\,\hat f_k(q),
    \qquad
    \widehat{(\nabla_-^{-1}f)}_k(q) := q^k\lambda_k(q)\,\hat f_k(q),
    \qquad k\neq0,
\end{equation*}
where
\begin{equation*}
    \lambda_k(q):=\dfrac{1}{q^k-1} \; .
\end{equation*}
At the formal level we indeed have  
\begin{equation*}
  \nabla(\nabla^{-1}f)=f
  \qand
  \nabla_-(\nabla_-^{-1}f)=f \;. 
\end{equation*}
 Proposition~\ref{prop:Cpm-properties} below shows that
these formal inverses are well-defined tame continuous maps on the relevant
function spaces. To prove this proposition, we first need to control the derivatives of $\lambda_n$. 

\begin{lemma}\label{lem:der-lambda-decomposition}
For every $n\in\N^*$ and $k\ge1$,  there are polynomials
$a_{k,j}\in\Z[X]$ for $1\le j\le k$, such that for every $q \in \C $ with $q^n \neq 1$, the $k$-th derivative of $\lambda_n$ can be written as
\begin{equation} \label{eq:decomposition}
    \partial_q^k\lambda_n(q) = \left(\sum_{j=1}^{k} a_{k,j}(n)\,q^{jn-k}\right)\lambda_n(q)^{k+1} .
\end{equation}

Moreover, the following properties hold: 
\begin{enumerate}
    \item[(i)] $a_{k,j}(n)=0$ whenever $jn<k$; 
    \item[(ii)] every nonzero $a_{k,j}(n)$ has sign $(-1)^k$;
    \item[(iii)]
    $\displaystyle\sum_{j=1}^k|a_{k,j}(n)| = k!\,n^k$.
\end{enumerate}
\end{lemma}

\begin{proof}
Let $n$ be fixed and proceed  by induction on $k\ge 1$.
For $k=1$, differentiating $\lambda_n(q)=1/(q^n-1)$ directly gives
\begin{equation}
    \label{eq:lambda-prime}
    \lambda_n'(q) = -n\,q^{n-1}\,\lambda_n(q)^2,
\end{equation}
which provides the base case  with $a_{1,1}(n):=-n$.

Assume~\eqref{eq:decomposition} holds at order $k\ge1$.
We extend the family of polynomials by setting $a_{k,j}:=0$ for $j\notin\{1,\ldots,k\}$. Differentiating  $q^{jn-k}\,\lambda_n(q)^{k+1}$ for  $j \in \N$  provides
\begin{equation} \label{eq:der-monomial}
    \partial_q \left[ q^{jn-k}\,\lambda_n(q)^{k+1}\right]
    = (jn-k)\,q^{jn-k-1}\,\lambda_n(q)^{k+1}
    + (k+1)\,q^{jn-k}\,\lambda_n(q)^{k}\,\lambda_n'(q)  \, .
\end{equation}
The identity
$\lambda_n(q)^{k+1}=(q^n-1)\lambda_n(q)^{k+2}
=q^n\lambda_n(q)^{k+2}-\lambda_n(q)^{k+2}$ gives
\begin{equation*}
    q^{jn - k -1} \lambda_n (q)^{k+1} =  \left( q^{(j+1)n - (k +1)}  - q^{ jn - (k+1)}  \right) \lambda_n(q)^{k+2} \; .
\end{equation*}

Together with~\eqref{eq:lambda-prime} and~\eqref{eq:der-monomial} the latter equation implies
\begin{equation*}
     \partial_q \left[ q^{jn-k}\,\lambda_n(q)^{k+1}\right]
    =  \left[  (jn-k) ( q^{(j+1)n-(k+1)} - q^{ jn - (k+1)}  )
     - n (k+1) q^{(j+1)n-(k+1)} \right] \lambda_n(q)^{k +2 } \; .
\end{equation*}
Then let 
\begin{equation}
    \label{eq:coeff-recursion-fresh}
    a_{k+1,j}(n)
   = -(jn-k)\,a_{k,j}(n) + \left[n(j-k-2)-k\right] a_{k,j-1}(n),
\end{equation}
which inductively constructs $a_{k+1,j} $ as a polynomial. 
We now verify $(i)-(iii)$. 
The property $(i)$ obviously holds from \eqref{eq:coeff-recursion-fresh} by separating the cases $jn<k$ and $jn=k$. The property $(ii)$ comes for free by the same recursive formula \eqref{eq:coeff-recursion-fresh}. 

It remains to prove~(iii). Set
\begin{equation*}
  S_k(n):=\sum_{j=1}^k a_{k,j}(n)
  \qand
  T_k(n):=\sum_{j=1}^k j\,a_{k,j}(n).
\end{equation*}
Summing the recursion~\eqref{eq:coeff-recursion-fresh} over $j$, the contribution
of the first term is
\begin{equation*}
    \sum_j \left[ -(jn-k) a_{k,j}(n) \right] = -n T_k(n) + k S_k(n),
\end{equation*}
and the contribution of the second term is
\begin{equation*}
    \sum_j\left[n(j-k-2)-k\right]a_{k,j-1}(n)
    = \sum_i\left[n(i-k-1)-k\right]a_{k,i}(n)
    = n\,T_k(n) - n(k+1)\,S_k(n) - k\,S_k(n).
\end{equation*}
Adding the two above contributions gives
\begin{equation*}
    S_{k+1}(n) = -n(k+1)\,S_k(n).
\end{equation*}
Since $S_k(n)= (-1)^k\,k!\,n^k $ by the induction hypothesis, it follows that 
\begin{equation*}
    S_{k+1} (n)  = (-1)^{k+1} \,(k+1)!\,n^{k+1} .
\end{equation*}
This proves~(iii).
\end{proof}

We now provide the derivative bounds for $\lambda_n$ on $\KK$.

\begin{proposition}
\label{prop:der-lambda-bound}
For every
$n \in \Z^*$, every $k \ge 0$, and every $q \in \KK^0$,
\begin{equation}
    \label{eq:der-lambda-bound}
    | \partial_q^k \lambda_n(q) |
     \le
   k! \, c_\gamma^{k+1} \, |n|^{(k+1)(\tau+1) + k} \, ,
\end{equation}
where one may take $c_\gamma:=2+\dfrac{2\sqrt2}{\gamma}$.
\end{proposition}

\begin{proof}
Fix $q \in \KK^0$.
We use throughout the following bound provided in \cite[Lemma~7]{CMS}:
\begin{equation}
    \label{eq:CMS7}
    |\lambda_n(q)| \le   \frac{\sqrt{2}}{\gamma}|n|^{1+\tau} \, ,
    \qquad \forall n \in \Z^*   .
\end{equation}
We may reduce to positive $n$. Indeed, for $n<0$, the identity
$\lambda_n=-(1+\lambda_{-n})$ shows that the estimate for $n$ follows from
the estimate for $-n$ whenever $k\ge1$. The case $k=0$ follows directly
from~\eqref{eq:CMS7}.

Now let $n\ge1$ and $k\ge1$. Since $q$ is not a root of unity,
Lemma~\ref{lem:der-lambda-decomposition} gives
\begin{equation*}
    \partial_q^k\lambda_n(q)
    = \sum_{j=1}^k a_{k,j}(n)\,q^{jn-k}\,\lambda_n(q)^{k+1}.
\end{equation*}
By (i) in Lemma~\ref{lem:der-lambda-decomposition}, the sum ranges only over $j$ such that $jn\ge k$.  We distinguish two regions. When
$|q|\le1/2$, no Diophantine structure is involved. The complementary region
$|q|\ge1/2$, which contains the Diophantine structure, is bounded away from
$0$.

\smallskip\noindent\textit{Case $|q| \le 1/2$.}
Since $jn-k\ge0$ and $|q|\le1$, $|q|^{jn-k}\le1$ for every surviving term, while $|\lambda_n(q)|\le2$. It follows that
\begin{equation*}
 | \partial_q^k \lambda_n(q) |
\le  |\lambda_n(q)|^{k+1} \,  \sum_{j=1}^k|a_{k,j}(n)|
= |\lambda_n(q)|^{k+1} k! \, n^k
\le  2 ^{k+1} k! \,  n^k   ,
\end{equation*}
which is bounded by the right-hand side of~\eqref{eq:der-lambda-bound}.

\smallskip\noindent\textit{Case $ |q|\ge 1/2$.}
Using again the identity $q^n\lambda_n(q)=1+\lambda_n(q)$, each surviving
term with $1\le j\le k$ satisfies
\begin{equation*}
    q^{jn-k}\lambda_n(q)^{k+1}
    = q^{-k}\left(q^n\lambda_n(q)\right)^j\lambda_n(q)^{k+1-j}
    = q^{-k}\left(1+\lambda_n(q)\right)^j\lambda_n(q)^{k+1-j}.
\end{equation*}
Since $|q|^{-k} \le 2^k$, by~\eqref{eq:CMS7} both
$|1+\lambda_n(q)|$ and $|\lambda_n(q)|$ are at most $ \frac{c_\gamma}{2}  n^{1+\tau}$. Hence, for every such $j$,
\begin{equation*}
    | q^{jn-k}\lambda_n(q)^{k+1}| \le 2^k \left( \frac{c_\gamma}{2} n^{1 + \tau } \right)^{k+1} \le c_\gamma^{k+1} n^{ (1+\tau) (k+1)}   \; .
\end{equation*}
Summing this estimate and using property~(iii) of
Lemma~\ref{lem:der-lambda-decomposition} gives
\begin{align*}
    | \partial_q^k\lambda_n(q)|
    &\le  \left(\sum_{j=1}^k|a_{k,j}(n)|  \right)   c_\gamma^{k+1} n^{(1+\tau)(k+1)}   \\
    &=  k! c_\gamma^{k+1} n^{(k+1)(\tau+1)+ k} \\
    &\le k! c_\gamma^{k+1} n^{(k+1)(\tau+2)} .
\end{align*}
This concludes the second case and hence the proof.
\end{proof}

As a direct corollary, we bound the seminorms of $\lambda_n$ for every
$n\in\Z^*$.
\begin{corollary}
\label{cor:der-lambda-seminorm}
For every $n \in \Z^*$ and every $m \ge 0$,
\begin{equation}
    \label{eq:der-lambda-seminorm}
    |\lambda_n|_{m} \le  m!\,c_\gamma^{\,m+1}\,|n|^{(m+1)(\tau+2)},
\end{equation}
where $c_\gamma$ is as in Proposition~\ref{prop:der-lambda-bound}. In
particular, $\lambda_n \in \Hol^\infty(\KK,\C)$ for every $n \in \Z^*$.
\end{corollary}

\begin{proof}
Recall that 
\begin{equation*}
    |\lambda_n|_{m }
    = \max_{0\le k\le m}\max\left\{
    \sup_{ q \in  \KK^0 }
    |\partial_q^k\lambda_n(q)|,
    \sup_{ q \in \KK^0 }
    |\partial_q ^k\lambda_n^\#(q)|\right\}
     \;.
\end{equation*}
The first term is controlled directly by
Proposition~\ref{prop:der-lambda-bound}.
For the second term, the identity
$\lambda_n^\# = \lambda_{-n}$ shows that the same bound applies with
$n$ replaced by $-n$, proving the stated bound.
For fixed $n$, the poles of $\lambda_n$ are finitely many roots of unity,
all disjoint from $\KK$. Hence $\lambda_n$ is holomorphic. 
\end{proof}

We now  provide the  basic properties of the adapted Fréchet spaces and their operators that will be used in the proof of  Theorem~\ref{thm:hol-KAM}. 
\begin{proposition}
\label{prop:Cpm-properties}
Fix $\rho,\gamma,\tau>0$. We have the following properties. 
\begin{enumerate}
    \item \emph{(Algebra.)} For $f,g\in\mathscr C_\rho^{+}$, we have 
    $fg\in\mathscr C_\rho^{+}$. For every $m\geq0$, we have the bound
    \begin{equation*}
        |fg|_{+,\rho',m} \le 2^m\,|f|_{+,\rho',m}\,|g|_{+,\rho',m}
    \end{equation*}
    for every $\rho'<\rho$. 
    Furthermore, for $ m \ge 1$  we have the sharper estimate
    \begin{equation}
    \label{eq:leibniz-split}
    |fg|_{+,\rho',m}
    \le |f|_{+,\rho',m}|g|_{+,\rho',0}
       +|f|_{+,\rho',0}|g|_{+,\rho',m}
       +2^m|f|_{+,\rho',m-1}|g|_{+,\rho',m-1} \;. 
    \end{equation}
    The same holds by replacing $+$ by $-$ or $\pm$. 
    \item \emph{(Action of $\nabla,\nabla_-$.)} The operators $\nabla,\nabla_- $ are linear maps
    \begin{equation*}
        \nabla : \mathscr C_\rho^{+} \to C_\rho,
        \qquad
        \nabla_- : \mathscr C_\rho^{-} \to C_\rho,
    \end{equation*}
    with $|\nabla f|_{\rho',m}\le2|f|_{+,\rho',m}$ and
    $|\nabla_-f|_{\rho',m}\le2|f|_{-,\rho',m}$ for every integer $m$ and for every $ 0 < \rho' < \rho$ without strip loss. 

\item \emph{(Inversion of $\nabla,\nabla_-$.)} The discrete operators
    \begin{equation*}
        \nabla^{-1} : C_\rho \to C_\rho^{+},
        \qquad
        \nabla_-^{-1} : C_\rho \to C_\rho^{-},
    \end{equation*}
    are well defined on the zero mean subspace. 
Moreover, for every $f\in C_\rho$, every $0<\rho''<\rho'<\rho$ such that
$\rho'-\rho''\le1$ and $ m \ge 0$, they satisfy the tame estimates
    \begin{equation}
        \label{eq:nabla-inv-tame-Cpm}
        |\nabla^{-1}f|_{+,\rho'',m} \le 
        \frac{ \widehat\xi_{m,\gamma,\tau}}{(\rho'-\rho'')^{(m+1)(\tau+2)+1}}\,|f|_{\rho',m}
 \end{equation}
 and
 \begin{equation*}
        |\nabla_-^{-1}f|_{-,\rho'',m} \le 
        \frac{ \widehat\xi_{m,\gamma,\tau}}{(\rho'-\rho'')^{(m+1)(\tau+2)+1}}\,|f|_{\rho',m},
    \end{equation*}
    where
    \begin{equation}
        \label{eq:Chat-def}
        \widehat\xi_{m,\gamma,\tau}
        := (m+1)^2\,2^{m+1} \, m! \, c_\gamma^{m+1} \, K \left( (m+1)(\tau+2) \right)
    \end{equation}
with $c_\gamma$ as in Proposition~\ref{prop:der-lambda-bound} and $K(\cdot)$ the constant of Lemma~\ref{lem:k-sum}.
\end{enumerate}

\end{proposition}

\begin{proof}
\smallskip\noindent\textit{(i) Algebra property.}
For every $0\le k\le m$, the $k$-th derivative of $fg$ exists by the Leibniz rule and we have the bound 
\begin{equation} \label{eq:leibniz-step}
   |\partial_q^k(fg)|_{\rho',0}
   \le\sum_{j=0}^k\binom{k}{j}
   |\partial_q^jf|_{\rho',0}|\partial_q^{k-j}g|_{\rho',0}
   \le2^m|f|_{\rho',m}|g|_{\rho',m}.
\end{equation}
Using
$\partial_q^k\left((fg)^{+}\right)=\partial_q^k(f^+g^+)$, we similarly
obtain $|\partial_q^k((fg)^{+})|_{\rho',0}
\le2^m|f^+|_{\rho',m}|g^+|_{\rho',m}$. Taking the supremum over $k\le m$ proves the first algebra property. 
For $m\geq1$, the same equation \eqref{eq:leibniz-step} provides the second inequality. 
The proof for the other choices of signs is similar.

\smallskip\noindent\textit{(ii) Action of $\nabla$.} The triangle inequality trivially gives 
\begin{equation*}
    |\nabla f|_{\rho',m}
    \le  |f^+|_{\rho',m}+|f|_{\rho',m}
    \le  2|f|_{+,\rho',m}, 
\end{equation*}
and $ \nabla f $ has zero mean.
 The same argument applies to $\nabla_-$.

\noindent\textit{(iii) Inversion of $\nabla$.} We prove the assertion for
$\nabla^{-1}$. The proof for $\nabla_-^{-1}$ follows from the identity $  \nabla_-^{-1}(f^\#) = -\left(\nabla^{-1}f\right)^\#$. 
Applying Lemma~\ref{lem:fourier-extraction} to
$\partial_q^\ell f$, for every $\ell\leq m$ and $k\in\Z^*$, yields
\begin{equation}
    \label{eq:fourier-cauchy}
    \sup_{q\in \KK^0}
    \left|\partial_q^\ell\hat f_k(q)\right|
    \le  |f|_{\rho',m}\,e^{-2\pi\rho'|k|}.
\end{equation}
By the Leibniz rule, for $0\le\ell\le m$,
\begin{equation*}
    \left|\partial_q^\ell\left[\lambda_k(q)\hat f_k(q)\right]\right|
    \le \sum_{j=0}^\ell\binom{\ell}{j}\left|\partial_q^j\lambda_k(q)\right|
    \cdot\left|\partial_q^{\ell-j}\hat f_k(q)\right|.
\end{equation*}
Combining Corollary~\ref{cor:der-lambda-seminorm}
with~\eqref{eq:fourier-cauchy} for $\ell-j\le m$ gives
\begin{equation*}
    \sup_{q\in\KK^0}
    \left|\partial_q^\ell\left[\lambda_k(q)\hat f_k(q)\right]\right|
    \le \left(\sum_{i=0}^\ell\binom{\ell}{i}i!\,c_\gamma^{i+1}|k|^{(i+1)(\tau+2)}\right)
    |f|_{\rho',m}\,e^{-2\pi\rho'|k|}.
\end{equation*}
The exact same bound on $\KK^\infty $ is then directly obtained by using the identity $  (\lambda_k\hat f_k)^\# =\lambda_{-k} \cdot \hat f_k^\# $  on $ \KK^0 $. 
Summing over $k\neq0$ against the weight $e^{2\pi\rho''|k|}$ and
applying Lemma~\ref{lem:k-sum} with $\alpha=(i+1)(\tau+2)$
and $D=\rho'-\rho''$ gives the bound
\begin{equation*}
    |\nabla^{-1}f|_{\rho'',m}
    \le \left(\sum_{\ell=0}^m\sum_{i=0}^\ell\binom{\ell}{i}i!\,c_\gamma^{i+1}
    \frac{K\left((i+1)(\tau+2)\right)}{(\rho'-\rho'')^{(i+1)(\tau+2)+1}}\right)
    |f|_{\rho',m}.
\end{equation*}

Since $K(\alpha)$ and the power of $(\rho'-\rho'')^{-1}$ are both
increasing in $i$, bounding every term in the double sum by the
$i=\ell=m$ term, and using $\sum_{\ell=0}^m\sum_{i=0}^\ell\binom{\ell}{i}
= 2^{m+1} - 1 \le 2^{m+1}$, gives
\begin{equation}
    \label{eq:nabla-inv-tame-self}
    |\nabla^{-1}f|_{\rho'',m}
    \le  \frac{\widehat\xi_{m,\gamma,\tau}}{ 2 (\rho'-\rho'')^{(m+1)(\tau+2)+1}}\,|f|_{\rho',m},
\end{equation}
with $\widehat\xi_{m,\gamma,\tau}$ as in~\eqref{eq:Chat-def}.

\smallskip 
We now study $(\nabla^{-1}f)^+$.
We use again the identity
$q^k\lambda_k(q) = 1+\lambda_k(q)$ which reads in Fourier modes as
\begin{equation*}
    \widehat{(\nabla^{-1}f)^+}_k(q)
    = q^k\lambda_k(q)\,\hat f_k(q)
    = \left(1+\lambda_k(q)\right)\hat f_k(q)
    = \hat f_k(q) + \widehat{(\nabla^{-1}f)}_k(q) \;. 
\end{equation*}
 Hence, by the triangle inequality
and~\eqref{eq:nabla-inv-tame-self},
\begin{align*}
    |(\nabla^{-1}f)^+|_{\rho'',m}
   &\le |f|_{\rho'',m} + |\nabla^{-1}f|_{\rho'',m} \\
   &\le |f|_{\rho',m} + \frac{\widehat\xi_{m,\gamma,\tau}} {2(\rho'-\rho'')^{(m+1)(\tau+2)+1}}\,|f|_{\rho',m} \\
   &\le \frac{ \widehat\xi_{m,\gamma,\tau}}{(\rho'-\rho'')^{(m+1)(\tau+2)+1}} \,|f|_{\rho',m} \;, 
\end{align*}
since  $\rho'-\rho''\le1$, and $\widehat\xi_{m,\gamma,\tau}\ge 2 $. Combining the two bounds gives 
\begin{equation*}
    |\nabla^{-1}f|_{+,\rho'',m}
    = \max\left(|\nabla^{-1}f|_{\rho'',m},\,|(\nabla^{-1}f)^+|_{\rho'',m}\right)
    \le \frac{ \widehat\xi_{m,\gamma,\tau}}{(\rho'-\rho'')^{(m+1)(\tau+2)+1}}\,|f|_{\rho',m},
\end{equation*}
which is~\eqref{eq:nabla-inv-tame-Cpm}. 

In particular the series defining $ \nabla^{-1} f$ is normally convergent in every seminorm, providing the required holomorphy property with tame estimates. 
\end{proof}

Observe that we obtain the following corollary under a stronger regularity assumption. 
\begin{corollary}
\label{cor:nabla-inv-Cpm}
Let  $0<\rho''<\rho'<\rho$ satisfy
$\rho'-\rho''\le1$. For every $f\in C_\rho^{-}$, we have
$\nabla^{-1}f\in C_\rho^{\pm}$, and for every $m\ge0$,
\begin{equation}
    \label{eq:nabla-inv-Cpm-strong}
    |\nabla^{-1}f|_{\pm,\rho'',m}
    \le \frac{\widehat\xi_{m,\gamma,\tau}}{(\rho'-\rho'')^{(m+1)(\tau+2)+1}}\,|f|_{-,\rho',m}.
\end{equation}
\end{corollary}

\begin{proof}
The desired result is obtained by applying Proposition~\ref{prop:Cpm-properties}(iii) to $f$ and to $f^-$, using the identity $ 
\left(\nabla^{-1}f\right)^-=\nabla^{-1}\left(f^-\right)$.
\end{proof}

\subsection{Solving the cohomological equation and control of the increments} \label{subsec:cohom}
We now solve the cohomological equation~\eqref{eq:cohom}. First observe the following elementary fact.
\begin{fact} \label{fact:zero-mean}
Let $V\in\Hol_\rho(\T)$ and let $u\in C_{\rho'}^{\pm}$ for some
$0<\rho'<\rho$. Assume that the composition
$V'\circ(\id_\T+u)$ is defined on $\T_{\rho'}$ and belongs to
$\mathscr C_{\rho'}$. Then the function $ (1 + \partial_\theta u) \mathcal{E} (u,V)$ has zero mean: 
\begin{equation*}
    \int_\T(1+\partial_\theta u)\,\mathcal E(u,V)=0.
\end{equation*}
\end{fact}

\begin{proof}
Periodicity and an integration by parts yield
\begin{equation*}
 \int_\T (\partial_\theta u)u=0,
 \qand
 \int_\T(\partial_\theta u)u^+
 =-\int_\T u(\partial_\theta u)^+
 =-\int_\T(\partial_\theta u)u^-.
\end{equation*}
It follows that
\begin{equation}\label{eq:int-u}
\int_\T(1+\partial_\theta u)\Delta u=0.
\end{equation}
 On the other hand,
periodicity gives
$
 (1+\partial_\theta u)V'(\id_\T+u)
 =\partial_\theta\left[V(\id_\T+u)\right],
$
and hence
\begin{equation} \label{eq:int-V}
 \int_\T(1+\partial_\theta u)V'(\id_\T+u)=0.
\end{equation}
Subtracting \eqref{eq:int-u} from \eqref{eq:int-V} ends the proof. 
\end{proof}

We will next solve the cohomological equation with appropriate
bounds and prove that the resulting full solution map is holomorphic.
The proof requires controlling the map $A\mapsto(AA^+)^{-1}$ using the following lemma. 
\begin{lemma}
\label{lem:Ainv-linear}
Let $0<\rho'<\rho$ and let
$f\in\mathscr C_{\rho}$ satisfy
$|f-1|_{\rho',0}\le  \frac{1}{2}$. Then
$f^{-1}\in\mathscr C_{\rho'} $ and, for every $m\geq1$, 
\begin{equation}
    \label{eq:Ainv-linear}
    \left| \frac{1}{f} \right|_{\rho',m}  \le  4 \,|f|_{\rho',m}
     +  2^{m+1}\,m!\, \left(1 + 2|f|_{\rho',m-1}\right)^m \;. 
\end{equation}
At order $m=0$, one has $|f^{-1}|_{\rho',0}\leq2$.
Moreover,  the map 
\begin{equation*}
\begin{array}{rcl}
\left\{f\in\mathscr C_\rho:
|f-1|_{\rho',0}<\dfrac12\right\}
&\longrightarrow&
\mathscr C_{\rho'}
\\[4pt]
f 
&\longmapsto&
f^{-1}
\end{array}
\end{equation*}
is holomorphic. 
In particular, if $f\in\mathscr C_{\rho'}^{\pm}$ and
$|f-1|_{\pm,\rho',0}\leq\frac12$, the same estimates hold with
$|\cdot|_{\rho',m}$ replaced by $|\cdot|_{\pm,\rho',m}$ and
$f^{-1}\in\mathscr C_{\rho'}^{\pm}$.
\end{lemma}

\begin{proof}
The hypothesis $|f-1|_{\rho',0}\leq\frac12$ gives
$f^{-1}=\sum_{k\geq0}(1-f)^k$ since 
\begin{equation}
\label{eq:f-0}
    \left|\frac1f\right|_{\rho',0}
    \le 
    \sum_{k\geq0}|1-f|_{\rho',0}^k
    \le 2.
\end{equation}
This, in particular, provides the bound at $ m = 0$.
For $m\geq1$, the derivative of order $m$ decomposes as
\begin{equation*}
\partial_q^m(f^{-1})
=
\sum_{\mathcal{P} \vdash \{1,\ldots,m\}}
(-1)^{|\mathcal P|}|\mathcal P|!\,
f^{-|\mathcal P|-1}
\prod_{B\in\mathcal{P}}\partial_q^{|B|}f,
\end{equation*}
where $ \mathcal{P} \vdash \{1,\ldots,m\}$ means that $ \mathcal{P }$ ranges over the partitions of $\{1,\ldots,m\}$. 
The term $ \mathcal{P} = \lbrace  \lbrace 1 , \dots m \rbrace \rbrace $ formed by a single subset contributes
$-f^{-2}\partial_q^m f$ and, by~\eqref{eq:f-0}, satisfies the bound 
\begin{equation*}
\left|f^{-2}\partial_q^m f\right|_{\rho',0}
\le 
\left|f^{-1}\right|_{\rho',0}^2
\left|\partial_q^m f\right|_{\rho',0}
\le 
4|f|_{\rho',m}.
\end{equation*}

Every other partition has at least two subsets, and hence each
of its subsets has at most $m-1$ elements. For such a partition $ \mathcal{P}$, if we denote 
$p:=|\mathcal P|$, its contribution is therefore bounded by
\begin{equation*}
\left|
f^{-p-1}
\prod_{B\in\mathcal P}\partial_q^{|B|}f
\right|_{\rho',0}
\le 
2^{p+1}|f|_{\rho',m-1}^{\,p}.
\end{equation*}
Since the number of partitions of $\{1,\ldots,m\}$ into $p$ subsets is
$\stirling{m}{p}$, we obtain
\begin{equation*}
\left|\partial_q^m(f^{-1})\right|_{\rho',0}
\le 
4|f|_{\rho',m}
+
\sum_{p=2}^m
\stirling{m}{p}p!\,
2^{p+1}|f|_{\rho',m-1}^{\,p}.
\end{equation*}

Lemma~\ref{lem:fubini-bound} now gives that 
\begin{equation*}
\left|\partial_q^m(f^{-1})\right|_{\rho',0}
\le 
4|f|_{\rho',m}
+
2^{m+1}m!\,
\left(1+2|f|_{\rho',m-1}\right)^m.
\end{equation*} 

Since the seminorms increase with their order and the right-hand side
also increases with $m$, we obtain~\eqref{eq:Ainv-linear}. Applying this result to discrete shifts provides the last assertion.
\end{proof}

With this inverse estimate in hand, we can solve the Levi--Moser
cohomological equation while retaining explicit control at every order. 
For $m \ge  0$ and $ \gamma, \tau > 0$, let us introduce the notations 
\begin{equation} \label{eq:def-tau-m-xi}
    \xi_{m,\gamma,\tau} := 2^{\tau_m+12}\,\widehat\xi_{m,\gamma,\tau}^{\,2},
    \qquad \tau_m := 2(m+1)(\tau+2)+3.
\end{equation}
Here we recall that $ \widehat\xi_{m,\gamma,\tau} $ is the constant of \eqref{eq:Chat-def} from Proposition~\ref{prop:Cpm-properties}. 

\begin{theorem}
\label{thm:solution-hol}
Let $0<\rho''<\rho'$ with $\rho'-\rho''\le1$. Consider
$u  \in C_{\rho'}^{\pm}$, set $ A := 1 +\partial_\theta u$ and assume that 
  \begin{equation}\label{eq:A-small}
    |  A - 1 |_{\pm,\rho',0}<\tfrac{1}{8} \;. 
 \end{equation}
For every $E\in\mathscr C_{\rho'}$ with $\langle AE\rangle=0$,
there is a unique $S\in\mathscr C_{\rho'}^{\pm}$    satisfying
\begin{equation*}
    [A,AS]_\Delta=AE,
    \qquad \langle AS\rangle=0.
\end{equation*}
We denote this normalized solution by $S=\mathcal{S}(A,E)$. It is given by the explicit formula 
\begin{equation}
    \label{eq:G-formula}
    \mathcal{S}(A,E):=S_0(A,E)
    -\left\langle A S_0(A,E)\right\rangle,
\end{equation}
where
\begin{equation*}
    S_0(A,E) :=\nabla^{-1}\!\left[
    \frac{\nabla_-^{-1}(AE)-\chi(A,E)}{AA^+}\right] 
    \qand 
    \chi(A,E) :=
    \frac{\displaystyle\left\langle
    \frac{\nabla_-^{-1}(AE)}{AA^+}\right\rangle}
    {\displaystyle\left\langle(AA^+)^{-1}\right\rangle} .
\end{equation*}
The  solution map
\begin{equation*}
\begin{array}{rcl}
\left\{(A,E)\in\mathscr C_{\rho'}^{\pm}\times\mathscr C_{\rho'}:
\langle A\rangle=1,\ \langle AE\rangle=0,\
|A-1|_{\pm,\rho',0}<\frac18\right\}
&\longrightarrow&
\mathscr C_{\rho''}^{\pm}\\
(A,E)&\longmapsto&\mathcal S(A,E)
\end{array}
\end{equation*}
is locally the restriction of a holomorphic map on an open subset of
the ambient product.

\smallskip With $h:=AS$ the corresponding increment, we furthermore have the following graded bounds. At order zero, we have: 
\begin{equation*}
 |S|_{\pm,\rho'',0}\le
 \frac{\xi_{0,\gamma,\tau}}
 {(\rho'-\rho'')^{\tau_0}}|E|_{\rho',0},
 \qquad
 |h|_{\pm,\rho'',0}\le
 \frac{\xi_{0,\gamma,\tau}}
 {(\rho'-\rho'')^{\tau_0}}|E|_{\rho',0}.
\end{equation*}

For $m\ge1$, put
\begin{equation*}
    \Theta_m := 2^{\,m^3+3m^2+9m+9}\,(m!)^{m+3}  
    \left(1+|\partial_\theta u|_{\pm,\rho',m-1}\right)^{2m^2+4m+2}
    |E|_{\rho',m-1} \qand  \Theta_0:=0. 
\end{equation*}
Then the  solution $S$  and the increment $h$ satisfy
\begin{align*}
 |S|_{\pm,\rho'',m}
 &\le
    \frac{\xi_{m,\gamma,\tau}}
    {(\rho'-\rho'')^{\tau_m}}\left(
    |E|_{\rho',m}
     +|\partial_\theta u|_{\pm,\rho',m}|E|_{\rho',0}
     +\Theta_m\right),\\
 |h|_{\pm,\rho'',m}
 &\le
    \frac{\xi_{m,\gamma,\tau}}
    {(\rho'-\rho'')^{\tau_m}}\left(
    |E|_{\rho',m}
     +|\partial_\theta u|_{\pm,\rho',m}|E|_{\rho',0}
     +\Theta_m\right) \;.
\end{align*}
\end{theorem}

\begin{proof}
In addition to the notations $S$, $S_0$ and $ \chi$, set  $ R := (AA^+)^{-1}$, $ g := \nabla_{-}^{-1} ( A E)$,   
$c:=\langle AS_0\rangle$,  and
$H:=(g-\chi)R$. 

We first show that if \eqref{eq:G-formula} defines $S$, then it is, up to an additive constant, the unique solution of $ [A,AS]_\Delta = AE$. 
By definition of $S$ and since $\langle AE\rangle=0$, we have
$$      \nabla_-\left(AA^+\nabla S\right) =    \nabla_-\left(AA^+\nabla S_0 \right)   = \nabla_{-}  \left( \nabla_-^{-1}( AE) - \chi  \right) = AE  . $$  
On the other hand, expanding $\nabla_-\left(AA^+\nabla S \right)$ gives 
\begin{align*}
    \nabla_-\left(AA^+\nabla S\right)
    &= AA^+S^+ + A^-AS^- - A(A^++A^-)S \\
    &= A\Delta(AS) - AS\,\Delta A \\
    &= [A,AS]_\Delta \;. 
\end{align*}
Consequently $S$ solves the
cohomological equation  $ [A,AS]_\Delta  = AE  $. 
Moreover, since $ \langle A\rangle=1$, we have $\langle AS\rangle=\langle A(S_0-\langle AS_0\rangle)\rangle=0$, providing the zero-mean normalisation. 

It remains to prove uniqueness. Let $T$ satisfy $ [ A ,AT ]_\Delta = AE$. Then the factorization above
provides $\nabla_-(AA^+\nabla (T-S) )=0$. Hence
$AA^+\nabla (T - S)  $  is independent of the angular variable. 
 Moreover, the smallness assumption gives
$|AA^+-1|_{-,\rho',0}<17/64$ and hence, by the Neumann series,
$|(AA^+)^{-1}-1|_{-,\rho',0}<17/47$. Thus
$\left\langle(AA^+)^{-1}\right\rangle$ does not vanish.
Consequently since $\nabla$ preserves zero mean it follows 
\begin{equation*}
   0 = \langle  \nabla (T-S) \rangle  = \left\langle \frac{AA^+ \nabla (T-S)}{AA^+}  \right\rangle  =  \left\langle \frac{1}{AA^+} \right\rangle AA^+ \nabla (T-S) \; . 
\end{equation*}
This leaves no option but $ \nabla (T-S) =0 $.  Since $ \langle A \rangle = 1$, the normalizing condition $  \langle  AT  \rangle  =  \langle A S  \rangle = 0 $ provides $ T = S$. This ends the proof of unicity. 

\smallskip
We now provide the desired estimates, which in particular give the well-definedness of $S$. 
To do so, we decompose the map $\mathcal{S}$ into elementary operations and control each of them successively. 
 
To simplify the notation further, for $ j \in \N $, write
\begin{equation*}
 w_j:=|\partial_\theta u|_{\pm,\rho',j},
 \qquad e_j:=|E|_{\rho',j}.
\end{equation*}
Also set $d:=\rho'-\rho''$ and
$\rho^{\mathrm{mid}}:=\rho''+d/2= \rho' - d/2 = (\rho'+\rho'')/2$.

\smallskip\noindent\emph{(i) The product $AE$.}
For $m\geq1$, splitting off the terms of orders $0$ and $m$ in the
Leibniz sum used in~\eqref{eq:leibniz-step} gives
\begin{equation}
    \label{eq:P-split}
    |AE|_{\rho',m}    \le   ( 1 + w_0) e_m +  w_m e_0  +  2^m w_{m-1} e_{m-1}  \le  \tfrac76\,e_m  +  w_m  e_0   +  2^m\,w_{m-1} e_{m-1} \;.
\end{equation}
At order $m=0$, the ordinary product estimate gives
$|AE|_{\rho',0}\le(1+w_0)e_0\le\tfrac76e_0$.

\smallskip\noindent\emph{(ii) The discrete primitive $g:=\nabla_-^{-1}(AE)$.}
Since $\langle AE\rangle=0$, $g$ is well defined, and
Proposition~\ref{prop:Cpm-properties}(iii) applies from $\rho'$ to
$\rho^{\mathrm{mid}}$,  giving for $m\ge1$
\begin{equation*}
    |g|_{-,\rho^{\mathrm{mid}},m} \le G_m \,|AE|_{\rho',m},
    \qquad \text{where}  \quad
    G_m \;:=\; \frac{2^{(m+1)(\tau+2)+1}\,\widehat\xi_{m,\gamma,\tau}}
    {d^{(m+1)(\tau+2)+1}} \;. 
\end{equation*}
Hence, by~\eqref{eq:P-split}, we have the bound
\begin{equation} \label{eq:bound-g}
|g|_{-,\rho^{\mathrm{mid}},m}\le G_m  \left( \frac{7}{6} e_m+w_m e_0 + 2^m w_{m-1} e_{m-1}  \right) \;.
\end{equation}
At order $m-1$, for $m\geq1$, we use the cruder bound 
\begin{equation}
\label{eq:bound-g-m-1}
|g|_{-,\rho^{\mathrm{mid}},m-1 }\le  G_{m}  2^m   ( 1+ w_{m-1} )  e_{m-1}\;. 
\end{equation}
At order $m=0$, we use simply
\begin{equation}
    \label{eq:g-order0}
    |g|_{-,\rho^{\mathrm{mid}},0} \le  G_0\,e_0 ,
    \qquad \text{ where }
    G_0 \;:=\; \tfrac{7}{3}\,\frac{2^{\tau+2}\widehat\xi_{0,\gamma,\tau}}{d^{\tau+3}}.
\end{equation}

 \smallskip\noindent\emph{(iii) The inverse $R:=(AA^+)^{-1}$.}
Writing 
$AA^+ - 1 = \partial_\theta u + \partial_\theta u^+ +\partial_\theta u \cdot  \partial_\theta u^+ $ the triangle inequality yields
\begin{equation*}
    |AA^+-1|_{-,\rho',0} \le 2 |\partial_\theta u|_{\pm, \rho',0}  + |\partial_\theta u|_{\pm, \rho',0}^2
    <   \frac{17}{64}  \;. 
\end{equation*}
The Neumann series then gives
\begin{equation}
\label{eq:R-order0}
|R|_{-,\rho',0}
\le
\sum_{k\geq0}|AA^+-1|_{-,\rho',0}^k
<
\frac{64}{47}  \; ( < 2) \;. 
\end{equation}
Moreover, using the identity $ R -1 = -( AA^+ -1 ) R $ and
Proposition~\ref{prop:Cpm-properties}(i) at order zero, we obtain
\begin{equation*}
|R-1|_{-,\rho',0}
\le
|AA^+-1|_{-,\rho',0}|R|_{-,\rho',0}
<
\frac{17}{64} \cdot \frac{64}{47}
=
\frac{17}{47}
<
\frac{1}{2}.
\end{equation*}
 Applying Lemma~\ref{lem:Ainv-linear} to both $ AA^+ $ and $ A^-A $  therefore gives $R\in\mathscr C_{\rho'}^{-}$ and for $ m \ge 1$ 
\begin{equation}
    \label{eq:R-prim}
    |R|_{-,\rho',m} \le 4|AA^+|_{-,\rho',m}
     +  2^{m+1}m!\left(1+2|AA^+|_{-,\rho',m-1}\right)^{m} \;.
\end{equation}
We deduce from~\eqref{eq:leibniz-split} that
\begin{equation*}
|AA^+|_{-,\rho',m}
\le
2+\frac{7}{3}w_m+2^m w_{m-1}^2.
\end{equation*}
Moreover,
\begin{equation*}
1+2|AA^+|_{-,\rho',m-1}
\le
2^{m+1}\left(1+w_{m-1}^2\right).
\end{equation*}
For $m\geq2$, this follows from the preceding estimate at order $m-1$
and the monotonicity of the seminorms. For $m=1$, it follows directly
from $|AA^+|_{-,\rho',0}\le(7/6)^2$.
This enables the bound 
\begin{equation}
\label{eq:R-split}
\begin{split}
|R|_{-,\rho',m}
&\le
8+\frac{28}{3}w_m+2^{m+2}w_{m-1}^2 
+2^{m+1}m!
\left(
2^{m+1}\left(1+w_{m-1}^2\right)
\right)^m \\
&\le
\frac{28}{3}w_m
+2^{m^2+2m+2}m!
\left(1+w_{m-1}^2\right)^m \;. 
\end{split}
\end{equation}

We shall also use a cruder estimate at order $m-1$. For $m\geq2$,
the above \eqref{eq:R-split} gives 
\begin{equation}
\label{eq:R-m-1}
\begin{split}
|R|_{-,\rho',m-1}
&\le
\frac{28}{3}w_{m-1}
+
2^{m^2+1}(m-1)!
\left(1+w_{m-2}^2\right)^{m-1} \\
&\le
2^{m^2+2}m!
\left(1+w_{m-1}^2\right)^m.
\end{split}
\end{equation}
By~\eqref{eq:R-order0} the above estimates extend to $ m=1$ as well. 

\smallskip\noindent\emph{(iv) The constant
$\chi:=\langle gR\rangle/\langle R\rangle$.}
Since  the mean $ \langle \cdot \rangle $ is $1$-Lipschitz  it holds
\begin{equation*}
|\langle R\rangle-1|_0
\le
|R-1|_{-,\rho',0}
<
\frac12.
\end{equation*}
Thus $\langle R\rangle$ does not vanish and
$|\langle R\rangle^{-1}|_0\le2$. In particular,
\begin{equation}
\label{eq:chi-order0}
|\chi|_0
\le
2|g|_{-,\rho^{\mathrm{mid}},0}|R|_{-,\rho',0}
\le
4G_0e_0.
\end{equation}
Lemma~\ref{lem:Ainv-linear}, applied  to 
$\langle R\rangle$ considered  as a constant on $ \T$, also shows that $\chi$ is well defined and
holomorphic. 

We now prove by induction on $m\ge 0 $ that
\begin{equation}
\label{eq:chi-split}
|\chi|_m
\le
\frac{14}{3}G_me_m
+
\left(4G_m+\frac{280}{3}G_0\right)w_me_0
+
\frac{1}{32}G_m\Theta_m.
\end{equation}
The base case is given by \eqref{eq:chi-order0}. For $m \ge 1$ let us assume that \eqref{eq:chi-split} holds at $ m -1$.  
For $1\leq k\leq m$, differentiating
$\langle R \rangle \cdot \chi=\langle gR\rangle$ gives
\begin{equation*}
\langle R\rangle\,\partial_q^k\chi
=
\left\langle R\,\partial_q^kg\right\rangle
+
\left\langle(g-\chi)\,\partial_q^kR\right\rangle +
\sum_{j=1}^{k-1}\binom{k}{j}
\left\langle
\partial_q^j(g-\chi)\,
\partial_q^{k-j}R
\right\rangle.
\end{equation*}

Consequently it holds 
\begin{equation*}
 \frac{1}{2}  | \partial_q^k\chi | 
\le 
|R|_{-,\rho',0}|g|_{-,\rho^{\mathrm{mid}},m}  + 
\left(
|g|_{-,\rho^{\mathrm{mid}},0}+|\chi|_0
\right)|R|_{-,\rho',m} +
2^{m}
\left(
|g|_{-,\rho^{\mathrm{mid}},m-1}+|\chi|_{m-1}
\right)|R|_{-,\rho',m-1}.
\end{equation*}

Taking the maximum over $0\leq k\leq m$ yields
\begin{equation}
\label{eq:chi-recursion}
\begin{split}
|\chi|_m
 & \le
2|R|_{-,\rho',0}|g|_{-,\rho^{\mathrm{mid}},m}  + 
2\left(
|g|_{-,\rho^{\mathrm{mid}},0}+|\chi|_0
\right)|R|_{-,\rho',m}\\
& \quad +
2^{m+1}
\left(
|g|_{-,\rho^{\mathrm{mid}},m-1}+|\chi|_{m-1}
\right)|R|_{-,\rho',m-1}.
\end{split}
\end{equation}

We bound successively the three summands  on the right-hand side of the above~\eqref{eq:chi-recursion}. The first term is bounded  using  \eqref{eq:R-order0}  and 
the  \eqref{eq:bound-g} by 
\begin{equation} \label{eq:first-term}
\begin{split}
2|R|_{-,\rho',0}|g|_{-,\rho^{\mathrm{mid}},m}
&\le 
\frac{14}{3}G_me_m
+4G_mw_me_0
+2^{m+2}G_mw_{m-1}e_{m-1} \\ 
&\le  \frac{14}{3}G_me_m
+4G_mw_me_0
+ \frac{1}{128} G_m \Theta_m   .
\end{split}
\end{equation}
The second term  is bounded using the base case and \eqref{eq:R-split} as
\begin{equation} \label{eq:second-term}
\begin{split}
2\left(
|g|_{-,\rho^{\mathrm{mid}},0}+|\chi|_0
\right)|R|_{-,\rho',m}
&\le
10G_0e_0\left[
\frac{28}{3}w_m
+
2^{m^2+2m+2}m!
\left(1+w_{m-1}^2\right)^m
\right] \\
&\le 
\frac{280}{3}G_0w_me_0
+ \frac{1}{128} G_m \Theta_m \;. 
\end{split}
\end{equation}
The third and last term is  obtained by~\eqref{eq:R-m-1} and  the induction hypothesis: 
\begin{equation} \label{eq:third-term}
\begin{split}
&2^{m+1}
\left(
|g|_{-,\rho^{\mathrm{mid}},m-1}+|\chi|_{m-1}
\right)|R|_{-,\rho',m-1}  \\ 
&\le
2^{m^2+2m+10}m!G_m
\left(1+w_{m-1}^2\right)^m
(1+w_{m-1})e_{m-1} 
+
2^{m^2+m-2}m!G_{m-1}
\left(1+w_{m-1}^2\right)^m\Theta_{m-1} \\
&\le 2^{m^2+2m+10}m!G_m
\left(1+w_{m-1}^2\right)^m
(1+w_{m-1})e_{m-1} 
+ \frac{1}{64} G_m \Theta_m
.
\end{split}
\end{equation}

Combining the three estimates \eqref{eq:first-term}, \eqref{eq:second-term} and \eqref{eq:third-term}  in~\eqref{eq:chi-recursion} proves
\begin{equation*}
|\chi|_m
\le
\frac{14}{3}G_me_m
+
\left(
4G_m+\frac{280}{3}G_0
\right)w_me_0
+
\frac{1}{32}G_m\Theta_m,
\end{equation*}
which is~\eqref{eq:chi-split} and closes the induction. 

\smallskip\noindent\emph{(v) The product $H:=(g-\chi)R$.}
Applying~\eqref{eq:leibniz-split} once again to the  product $(g-\chi)R$ gives
\begin{equation*}
\begin{aligned}
|H|_{-,\rho^{\mathrm{mid}},m}
&\le
\left(
|g|_{-,\rho^{\mathrm{mid}},m}+|\chi|_m
\right)|R|_{-,\rho',0}\\
&\quad+
\left(
|g|_{-,\rho^{\mathrm{mid}},0}+|\chi|_0
\right)|R|_{-,\rho',m}\\
&\quad+
2^m\left(
|g|_{-,\rho^{\mathrm{mid}},m-1}+|\chi|_{m-1}
\right)|R|_{-,\rho',m-1}.
\end{aligned}
\end{equation*}
As in $(iv)$, we bound the three terms  successively. 
The first term is bounded using \eqref{eq:R-order0}, \eqref{eq:bound-g}, and
\eqref{eq:chi-split}: 
\begin{equation*}
\begin{split}
\left(
|g|_{-,\rho^{\mathrm{mid}},m}+|\chi|_m
\right)|R|_{-,\rho',0}  
&\le
2G_m\left(
\frac76e_m+w_me_0+2^mw_{m-1}e_{m-1}
\right) \\
&\quad + 
2\left[
\frac{14}{3}G_me_m
+
\left(
4G_m+\frac{280}{3}G_0
\right)w_me_0
+
\frac1{32}G_m\Theta_m
\right] \\
&\le 
\frac{35}{3}G_me_m
+
\left(
10G_m+\frac{560}{3}G_0
\right)w_me_0 
+
\frac1{8}G_m\Theta_m.
\end{split}
\end{equation*}
For the second term,  \eqref{eq:g-order0}, \eqref{eq:chi-order0}, and
\eqref{eq:R-split} provide
\begin{equation*}
\begin{split}
\left(
|g|_{-,\rho^{\mathrm{mid}},0}+|\chi|_0
\right)|R|_{-,\rho',m} 
&\le
5 G_0 e_0
\left[
\frac{28}{3}w_m
+
2^{m^2+2m+2}m!
\left(1+w_{m-1}^2\right)^m
\right] \\
\le 
\frac{140}{3}G_0w_me_0
+ \frac{1}{128} G_m \Theta_m. 
\end{split}
\end{equation*}
Finally, the third term is bounded by 
\begin{equation*}
\begin{split}
2^m\left(
|g|_{-,\rho^{\mathrm{mid}},m-1}+|\chi|_{m-1}
\right)|R|_{-,\rho',m-1}
&\le
2^{m^2+2m+9}m!G_m
(1+w_{m-1})^{2m+1}e_{m-1}\\
&\quad+
2^{m^2+m-3}m!G_{m-1}
\left(1+w_{m-1}^2\right)^m\Theta_{m-1} \\
&\le 2^{m^2+2m+9}m!G_m
(1+w_{m-1})^{2m+1}e_{m-1} + \frac{1}{128} G_m \Theta_m .
\end{split}
\end{equation*}
Consequently, summing the three above bounds provides 
\begin{equation}
\label{eq:H-split}
|H|_{-,\rho^{\mathrm{mid}},m}
\le
\frac{35}{3}G_me_m
+
\left(
10G_m+\frac{700}{3}G_0
\right)w_me_0
+  \frac{1}{4} G_m\Theta_m \; 
\end{equation}
and at order zero 
\begin{equation}\label{eq:H-split-0}
|H|_{-,\rho^{\mathrm{mid}},0}
\le
\left(
|g|_{-,\rho^{\mathrm{mid}},0}+|\chi|_0
\right)|R|_{-,\rho',0} \le
10G_0e_0.
\end{equation}

\smallskip\noindent\emph{(vi) Inverting the discrete operator  $S_0:=\nabla^{-1}H$.} 
By the definition of $\chi$, we obtain the zero mean normalization
\begin{equation*}
\langle H\rangle
=
\langle gR\rangle-\chi\langle R\rangle
=
0.
\end{equation*}
Since $ H \in C^-_{\rmid}$, the inverse  $S_0=\nabla^{-1}H $ is well defined  and lies in $C_{\rmid }^\pm$.   
Moreover we have the bound
\begin{equation*}
|S_0|_{\pm,\rho'',m}
 \le
G_m|H|_{-,\rho^{\mathrm{mid}},m} \le
\frac{35}{3}G_m^2e_m
+
\left(
10G_m^2+\frac{700}{3}G_mG_0
\right)w_me_0
+
\frac{1}{4} G_m^2\Theta_m.
\end{equation*}
The definition of $\xi_{m,\gamma,\tau}$ now gives the cruder bound
\begin{equation*}
|S_0|_{\pm,\rho'',m}
\le
\frac{\xi_{m,\gamma,\tau}}
{16d^{\tau_m}}
\left(e_m+w_me_0+\Theta_m\right).
\end{equation*}
At order zero, we similarly obtain
\begin{equation*}
|S_0|_{\pm,\rho'',0}
\le
10G_0^2e_0
\le
\frac{\xi_{0,\gamma,\tau}}
{16d^{\tau_0}}e_0.
\end{equation*}

\smallskip\noindent\emph{(vii) The solution
$S:=S_0-\langle AS_0\rangle$.}
Set $ c:=\langle AS_0\rangle$, so that $S=S_0-c$. For $m \ge 1$, one final application of  \eqref{eq:leibniz-split} gives
\begin{equation}
\label{eq:c-leibniz}
|c|_m
\le
\frac76|S_0|_{\pm,\rho'',m}
+
w_m|S_0|_{\pm,\rho'',0}
+
2^mw_{m-1}|S_0|_{\pm,\rho'',m-1}.
\end{equation}
Once again we estimate successively the three terms from left to right in
\eqref{eq:c-leibniz}. By $(vi)$, the first one satisfies
\begin{equation*}
\frac76|S_0|_{\pm,\rho'',m}
\leq
\frac{7\xi_{m,\gamma,\tau}}
{96d^{\tau_m}}
\left(
e_m+w_me_0+\Theta_m
\right) 
\end{equation*}
 and the second one verifies
\begin{equation*}
w_m|S_0|_{\pm,\rho'',0}
\leq
\frac{\xi_{0,\gamma,\tau}}
{16d^{\tau_0}}w_me_0 
\leq
\frac{\xi_{m,\gamma,\tau}}
{16d^{\tau_m}} w_me_0 \;.
\end{equation*}

For the last term it holds
\begin{equation*}
2^mw_{m-1}|S_0|_{\pm,\rho'',m-1}
\leq
\frac{\xi_{m,\gamma,\tau}}
{16d^{\tau_m}}
2^mw_{m-1}
\left(
e_{m-1}+w_{m-1}e_0+\Theta_{m-1}
\right) \leq
\frac{\xi_{m,\gamma,\tau}}
{16d^{\tau_m}}\Theta_m.
\end{equation*}
Combining these three estimates in~\eqref{eq:c-leibniz} yields
\begin{equation*}
|c|_m
\leq
\frac{\xi_{m,\gamma,\tau}}
{16d^{\tau_m}}
\left[
\frac76
\left(
e_m+w_me_0+\Theta_m
\right)
+w_me_0+\Theta_m
\right]\leq
\frac{3\xi_{m,\gamma,\tau}}
{16d^{\tau_m}}
\left(
e_m+w_me_0+\Theta_m
\right).
\end{equation*}
At order zero, the same conclusion follows directly from
\begin{equation*}
|c|_0
\leq
\frac76|S_0|_{\pm,\rho'',0}
<
\frac{3\xi_{0,\gamma,\tau}}
{16d^{\tau_0}}e_0.
\end{equation*}
Therefore, we conclude for the solution $S$  that
\begin{equation} \label{eq:S-final}
|S|_{\pm,\rho'',m}
\leq
|S_0|_{\pm,\rho'',m}+|c|_m \leq
\frac{\xi_{m,\gamma,\tau}}
{4 d^{\tau_m}}
\left(
e_m+w_me_0+\Theta_m
\right)\leq
\frac{\xi_{m,\gamma,\tau}}
{d^{\tau_m}}
\left(
e_m+w_me_0+\Theta_m
\right) 
\end{equation}
which is the asserted estimate on $S$.
 
  \smallskip\noindent\emph{(viii) The increment $h:=AS$.}
By construction $ \langle h \rangle = 0$.
Equation \eqref{eq:leibniz-split} gives
\begin{equation*}
\begin{split}
|h|_{\pm,\rho'',m}
&\leq
\frac76|S|_{\pm,\rho'',m}
+
w_m|S|_{\pm,\rho'',0}
+
2^mw_{m-1}|S|_{\pm,\rho'',m-1}.
\end{split}
\end{equation*}
By the penultimate inequality in \eqref{eq:S-final}, we obtain
\begin{equation*}
\begin{split}
|h|_{\pm,\rho'',m}
&\leq
\frac{\xi_{m,\gamma,\tau}}
{4 d^{\tau_m}}
\left[
\frac76
\left(
e_m+w_me_0+\Theta_m
\right)
+w_me_0+\Theta_m
\right] \\
&\leq
\frac{\xi_{m,\gamma,\tau}}
{d^{\tau_m}}
\left(
e_m+w_me_0+\Theta_m
\right).
\end{split}
\end{equation*}
At order zero, we similarly have
\begin{equation*}
|h|_{\pm,\rho'',0}
\leq
\frac76|S|_{\pm,\rho'',0}
\leq
\frac{7\xi_{0,\gamma,\tau}}
{24d^{\tau_0}}e_0
\leq
\frac{\xi_{0,\gamma,\tau}}
{d^{\tau_0}}e_0.
\end{equation*}
Finally, the local holomorphic extension just asserted is obtained as
follows. Replace $AE$ by $AE-\langle AE\rangle$ and normalize $S_0$ by
$S_0-\langle AS_0\rangle/\langle A\rangle$. The smallness of $A-1$
keeps all denominators nonzero, and holomorphy of all involved operators
follows from~\eqref{eq:leibniz-step},
Proposition~\ref{prop:Cpm-properties}(i),(iii),
Corollary~\ref{cor:nabla-inv-Cpm}, and
Lemmas~\ref{lem:Ainv-linear} and~\ref{lem:linear-postcomposition}.
This completes the proof.
\end{proof}

\subsection{Setting of the iteration scheme} \label{subsec:setting-it-scheme} 
We now fix some notation to run the iteration scheme and properly state estimates on the updated errors. 
Fix $\gamma,\tau>0$, $\rho>0$, and $ \rho_\infty \in(0,\rho)$.  Up to replacing $\rho_\infty$ we can assume without any loss that 
\[ \rho - \rho_\infty \le 1  . \]
Indeed taking a larger strip $ \rho_\infty$ only leads to stronger regularity in the KAM theorem. 

Choose $\rho_0 \in(\rho_\infty,\rho)$.  We will actually choose it close to $ \rho_\infty$ in \eqref{eq:star-standing} below. The successive strips in the iteration are chosen recursively by 
\begin{equation}
    \label{eq:strip-seq}
   \rho_{n+1} :=  \rho_{n} - d_n \quad \text{ with } d_n := \frac{6(\rho_0-\rho_\infty)}{\pi^2(n+1)^2} \;.
\end{equation}
Fix, in addition, two further reference strips and the corresponding gap,
\begin{equation}
    \label{eq:reference-strips}
    \rho^\star := \frac{\rho_0+\rho}{2}, \qquad
    \rho^\sharp := \frac{\rho_0+3\rho}{4} \qand 
    \bar d := \rho^\star - \rho_0 = \frac{\rho-\rho_0}{2} > 0.
\end{equation}
In particular, note that $0<\bar{d}< 1$.  These strips are presented in Figure~\ref{fig:strips}.

\begin{figure}[H]
    \centering
    \includegraphics[width=0.6\linewidth]{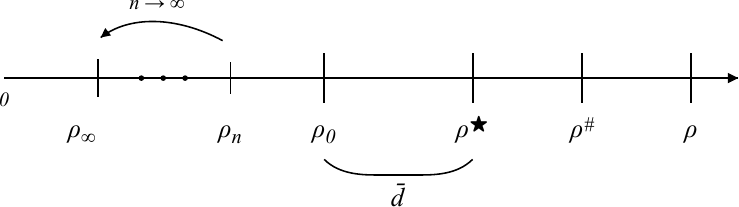}
    \caption{Hierarchy of the analytic strips in the iteration scheme.}
    \label{fig:strips}
\end{figure}
Note that every
strip $\rho_n$ of~\eqref{eq:strip-seq} satisfies $\rho_n \le \rho_0 <\rho^\star$. Thus
we have $\rho^\star - \rho_n \ge \rho^\star-\rho_0 = \bar d$ for every $n\ge0$.

Recall that $ \xi_{0,\gamma,\tau} $ is provided by Proposition~\ref{prop:Cpm-properties}.
Set, once and for all,
\begin{equation} \label{eq:xi-err-0}
 \xi^{\mathrm{err}}_{0,\gamma,\tau}
 :=\xi_{0,\gamma,\tau}
 +\frac{1050\,\xi_{0,\gamma,\tau}^{2}}{\bar d^3}.
\end{equation}

Set
\begin{equation} \label{eq:eps-0}
  \varepsilon_0:=\frac{\bar d}{18\,\xi_{0,\gamma,\tau}}.
\end{equation}

 We then choose $\rho_0$ sufficiently close to $\rho_\infty$ such that
\begin{equation}
    \label{eq:star-standing}
    (\rho_0-\rho_\infty)^{\tau_0}
    \le 2^{4\tau_0}\,\xi^{\mathrm{err}}_{0,\gamma,\tau}\,\varepsilon_0\;.
\end{equation}
This will guarantee the smallness hypothesis~\eqref{eq:E-small-quad} of
Proposition~\ref{prop:quadratic} at every step of the induction allowing quadratically small error in the order zero seminorms. 

We finally set
\begin{equation} \label{eq:def-eps-r-U}
\varepsilon_{\mathcal U}
:=
\left(
8+
\frac{2^5\xi_{0,\gamma,\tau}
(\rho_0-\rho_\infty)^{\tau_0-1}}
{2^{3\tau_0}\xi^{\mathrm{err}}_{0,\gamma,\tau}}
\right)^{-1} \qand 
r_{\mathcal U}
:=
\frac{
\varepsilon_{\mathcal U}d_0^{2\tau_0}
(\rho^\sharp-\rho_0)
}{
2^{4\tau_0}\xi^{\mathrm{err}}_{0,\gamma,\tau}
+
\varepsilon_{\mathcal U}d_0^{2\tau_0}
(\rho^\sharp-\rho_0)
} \;. 
\end{equation}
and define the open convex KAM neighborhood

\begin{equation}
\label{eq:U-def}
\mathcal U
:=
\left\{
V\in\Hol_\rho(\T):
|V|_{\rho^\sharp}<r_{\mathcal U}
\right\}.
\end{equation}
The constant $ \varepsilon_{\mathcal{U}}$ is chosen to satisfy the  inequalities
\begin{equation}
    \label{eq:U-smallness}
\varepsilon_{\mathcal U}\le\frac18,
\qquad
\varepsilon_{\mathcal U}
\le
\frac{2^{3\tau_0-5}\xi^{\mathrm{err}}_{0,\gamma,\tau}}
{\xi_{0,\gamma,\tau}
(\rho_0-\rho_\infty)^{\tau_0-1}}.
\end{equation}
This is the neighbourhood $\mathcal U$ appearing in Theorem~\ref{thm:hol-KAM}.

For $V\in\mathcal U$, when possible, define recursively
\begin{equation} \label{eq:def-incr} 
\begin{split}
   & u_0:=0,\qquad A_n:=1+\partial_\theta u_n,\qquad
    E_n:=\mathcal E(u_n,V):=V'(\id_\T+u_n)-\Delta u_n, \\ 
   & S_n:=\mathcal{S}(A_n,E_n),\qquad
    h_n:=A_nS_n \qand u_{n+1}:=u_n+h_n.
    \end{split}
\end{equation}
Actually, Proposition~\ref{prop:induction} will show that they all are indeed well defined.

Since the definition of the error $E_n$ involves the map
$(u,V)\mapsto V'(\id_\T+u)-\Delta u$, we need to control its norm. This is
done using the following lemma, proved in the $C^1$-holomorphic case
in~\cite[Lemma~10]{CMS}.
\begin{lemma}
\label{lem:composition-q-deriv}
Let $\rho>\rho'>\rho''>0$ and let $\widetilde\rho>\rho''$. Let
$f\in\mathscr H_\rho(\T)$ and let
$\phi\in\mathscr C_{\widetilde\rho}$ such that for some $\kappa\in(0,1)$ it holds
\begin{equation}
    \label{eq:phi-small-order-zero}
    |\phi|_{\rho'',0} \le \kappa\,(\rho'-\rho'').
\end{equation}
Then $f\circ(\id_\T+\phi)\in\mathscr C_{\rho''}$ and for every $m \ge  0$, it verifies
\begin{equation}
    \label{eq:comp-bound-order-zero}
    \left|f\circ(\id_\T+\phi)\right|_{\rho'',m}
    \le 2^m\,m!\,
    \max\!\left(1,\;\frac{|\phi|_{\rho'',m}}{(1-\kappa)(\rho'-\rho'')}\right)^{\!m}
    |f|_{\rho'}.
\end{equation}
In particular the map 
\begin{equation*}
\begin{array}{rcl}
\mathscr H_\rho(\T)\times
\left\{
\phi\in\mathscr C_{\widetilde\rho}:
|\phi|_{\rho'',0}<\kappa(\rho'-\rho'')
\right\}
&\longrightarrow&
\mathscr C_{\rho''}
\\
(f,\phi)
&\longmapsto&
f\circ(\id_\T+\phi).
\end{array}
\end{equation*}
is holomorphic.

Furthermore, for $m\geq1$, one also has the sharper graded bounds
\begin{equation}
\label{eq:comp-bound-tame}
 \left|f\circ(\id_\T+\phi)\right|_{\rho'',m}
 \le |f|_{\rho'}\left[
 \frac{|\phi|_{\rho'',m}}{(1-\kappa)(\rho'-\rho'')}
 +2^m m!\max\!\left(
 1,\frac{|\phi|_{\rho'',m-1}}
 {(1-\kappa)(\rho'-\rho'')}
 \right)^m\right].
\end{equation}

\end{lemma}

\begin{proof} 
Set the radius
\begin{equation*}
r_\kappa:=(1-\kappa)(\rho'-\rho'').
\end{equation*}
For every $\theta\in\T_{\rho''}$ and $q\in\KK$,
assumption~\eqref{eq:phi-small-order-zero} gives
\begin{equation*}
\left|\Im\left(\theta+\phi(q,\theta)\right)\right|
\leq
\rho''+\kappa(\rho'-\rho'')
=
\rho'-r_\kappa.
\end{equation*}
Consequently, for every $0<r<r_\kappa$, the closed disc of radius
$r$ centred at
\begin{equation*}
z:=\theta+\phi(q,\theta)
\end{equation*}
is contained in $\T_{\rho'}$. Cauchy's formula gives, for every
$p\geq0$,
\begin{equation*}
f^{(p)}(z)
=
\frac{p!}{2\pi i}
\int_{|\zeta-z|=r}
\frac{f(\zeta)}{(\zeta-z)^{p+1}}\,d\zeta.
\end{equation*}
Therefore we deduce the bound 
\begin{equation*}
\left|f^{(p)}(z)\right|
 \leq
\frac{p!}{2\pi}
\int_{|\zeta-z|=r}
\frac{|f(\zeta)|}{|\zeta-z|^{p+1}}\,|d\zeta|
\leq
\frac{p!}{2\pi}
\frac{|f|_{\rho'}}{r^{p+1}}(2\pi r)
=
p!r^{-p}|f|_{\rho'}.
\end{equation*}
Letting $r$ tend to $r_\kappa$ yields
\begin{equation}
\label{eq:cauchy-p}
\left|
f^{(p)}\left(\theta+\phi(q,\theta)\right)
\right|
\leq
p!r_\kappa^{-p}|f|_{\rho'}.
\end{equation}
Note that at $ p= 0$, this already gives \eqref{eq:comp-bound-order-zero} in the  $m = 0$ case. We now prove the two estimates for $m\geq1$.

Let us write  $G:=f\circ(\id_\T+\phi)$. The $m$-derivative is given by 
\begin{equation*}
G^{(m)}(q)(\theta)
=
\sum_{\mathcal P\vdash\left\{1,\ldots,m\right\}}
f^{(|\mathcal P|)}
\left(\theta+\phi(q,\theta)\right)
\prod_{A\in\mathcal P}
\phi^{(|A|)}(q)(\theta).
\end{equation*}
where the sum is over set partitions $ \mathcal{P}$ of $ \left\{ 1 , \dots , m  \right\}$.  We then bound the different terms of the sum separately.
For every partition $\mathcal P$ containing 
$p:=|\mathcal P|$ subsets, estimate~\eqref{eq:cauchy-p} gives
\begin{equation*}
\left|
f^{(|\mathcal P|)}
\circ(\id_\T+\phi)
\prod_{A\in\mathcal P}\phi^{(|A|)}
\right|_{\rho''}
\leq
p!|f|_{\rho'}
\left(
\frac{|\phi|_{\rho'',m}}{r_\kappa}
\right)^p 
\end{equation*}
Since the number of partitions with $p$ subsets is $\stirling{m}{p}$,
Lemma~\ref{lem:fubini-bound} gives
\begin{equation*}
|G^{(m)}(q)|_{\rho''}
\leq
2^m m!\max\left\{
1,\frac{|\phi|_{\rho'',m}}{r_\kappa}
\right\}^m|f|_{\rho'}.
\end{equation*}
The same estimate at every order $0\leq\ell\leq m$, together with the monotonicity of the seminorms and of the right-hand side of the above inequality, gives \eqref{eq:comp-bound-order-zero}.

The term of the derivative containing $\phi^{(m)}$ is
$\left[f'\circ(\id_\T+\phi)\right]\phi^{(m)}$ and satisfies
\begin{equation*}
\left|
\left[f'\circ(\id_\T+\phi)\right]\phi^{(m)}
\right|_{\rho''}
\leq
|f|_{\rho'}\frac{|\phi|_{\rho'',m}}{r_\kappa}.
\end{equation*}
Every remaining term of the derivative contains only derivatives of $\phi$ of order at most $m-1$ and their sum is bounded by
\begin{equation*}
\begin{aligned}
&\left|
\sum_{\substack{
\mathcal P\vdash\left\{1,\ldots,m\right\}\\
|\mathcal P|\geq2
}}
f^{(|\mathcal P|)}
\circ(\id_\T+\phi)
\prod_{A\in\mathcal P}\phi^{(|A|)}
\right|_{\rho''}\\
&\leq
|f|_{\rho'}
\sum_{p=2}^{m}
\stirling{m}{p}p!
\left(
\frac{|\phi|_{\rho'',m-1}}{r_\kappa}
\right)^p\\
&\leq
2^m m!|f|_{\rho'}
\max\left\{
1,
\frac{|\phi|_{\rho'',m-1}}{r_\kappa}
\right\}^{m}.
\end{aligned}
\end{equation*}
It follows that 
\begin{equation*}
|G^{(m)}(q)|_{\rho''}
\leq
|f|_{\rho'}
\left[
\frac{|\phi|_{\rho'',m}}{r_\kappa}
+
2^m m!
\max\left\{
1,
\frac{|\phi|_{\rho'',m-1}}{r_\kappa}
\right\}^{m}
\right].
\end{equation*}
The same estimate bounds all derivatives of order at most $m-1$. 
Taking the maximum over all derivatives of order at most $m$ therefore gives \eqref{eq:comp-bound-tame}.

The derivative formulas are precisely the Whitney derivatives, hence $ G\in\mathscr C_{\rho''}$.
The argument proving~\eqref{eq:comp-bound-order-zero} also proves
holomorphy. Indeed, along every  intersection of a complex affine line with its domain, the
composition is pointwise holomorphic, while
\eqref{eq:comp-bound-order-zero} gives locally uniform bounds in every
seminorm of $\mathscr C_{\rho''}$. Hence the composition map is
holomorphic with values in $\mathscr C_{\rho''}$.  The strip
$\widetilde\rho>\rho''$ on which $\phi$ is defined is introduced only so that the condition 
\begin{equation*}
|\phi|_{\rho'',0}<\kappa(\rho'-\rho'')
\end{equation*}
defines an open subset of that space. 
\end{proof}

To estimate the new error, we use the following exact identity.

\begin{lemma}
\label{lem:exact-identity}
Let $0<\rho''<\rho'<\rho$ satisfy
$\rho'-\rho''\leq1$. Let $V\in\Hol_\rho(\T)$ and
$u\in C_{\rho'}^{\pm}$ satisfy
\begin{equation*}
|u|_{\rho',0}<\rho-\rho',
\qquad
|\partial_\theta u|_{\pm,\rho',0}<\frac18.
\end{equation*}
Set $  A:=1+\partial_\theta u$, $ E:=\mathcal{E}(u,V)= V'(\id_\T+u)-\Delta u$, and let $  S:=\mathcal S(A,E) \in \mathscr{C}_{\rho'}^{\pm}$ be the normalized solution provided by
Theorem~\ref{thm:solution-hol}. Assume moreover that
\begin{equation} \label{eq:strong-for-id}
|u |_{\rho'',0} + |AS|_{\rho'',0} < \rho-\rho''.
\end{equation}
Then for 
$\widetilde u:=u+AS$, the updated error is given by the identity
\begin{equation}
\label{eq:exact-identity}
\mathcal E(\widetilde u,V)
=
S\cdot\partial_\theta E
+
(AS)^2
\int_0^1
(1-t)V'''(\id_\T+u+tAS)\,dt
\end{equation}
in $\mathscr C_{\rho''}$.
\end{lemma}

\begin{proof}
First of all, applying Lemma~\ref{lem:composition-q-deriv} on every strip strictly smaller than $\rho'$ gives
$E\in\mathscr C_{\rho'}$, while
Fact~\ref{fact:zero-mean} gives $\langle AE\rangle=0$.
Thus Theorem~\ref{thm:solution-hol} applies and defines $S$.
The hypothesis  \eqref{eq:strong-for-id} allows us to apply Lemma~\ref{lem:composition-q-deriv} and obtain that $ V''' ( \id_\T + u + t AS )$ lies in 
$\mathscr C_{\rho''}$ for every $t \in [0,1]$. 
 Taylor's formula with integral remainder at order $2$ gives
\begin{equation*}
    V'(\id_\T+u+AS) = V'(\id_\T+u) + V''(\id_\T+u)\cdot AS
    + (AS)^2 \int_0^1(1-t)V'''(\id_\T+u+tAS)\,dt.
\end{equation*}
Using linearity of $\Delta$ and $E=V'(\id_\T+u) - \Delta u$, the above equation leads to 
\begin{equation}
    \label{eq:Etilde-expand}
    \widetilde E = E + \left[V''(\id_\T+u)\cdot AS -\Delta(AS)\right]
    + (AS)^2 \int_0^1(1-t)V'''(\id_\T+u+tAS)\,dt.
\end{equation}
By definition, $S$ solves $[A,AS]_\Delta=AE$, hence $\Delta(AS)=E+S \Delta A$ and substituting into~\eqref{eq:Etilde-expand} provides
\begin{equation*}
    \widetilde E = S \cdot \left[V''(\id_\T+u)\cdot A - \Delta A\right]
    + (AS)^2 \int_0^1(1-t)V'''(\id_\T+u+tAS)\,dt.
\end{equation*}
Using the chain rule and commutativity of $ \Delta$ with $ \partial_\theta$, one obtains 
\begin{equation*}
    V''(\id_\T+u)\cdot A - \Delta A	
    = \partial_\theta\left[V'(\id_\T+u)\right] - \partial_\theta(\Delta u)
    = \partial_\theta\left[V'(\id_\T+u) -\Delta u\right] = \partial_\theta E,
\end{equation*}
which gives~\eqref{eq:exact-identity} and ends the proof.
\end{proof}

\subsection{Convergence of the Levi--Moser iteration scheme at
order zero} \label{subsec:zero-scheme}
A quadratic estimate at arbitrary order would require smallness conditions
on the potential making the neighborhood of convergence of the iteration scheme shrink and preventing us from proving holomorphy in the potential.
We  overcome this difficulty by proceeding by induction on the order of the derivative in the $q \in \KK$ variable.
 
The first step consists in the initialization, that is, proving the Levi--Moser iteration scheme in the family of seminorms $( |\cdot|_{ \rho, 0})_{ \rho >0 } $, and this uniformly for every potential in $ \mathcal{U}$. 
The first step below gives the quadratic smallness of the updated error. 
\begin{proposition}
\label{prop:quadratic}
Let $0<\rho''<\rho'\le\rho_0$ with $\rho'-\rho''\le1$. Let
$V\in\Hol_\rho(\T)$ and $u\in C_{\rho'}^{\pm}$ such that
\begin{equation*}
    |\partial_\theta u|_{\pm,\rho',0} < \frac{1}{8}, \qquad
    |u|_{\pm,\rho',0} \le \frac{\bar{d}}{9} \;. 
\end{equation*}
Assume that  $E:=\mathcal E(u,V)=V'(\id_\T+u)-\Delta u$ satisfies
\begin{equation}
    \label{eq:E-small-quad}
    |E|_{\rho',0} \le \varepsilon_0(\rho'-\rho'')^{\tau_0}.
\end{equation}
Then  
$$\widetilde E:=\mathcal{E} \Big( u + (1 + \partial_\theta u) \mathcal{S} \left( 1 + \partial_\theta u , E ) \right)   ,V \Big) $$  is bounded by
\begin{equation*}
    |\widetilde E|_{\rho'',0} \le
    \frac{\xi^{\mathrm{err}}_{0,\gamma,\tau}\,(1+|V|_{\rho^\sharp})}
    {(\rho'-\rho'')^{2\tau_0}}\,|E|_{\rho',0}^2,
\end{equation*}
where $\xi^{\mathrm{err}}_{0,\gamma,\tau}$ is given in \eqref{eq:xi-err-0}. 
\end{proposition}

\begin{proof}
As in the preceding results, we use the notations $ A = 1 + \partial_\theta u$, $ S = \mathcal{S}(A,E)$ and $ h = AS$. 
The hypotheses of the proposition ensure that we can apply 
Theorem~\ref{thm:solution-hol}, at $m=0$,  giving that the increment $h$ is bounded by 
\begin{equation}
\label{eq:bound-AS-0}
    |h|_{\pm,\rho'',0} \le
    \frac{\,\xi_{0,\gamma,\tau}}{(\rho'-\rho'')^{\tau_0}}\,
    |E|_{\rho',0}.
\end{equation}
By~\eqref{eq:E-small-quad}  and since $\varepsilon_0 = \frac{ \bar{d}}{ 18 \xi_{ 0, \gamma,\tau} }$, this
implies $|h|_{\pm,\rho'',0}\le\bar{d}/18$. Therefore we have 
\begin{equation} \label{eq:hyp-star}
    |u|_{\pm,\rho'',0} +  |h|_{\pm,\rho'',0}
    \le \frac{\bar d}{9}+\frac{\bar d}{18}
    =\frac{\bar d}{6}
    \le \frac{\rho^\star-\rho''}{6} \;. 
\end{equation}
This allows us to apply Lemma~\ref{lem:exact-identity}  and write 
\begin{equation*}
    \widetilde E=T_1+T_2,
\end{equation*}
where 
\begin{equation*}
    T_1 := S\cdot\partial_\theta E \qand T_2 := (AS)^2\int_0^1(1-t)\,V'''(\id_\T+u+tAS)\,dt.
\end{equation*} 
We begin by bounding $T_1$.
Lemma~\ref{lem:der-hol} provides
$|\partial_\theta E|_{\rho'',0}\le(\rho'-\rho'')^{-1}|E|_{\rho',0}$. 
Using Theorem~\ref{thm:solution-hol} at $m=0$ and the product estimate~\eqref{eq:leibniz-step}, we therefore obtain 
\begin{equation}
    \label{eq:term1-final}
    |T_1|_{\rho'',0} \le
    \frac{\xi_{0,\gamma,\tau}}{(\rho'-\rho'')^{\tau_0+1}}\,|E|_{\rho',0}^2.
\end{equation}

We now bound the integral term $T_2$. First,
Lemma~\ref{lem:der-hol}, applied with $p=3$ and
$\rho^\sharp-\rho^\star=\bar{d}/2$, gives
\begin{equation*}
    |V'''|_{\rho^\star} \le \frac{3!}{(\bar{d}/2)^3}\,|V|_{\rho^\sharp}
    = \frac{48}{\bar{d}^3}\,|V|_{\rho^\sharp}.
\end{equation*}
Then  \eqref{eq:hyp-star} allows us to apply Lemma~\ref{lem:composition-q-deriv} for every $t\in[0,1]$ at $m=0$ with $f=V'''$, $\phi=u+tAS$ , using  $\rho^\star$ for  $ V'''$ and the present $\rho'$ for the strip of $\phi$, to obtain 
\begin{equation*}
    \sup_{t\in[0,1]}|V'''(\id_\T+u+tAS)|_{\rho'',0}
    \le |V'''|_{\rho^\star} \le \frac{48}{\bar{d}^3}\,|V|_{\rho^\sharp}.
\end{equation*}
Combining this estimate with~\eqref{eq:bound-AS-0} and~\eqref{eq:leibniz-step} gives
\begin{equation} \label{eq:term2-final}
    |T_2|_{\rho'',0} \le \frac{1050\,\xi_{0,\gamma,\tau}^{\,2}}{\bar d^3(\rho'-\rho'')^{2\tau_0}}\,
    |E|^2_{\rho',0} |V|_{\rho^\sharp}  \;.   
\end{equation}

Since $\rho'-\rho''\le1$ and $2\tau_0\ge\tau_0+1$, the bounds~\eqref{eq:term1-final} and~\eqref{eq:term2-final} give
\begin{align*}
    | \widetilde E |_{\rho'',0 } &\le  | T_1 |_{\rho'',0 } + | T_2 |_{\rho'',0 } \\
     &\le
    \frac{\xi_{0,\gamma,\tau}}{(\rho'-\rho'')^{\tau_0+1}}\,|E|_{\rho',0}^2 + \frac{1050\,\xi_{0,\gamma,\tau}^{\,2}}{\bar d^3(\rho'-\rho'')^{2\tau_0}}\,
    |E|^2_{\rho',0} |V|_{\rho^\sharp}  \\  
    &\le \xi^{\mathrm{err}}_{0,\gamma,\tau}\,\frac{1+|V|_{\rho^\sharp}}{(\rho'-\rho'')^{2\tau_0}}\,
    |E|^2_{\rho',0}  \;.  
\end{align*}
\end{proof}

We are now ready to run the Levi--Moser iteration scheme at order $0$ in the $q$-variable.

\smallskip
Fix $V\in\mathcal U$. With the notations from Section~\ref{subsec:setting-it-scheme}, set
\begin{equation}\label{eq:def-alpha-delta}
  \delta_n:=|E_n|_{\rho_n,0},
  \qquad
  \alpha_n := \alpha_n (V) := 
  \frac{2^{4\tau_0}\,\xi^{\mathrm{err}}_{0,\gamma,\tau}
  \left(1+|V|_{\rho^\sharp}\right)}{d_n^{2\tau_0}}\,\delta_n.
\end{equation}

Let us observe the following.

\begin{fact}
\label{fact:U-admissible}
It holds $ \alpha_0 <\varepsilon_{\mathcal U} < 1/8$.
\end{fact}

\begin{proof}
For $V\in\mathcal U$, Lemma~\ref{lem:der-hol} gives
\begin{equation*}
|V'|_{\rho_0}
\le
\frac{|V|_{\rho^\sharp}}{\rho^\sharp-\rho_0}.
\end{equation*}
By definition \eqref{eq:def-eps-r-U}, one has  $0<r_{\mathcal U}<1$ and
\begin{equation*}
\frac{r_{\mathcal U}}{1-r_{\mathcal U}}
=
\frac{\varepsilon_{\mathcal U}d_0^{2\tau_0}
(\rho^\sharp-\rho_0)}
{2^{4\tau_0}\xi^{\mathrm{err}}_{0,\gamma,\tau}}.
\end{equation*}
It follows that 
\begin{equation*}
\alpha_0 =
\frac{
2^{4\tau_0}\xi^{\mathrm{err}}_{0,\gamma,\tau}
\left(1+|V|_{\rho^\sharp}\right)
}{
d_0^{2\tau_0}
}
|V'|_{\rho_0} <
\frac{
2^{4\tau_0}\xi^{\mathrm{err}}_{0,\gamma,\tau}
}{
d_0^{2\tau_0}
}
\frac{(1+r_{\mathcal U})r_{\mathcal U}}
{\rho^\sharp-\rho_0}
<
\frac{2^{4\tau_0}\xi^{\mathrm{err}}_{0,\gamma,\tau}}
{d_0^{2\tau_0}(\rho^\sharp-\rho_0)}
\frac{r_{\mathcal U}}{1-r_{\mathcal U}}
=
\varepsilon_{\mathcal U}.
\end{equation*}
The second bound $ \varepsilon_{\mathcal U} < 1/8$ was provided by definition. 
\end{proof}

We show superexponential decay of $ (\alpha_n)_n$ in the following. 

\begin{proposition}
\label{prop:induction}
For every $n\ge0$:
\begin{itemize}
    \item[(I$_n$)] $\alpha_n\le\alpha_0^{2^n}$,
    \item[(II$_n$)] $u_n\in C_{\rho_n}^{\pm}$ with
    $|\partial_\theta u_n|_{\pm,\rho_n,0}\le\frac18(1-2^{-n})$ and
    $|u_n|_{\pm,\rho_n,0}\le\frac{\bar d}{9}(1-2^{-n})$.
\end{itemize}
\end{proposition}

\begin{proof}
We proceed by induction on $n \ge 0$. 
For the base case $n=0$, (I$_0$) is trivial, and (II$_0$) holds for $u_0=0$. Assume (I$_n$) and (II$_n$) hold for $n\ge0$. Since
$d_n\le\rho_0-\rho_\infty$,  equation~\eqref{eq:star-standing} and (I$_n$) give
\begin{equation*}
    \delta_n
    \le \frac{d_n^{2\tau_0}}
    {2^{4\tau_0}\xi^{\mathrm{err}}_{0,\gamma,\tau}
    (1+|V|_{\rho^\sharp})}\,\alpha_0^{2^n}
    \le \varepsilon_0\,d_n^{\tau_0}.
\end{equation*}
Together with $|u_n|_{\pm,\rho_n,0}\le\bar d/9$ from (II$_n$),
the inequality $\rho^\star-\rho_n\geq\bar d$ allows us to apply
Lemma~\ref{lem:composition-q-deriv} to provide the required regularity of $V'\circ(\id_\T+u_n)$. Hence
Proposition~\ref{prop:quadratic} applies at $\rho'=\rho_n$
and $\rho''=\rho_{n+1}$, giving
$S_n=\mathcal{S}(A_n,E_n)\in\mathscr C_{\rho_{n+1}}^{\pm}$ and
\begin{equation*}
 \delta_{n+1}
 \le 
 \frac{\xi^{\mathrm{err}}_{0,\gamma,\tau}(1+|V|_{\rho^\sharp})}
 {d_n^{2\tau_0}}\,\delta_n^2.
\end{equation*}
It follows by definition of $\alpha_{n+1}$ that
\begin{equation*}
 \alpha_{n+1} 
 \le 
 \frac{2^{4\tau_0}\xi^{\mathrm{err}}_{0,\gamma,\tau}
 (1+|V|_{\rho^\sharp})}{d_{n+1}^{2\tau_0}}
 \frac{\xi^{\mathrm{err}}_{0,\gamma,\tau}
 (1+|V|_{\rho^\sharp})}{d_n^{2\tau_0}}\,\delta_n^2 
 = \frac{1}{2^{4\tau_0}}
 \left(\frac{d_n}{d_{n+1}}\right)^{2\tau_0}\alpha_n^2.
 \end{equation*}
The explicit choice~\eqref{eq:strip-seq} gives
\begin{equation*}
 \frac{d_n}{d_{n+1}}
 =\left(\frac{n+2}{n+1}\right)^2\le 4.
\end{equation*}
Consequently, the factor preceding $\alpha_n^2$ is at most one, and
(I$_n$) yields
\begin{equation*}
 \alpha_{n+1}\le \alpha_n^2
 \le \alpha_0^{2^{n+1}},
\end{equation*}
which proves (I$_{n+1}$).

To prove (II$_{n+1}$), we first bound $h_n$. Since
$|\partial_\theta u_n|_{\pm,\rho_n,0}<1/8$ by (II$_n$),
Theorem~\ref{thm:solution-hol}, applied from $\rho_n$ to
$\rho_{n+1}$, gives
\begin{equation*}
|h_n|_{\pm,\rho_{n+1},0}
\le
\frac{\xi_{0,\gamma,\tau}}{d_n^{\tau_0}}\delta_n
\le
\frac{\xi_{0,\gamma,\tau}d_n^{\tau_0}}
{2^{4\tau_0}\xi^{\mathrm{err}}_{0,\gamma,\tau}}
\alpha_n.
\end{equation*}
Since $d_n\le\rho_0-\rho_\infty$,
\eqref{eq:star-standing} and the definition of $\varepsilon_0$ yield
\begin{equation}
\label{eq:bound-h-n}
|h_n|_{\pm,\rho_{n+1},0}
\le
\frac{\bar d}{18}\alpha_n
\le
\frac{\bar d}{18}\alpha_0^{2^n}
\le
\frac{\bar d}{18}\,2^{-n}.
\end{equation}

To bound the derivative of $h_n$, consider the intermediate
strip
\begin{equation*}
\rho_n^{\mathrm{mid}}
:=
\frac{\rho_n+\rho_{n+1}}{2},
\qquad
\rho_n-\rho_n^{\mathrm{mid}}
=
\rho_n^{\mathrm{mid}}-\rho_{n+1}
=
\frac{d_n}{2}.
\end{equation*}
Applying Theorem~\ref{thm:solution-hol} from $\rho_n$ to
$\rho_n^{\mathrm{mid}}$ gives
\begin{equation}
\label{eq:bound-h-n-mid}
|h_n|_{\pm,\rho_n^{\mathrm{mid}},0}
\le
\frac{2^{\tau_0}\xi_{0,\gamma,\tau}}{d_n^{\tau_0}}\delta_n
\le
\frac{\xi_{0,\gamma,\tau}d_n^{\tau_0}}
{2^{3\tau_0}\xi^{\mathrm{err}}_{0,\gamma,\tau}}
\alpha_n.
\end{equation}
Applying Lemma~\ref{lem:der-hol} across the remaining strip gap $d_n/2$, we obtain that
\begin{equation}
\label{eq:bound-dh-n}
|\partial_\theta h_n|_{\pm,\rho_{n+1},0}
 \le
\frac{2}{d_n}|h_n|_{\pm,\rho_n^{\mathrm{mid}},0} \le
\frac{\xi_{0,\gamma,\tau}d_n^{\tau_0-1}}
{2^{3\tau_0-1}\xi^{\mathrm{err}}_{0,\gamma,\tau}}
\alpha_n \le
\frac{\xi_{0,\gamma,\tau}
(\rho_0-\rho_\infty)^{\tau_0-1}}
{2^{3\tau_0-1}\xi^{\mathrm{err}}_{0,\gamma,\tau}}
\alpha_0^{2^n}.
\end{equation}
The definition of $\varepsilon_{\mathcal U}$ in \eqref{eq:def-eps-r-U} and
$\alpha_0<\varepsilon_{\mathcal U}$ give
\begin{equation*}
\frac{\xi_{0,\gamma,\tau}
(\rho_0-\rho_\infty)^{\tau_0-1}}
{2^{3\tau_0-1}\xi^{\mathrm{err}}_{0,\gamma,\tau}}
\alpha_0
<
\frac{1}{16}.
\end{equation*}
Consequently,
\begin{equation}
\label{eq:bound-dh-n-final}
|\partial_\theta h_n|_{\pm,\rho_{n+1},0}
\le
\frac{1}{16}\,2^{-n}.
\end{equation}

Since $u_{n+1}=u_n+h_n$, we have
$u_{n+1}\in C_{\rho_{n+1}}^\pm$, and (II$_n$),
\eqref{eq:bound-h-n}, and~\eqref{eq:bound-dh-n-final} give
\begin{align*}
|u_{n+1}|_{\pm,\rho_{n+1},0}
&\le
\frac{\bar d}{9}\left(1-2^{-n}\right)
+\frac{\bar d}{18}\,2^{-n}
=
\frac{\bar d}{9}\left(1-2^{-(n+1)}\right),\\
|\partial_\theta u_{n+1}|_{\pm,\rho_{n+1},0}
&\le
\frac18\left(1-2^{-n}\right)
+\frac1{16}\,2^{-n}
=
\frac18\left(1-2^{-(n+1)}\right).
\end{align*}
This proves (II$_{n+1}$). We end the proof by running the induction.
\end{proof}
\subsection{Convergence of the scheme at every order}
\label{subsec:convergence-every-order}
The purpose here is to prove convergence of the scheme at every order. 
As at order zero, we bound the updated error term at order $m$, but this time linearly in the previous error at order $m$ and polynomially in the error at order $m-1$. 
 For $m\ge 0$ and $0<\rho''<\rho'\leq\rho_0$ with $\rho'-\rho''\leq1$, consider
$V\in\Hol_\rho(\T)$, and $u\in C_{\tilde{\rho}}^{\pm}$ for some $ \tilde{\rho} > \rho'$. Denote 
$A:=1+\partial_\theta u\in\mathscr C_{\tilde{\rho} }^{\pm}$,   $ 
 E:=\mathcal E(u,V)$, $   S:=\mathcal S(A,E),$ and $  h:=AS$ when they are well defined. 
 Then we denote 
\begin{equation}
 \label{eq:def-bm}
 b_{m}:=1+|E|_{\rho',0}+|u|_{\pm,\rho',m}
 +|\partial_\theta u|_{\pm,\rho',m}
 +|h|_{\pm,\rho'',m}.
\end{equation}
We introduce the explicit constants
\begin{equation} \label{eq:def-vartheta} 
\vartheta_0:=0,
\qquad
\vartheta_j
:=
2^{j^3+3j^2+9j+9}(j!)^{j+3} 
\quad \text{for }  j\geq1 ,  
\end{equation}

\begin{equation} \label{eq:xi-def}
\xi^{\mathrm{sol}}_{m,\gamma,\tau}
:=
8\max\left\{
\begin{gathered}
\xi_{0,\gamma,\tau},\\
\xi_{m-1,\gamma,\tau}\left(2+\vartheta_{m-1}\right),\\
\xi_{m,\gamma,\tau}\left(2+2^{2m}\vartheta_m\right)
\end{gathered}
\right\}, 
\end{equation}
with the convention  $\xi_{-1,\gamma,\tau}  = 0$; and
\begin{equation}\label{eq:xi-def-2}
\xi^{\mathrm{pot}}_{m,\bar d}
:=
\max_{0\leq r\leq m} (r+3)! \left( \frac{2}{\bar{d}} \right)^{r+3}
\qand 
\xi^{\mathrm{err}}_{m,\gamma,\tau,\bar d}
:=
2^{m+6}3^m m^m\,
\xi^{\mathrm{pot}}_{m,\bar d}
\left(1+\xi^{\mathrm{sol}}_{m,\gamma,\tau}\right)^3.
\end{equation}

Under the same assumptions as in Proposition~\ref{prop:quadratic}, we obtain the following bounds on the update. 
\begin{theorem}
\label{thm:error-order-m-new}
With the above notations and assumptions, assume moreover that 
\begin{equation*}
  |\partial_\theta u|_{\pm,\rho',0}<\frac18,
  \qquad |u|_{\pm,\rho',0}\le\frac{\bar d}{9} 
\end{equation*}
and
\begin{equation*}
 |E|_{\rho',0}\le\varepsilon_0(\rho'-\rho'')^{\tau_0}.
\end{equation*}
For $m \ge 1$ the updated error then verifies 
\begin{equation}
  \label{eq:error-order-m-new}
 \begin{aligned}
 & |\mathcal E(u+h,V)|_{\rho'',m} \\
&  \le 
  \frac{\xi^{\mathrm{err}}_{m,\gamma,\tau,\bar d}
  (1+|V|_{\rho^\sharp})}
  {(\rho'-\rho'')^{2\tau_m+m+3}}
  \left[
  \begin{aligned}
  &b_{m-1}|E|_{\rho',0}|E|_{\rho',m}
  +b_{m-1}|E|_{\rho',0}^{2}
  \left(|\partial_\theta u|_{\pm,\rho',m}
  +|u|_{\pm,\rho',m}\right)\\
  &\quad \quad  \quad \quad 
  +b_{m-1}^{2m^2+5m+6}
  \left(|E|_{\rho',m-1}^{2}
  +|h|_{\pm,\rho'',m-1}^{2}\right)
  \end{aligned}
  \right].
 \end{aligned} 
\end{equation}
\end{theorem}

\begin{proof}
Fix $m\ge1$, set  $d:=\rho'-\rho''$ and for 
$0\le j\le m$, denote
\begin{equation*}
 e_j:=|E|_{\rho',j},\qquad
 w_j:=|\partial_\theta u|_{\pm,\rho',j} \qand H_j:=|h|_{\pm,\rho'',j}.
\end{equation*}
Since in particular we satisfy hypotheses of 
Proposition~\ref{prop:quadratic}, we similarly apply Lemma~\ref{lem:exact-identity} and write 
\begin{equation*}
  \widetilde E=T_1+T_2,\qquad
  T_1:=S\,\partial_\theta E,\qquad T_2:=h^2F,  
\end{equation*}
where we denote  
\begin{equation*}
  F:=\int_0^1(1-t)F_t\,dt \qquad \text{ with }  F_t:=V'''(\id_\T+u+th), \quad \text{ for } 0 \le t \le 1 \;.
\end{equation*} 

We will bound $ T_1 $ and $T_2$ separately exactly as in Proposition~\ref{prop:quadratic} but we will separate terms of order $0$, $m-1$ and $m$. 	
 We first bound $T_1$. By Lemma~\ref{lem:der-hol}, for every $ 0\leq j\leq m$ it holds
\begin{equation*}
|\partial_\theta E|_{\rho'',j}
\leq
\frac{e_j}{d} \;. 
\end{equation*}
Hence the  estimate~\eqref{eq:leibniz-split} gives
\begin{equation}
\label{eq:T1-product}
|T_1|_{\rho'',m}
\leq
\frac{1}{d}
\left(
|S|_{\pm,\rho'',0}e_m
+
|S|_{\pm,\rho'',m}e_0
+
2^m|S|_{\pm,\rho'',m-1}e_{m-1}
\right).
\end{equation}

At order zero, Theorem~\ref{thm:solution-hol} gives
\begin{equation*}
|S|_{\pm,\rho'',0}
\leq
\xi_{0,\gamma,\tau}d^{-\tau_0}e_0.
\end{equation*}
By the definition of $\Theta_m$ and
\eqref{eq:def-bm}--\eqref{eq:def-vartheta}, it holds
\begin{equation*}
\Theta_m
\leq
\vartheta_m
b_{m-1}^{2m^2+4m+2}e_{m-1}.
\end{equation*}
Consequently, the order $m$ estimate of Theorem~\ref{thm:solution-hol} therefore gives
\begin{equation*}
|S|_{\pm,\rho'',m}
\leq
\xi_{m,\gamma,\tau}d^{-\tau_m}
\left(
e_m+w_me_0+
\vartheta_m b_{m-1}^{2m^2+4m+2}e_{m-1}
\right).
\end{equation*}
Similarly, at order $m-1$ we have
\begin{equation*}
\Theta_{m-1}
\leq
\vartheta_{m-1}b_{m-1}^{2m^2}e_{m-1},
\end{equation*}
so that 
\begin{equation*}
\begin{aligned}
|S|_{\pm,\rho'',m-1}
&\leq
\xi_{m-1,\gamma,\tau}(2+\vartheta_{m-1})
d^{-\tau_m}b_{m-1}^{2m^2}e_{m-1}\\
&\leq
\frac{\xi^{\mathrm{sol}}_{m,\gamma,\tau}}
{d^{\tau_m}}
b_{m-1}^{2m^2}e_{m-1}.
\end{aligned}
\end{equation*}
Substituting these three estimates into
\eqref{eq:T1-product}  yields
\begin{equation}
\label{eq:T1-final}
|T_1|_{\rho'',m}
\leq
\frac{2^{m+1}\xi^{\mathrm{sol}}_{m,\gamma,\tau}}
{d^{\tau_m+1}}
\left(
e_0e_m+e_0^2w_m+
b_{m-1}^{2m^2+4m+2}e_{m-1}^2
\right).
\end{equation}
Consequently,  we have in particular the cruder bound
\begin{equation}
\label{eq:error-order-first-term}
\begin{aligned}
|T_1|_{\rho'',m}
\leq{}&
\frac{\xi^{\mathrm{err}}_{m,\gamma,\tau,\bar d}
\left(1+|V|_{\rho^\sharp}\right)}
{2d^{2\tau_m+m+3}}
\left[
\begin{aligned}
&b_{m-1}e_0e_m
+b_{m-1}e_0^2
\left(w_m+|u|_{\pm,\rho',m}\right)\\
&+
b_{m-1}^{2m^2+5m+6}
\left(e_{m-1}^2+H_{m-1}^2\right)
\end{aligned}
\right].
\end{aligned}
\end{equation}

We now bound the term $T_2$. 
For $t\in[0,1]$, write
$\phi_t:=u+th$. Theorem~\ref{thm:solution-hol} and the
assumption on $e_0$ give
\begin{equation*}
H_0
\leq
\xi_{0,\gamma,\tau}d^{-\tau_0}e_0
\leq
\frac{\bar d}{18}.
\end{equation*}
Hence, uniformly for $t\in[0,1]$,
\begin{equation*}
|\phi_t|_{\pm,\rho'',0}
\leq
|u|_{\pm,\rho',0}+H_0
\leq
\frac{\bar d}{6}
\leq
\frac{\rho^\star-\rho''}{6}.
\end{equation*}
For every $0\leq r\leq m$,
Lemma~\ref{lem:composition-q-deriv} at order zero, followed by
Lemma~\ref{lem:der-hol}, gives
\begin{equation}
\label{eq:error-order-potential-composition}
\left|
V^{(r+3)}(\id_\T+\phi_t)
\right|_{\rho'',0}
\le \frac{(r+3)! 2^{r+3}}{\bar{d}^{r+3}} |V|_{\rho^\sharp}  \le 
\xi^{\mathrm{pot}}_{m,\bar d}|V|_{\rho^\sharp}.
\end{equation}

For $1\leq\ell\leq m$, the $\ell$-th derivative of $F_t$ is given by classical derivation formula as
\begin{equation*}
\partial_q^\ell F_t
=
\sum_{\mathcal P\vdash\{1,\ldots,\ell\}}
V^{(|\mathcal P|+3)}(\id_\T+\phi_t)
\prod_{B\in\mathcal P}\partial_q^{|B|}\phi_t.
\end{equation*}
We compute contributions of the different summands separately. 
The singleton partition is the only one whose term contains
$\partial_q^\ell \phi_t$. For every other term, all derivatives of $\phi_t$ have order at most $m-1$ and are therefore bounded by $b_{m-1}$. Since the number of partitions is at most $m^m$, \eqref{eq:error-order-potential-composition} gives, uniformly in
$t\in[0,1]$, that
\begin{align*}
|F_t|_{\rho'',0}
&\leq
\xi^{\mathrm{pot}}_{m,\bar d}|V|_{\rho^\sharp},\\
|F_t|_{\rho'',m}
&\leq
\xi^{\mathrm{pot}}_{m,\bar d}|V|_{\rho^\sharp}
\left(
|u|_{\pm,\rho',m}
+H_m
+m^m b_{m-1}^m
\right),\\
|F_t|_{\rho'',m-1}
&\leq
m^m\xi^{\mathrm{pot}}_{m,\bar d}|V|_{\rho^\sharp}
b_{m-1}^m.
\end{align*}
The above bound holds as well when substituting $F_t$ by 
$F=\int_0^1(1-t)F_t\,dt$.

Using the splitting~\eqref{eq:leibniz-split}, first to
$h^2$, and then to $h^2F$,  gives
\begin{equation*}
\begin{aligned}
|T_2|_{\rho'',m}
\leq{}&
\left(2H_0H_m+2^mH_{m-1}^2\right)|F|_{\rho'',0}
+H_0^2|F|_{\rho'',m}
+2^{2m-1}H_{m-1}^2|F|_{\rho'',m-1}\\
\leq{}&
2^{2m+1}m^m
\xi^{\mathrm{pot}}_{m,\bar d}|V|_{\rho^\sharp}
\bigg[
H_0H_m
+H_0^2\left(|u|_{\pm,\rho',m}+H_m\right)
+b_{m-1}^mH_{m-1}^2
\bigg].
\end{aligned}
\end{equation*}

At orders zero and $m$, Theorem~\ref{thm:solution-hol}, together with
the definition of $\xi^{\mathrm{sol}}_{m,\gamma,\tau}$, gives
\begin{equation*}
H_0
\leq
\frac{\xi^{\mathrm{sol}}_{m,\gamma,\tau}}
{d^{\tau_0}}e_0 \qand 
H_m \leq
\frac{\xi^{\mathrm{sol}}_{m,\gamma,\tau}}
{d^{\tau_m}}
\left(
e_m+w_me_0+
b_{m-1}^{2m^2+4m+2}e_{m-1}
\right).
\end{equation*}
We consequently obtain
\begin{equation*}
\begin{aligned}
|T_2|_{\rho'',m}
\leq {}&
\frac{
2^{2m+1}m^m\xi^{\mathrm{pot}}_{m,\bar d}
\left(1+\xi^{\mathrm{sol}}_{m,\gamma,\tau}\right)^3
|V|_{\rho^\sharp}}
{d^{2\tau_0+\tau_m}}
\\
&\times
\left[
b_{m-1}e_0e_m
+b_{m-1}e_0^2
\left(w_m+|u|_{\pm,\rho',m}\right)
+b_{m-1}^{2m^2+4m+3}e_{m-1}^2
+b_{m-1}^mH_{m-1}^2
\right].
\end{aligned}
\end{equation*}
Finally, using the trivial bounds
\begin{equation*}
2\tau_0+\tau_m\leq2\tau_m+m+3,
\qquad
2m^2+4m+3\leq2m^2+5m+6,
\qquad
m\leq2m^2+5m+6,
\end{equation*}
the definition of $\xi^{\mathrm{err}}_{m,\gamma,\tau,\bar d}$ yields then 
\begin{equation}
\label{eq:error-order-second-term}
\begin{aligned}
|T_2|_{\rho'',m}
\leq{}&
\frac{\xi^{\mathrm{err}}_{m,\gamma,\tau,\bar d}
\left(1+|V|_{\rho^\sharp}\right)}
{2d^{2\tau_m+m+3}}
\left[
\begin{aligned}
&b_{m-1}e_0e_m
+b_{m-1}e_0^2
\left(w_m+|u|_{\pm,\rho',m}\right)\\
&+
b_{m-1}^{2m^2+5m+6}
\left(e_{m-1}^2+H_{m-1}^2\right)
\end{aligned}
\right].
\end{aligned}
\end{equation}
Combining
\eqref{eq:error-order-first-term} and
\eqref{eq:error-order-second-term} completes the proof. 
\end{proof}

We now combine the preceding theorem with the order-zero quadratic
estimate to prove convergence at every order.

\begin{theorem}
\label{thm:convergence-order-m-new}
For every $m\ge0$, there exist an exponent $p_m\ge0$ and finite
constants $K_m,L_m>0$, independent of $V$, such that for every
$V\in\mathcal U$ and every $n \ge 0$, letting
$\alpha_0=\alpha_0(V)$ be given by
\eqref{eq:def-alpha-delta}, we have the bounds
\begin{equation}
  \label{eq:Hm-new}
  |E_n|_{\rho_n,m}
  \le K_m(n+1)^{p_m}\alpha_0^{2^{n-1}},
  \qquad
  |h_n|_{\pm,\rho_{n+1},m}
  \le L_m(n+1)^{p_m+2\tau_m}\alpha_0^{2^{n-1}} \;.
\end{equation}
Consequently, for every $V\in\mathcal U$, the limit
\begin{equation*}
 U(V):=\lim_{n\to\infty}u_n\in C_{\rho_\infty}^{\pm}
\end{equation*}
exists and satisfies $\mathcal E(U(V),V)=0$ and  $ U  \in \Hol^\infty ( \mathcal{U} ,  C ^\pm_{\rho_\infty} )$. Moreover, $U$ is tame holomorphic from $ \mathcal{U} $ to $ C^\pm_{\rho_\infty}$. 
\end{theorem}

\begin{proof}
For every $n\geq0$, set the intermediate strip 
\begin{equation*}
\widetilde\rho_n:=\frac{\rho_n+\rho_{n+1}}2,
\qquad \text{ so that } 
\rho_n-\widetilde\rho_n
=\widetilde\rho_n-\rho_{n+1}
=\frac{d_n}{2}.
\end{equation*}
We prove by induction on $m$ the stronger conclusion that there exist
$p_m\geq0$ and $K_m,L_m>0$, independent of $V$, such that for every
$V\in\mathcal U$ and every $n\geq0$,
\begin{equation*}
(\mathcal Q_m) \colon \qquad
\left\{
\begin{aligned}
&|E_n|_{\rho_n,m} \leq K_m(n+1)^{p_m}\alpha_0^{2^{n-1}} \\ 
&|h_n|_{\pm,\widetilde\rho_n,m} \leq L_m(n+1)^{p_m+2\tau_m}\alpha_0^{2^{n-1}}\\
&|u_n|_{\pm,\rho_n,m}
+|\partial_\theta u_n|_{\pm,\rho_n,m}
\leq 2K_m
\end{aligned}
\right.
\end{equation*}
holds. By Fact~\ref{fact:U-admissible},
\begin{equation*}
0\leq\alpha_0(V)<\varepsilon_{\mathcal U}<\frac18
\qquad\text{for every }V\in\mathcal U.
\end{equation*}
Note that if $V=0$, then $E_n=h_n=u_n=0$ for every $n$ and $ (\mathcal{Q}_m)$ trivially holds independently of the choice of $ p_m, K_m$ and $L_m$. In the estimates containing negative powers of $\alpha_0$, we may consequently restrict ourselves to $V\neq0$ for which $\alpha_0>0$.

We first show $(\mathcal Q_0)$. Set
\begin{equation*}
 \mu_0:=
 \frac{d_0^{2\tau_0}}
 {2^{4\tau_0}\xi^{\mathrm{err}}_{0,\gamma,\tau}}.
\end{equation*}
Proposition~\ref{prop:induction} gives
\begin{equation} \label{eq:e-0}
|E_n|_{\rho_n,0} ( = \delta_n) 
\leq\mu_0\alpha_0^{2^n}
\leq\mu_0\alpha_0^{2^{n-1}}.
\end{equation}
Thus the first inequality in $(\mathcal Q_0)$ holds
with $p_0=0$ and, for example,
\begin{equation*}
K_0:=\max\left\{1,\mu_0\right\} ,\qquad
 L_0:=
 \max\left\{1, \;
 2^{\tau_0}\xi_{0,\gamma,\tau}
 \left(\frac{\pi^2}{6(\rho_0-\rho_\infty)}\right)^{\tau_0}\mu_0\right\}.
\end{equation*}
Indeed, the first inequality in $(\mathcal Q_0)$  is  \eqref{eq:e-0}. For the second inequality,  Theorem~\ref{thm:solution-hol}, applied from $\rho_n$ to
$\widetilde\rho_n$, gives
\begin{align*}
|h_n|_{\pm,\widetilde\rho_n,0}
\leq
\frac{2^{\tau_0}\xi_{0,\gamma,\tau}}{d_n^{\tau_0}}
|E_n|_{\rho_n,0}\leq
2^{\tau_0}\xi_{0,\gamma,\tau}
\left(\frac{\pi^2}{6(\rho_0-\rho_\infty)}\right)^{\tau_0}
\mu_0(n+1)^{2\tau_0}\alpha_0^{2^n} \leq
L_0(n+1)^{2\tau_0}\alpha_0^{2^{n-1}}.
\end{align*}
Finally, (II$_n$) in Proposition~\ref{prop:induction} gives the third
inequality in $(\mathcal Q_0)$, concluding the proof for the base case $m = 0$.

Fix $m\ge1$ for which $(\mathcal Q_{m-1})$ holds true and let us show $(\mathcal{Q}_m)$. Set then 
\begin{equation*}
 B_{m-1}:=1+\sup_{\substack{V\in\mathcal U\\n\ge0}}\left(
 |E_n|_{\rho_n,0}
 +|u_n|_{\pm,\rho_n,m-1}
 +|\partial_\theta u_n|_{\pm,\rho_n,m-1}
 +|h_n|_{\pm,\widetilde\rho_n,m-1}\right)<\infty.
\end{equation*}
Its finiteness follows from  all three inequalities in $(\mathcal Q_{m-1})$. 
Theorem~\ref{thm:solution-hol} applies at order $m$, with source strip
$\rho_n$ and target strip $\widetilde\rho_n$, giving
\begin{equation*}
\begin{aligned}
|h_n|_{\pm,\widetilde\rho_n,m}
\leq
\frac{\xi_{m,\gamma,\tau}}{(d_n/2)^{\tau_m}}
\Bigl[
&|E_n|_{\rho_n,m}
+|\partial_\theta u_n|_{\pm,\rho_n,m}
 |E_n|_{\rho_n,0}\\
&+\vartheta_m
 \left(1+|\partial_\theta u_n|_{\pm,\rho_n,m-1}\right)^{2m^2+4m+2}
 |E_n|_{\rho_n,m-1}
\Bigr].
\end{aligned}
\end{equation*}
Using the inequalities 
\begin{equation*}
1+|\partial_\theta u_n|_{\pm,\rho_n,m-1}
\leq B_{m-1}
\qand 
|E_n|_{\rho_n,m-1}
\leq
K_{m-1}(n+1)^{p_{m-1}}
\alpha_0^{2^{n-1}}, 
\end{equation*}
we consequently obtain that 
\begin{equation*}
\begin{aligned}
|h_n|_{\pm,\widetilde\rho_n,m}
\leq {}&
2^{\tau_m}\xi_{m,\gamma,\tau}
\left(\frac{\pi^2}{6(\rho_0-\rho_\infty)}\right)^{\tau_m}
(n+1)^{2\tau_m}\\
&\times
\Bigl[
|E_n|_{\rho_n,m}
+\mu_0\alpha_0^{2^n}
 |\partial_\theta u_n|_{\pm,\rho_n,m}
+\vartheta_m B_{m-1}^{2m^2+4m+2}
 K_{m-1}(n+1)^{p_{m-1}}
 \alpha_0^{2^{n-1}}
\Bigr]\;. 
\end{aligned} 
\end{equation*}
Set 
\begin{align*}
\xi_m:={}&
2^{\tau_m+m^3+3m^2+11m+18}(m!)^{m+3}
\left(
1+\xi_{m,\gamma,\tau}
+\xi^{\mathrm{err}}_{m,\gamma,\tau,\bar d}
\right)
\left(1+r_{\mathcal U}\right)\\
&{}\times
\left(
1+\mu_0+B_{m-1}+K_{m-1}+L_{m-1}
+\frac{\pi^2}{6(\rho_0-\rho_\infty)}
\right)^{2\tau_m+2m^2+6m+12}.
\end{align*}
We therefore obtain the cruder bound
\begin{equation}
\label{eq:order-m-increment-bootstrap}
|h_n|_{\pm,\widetilde\rho_n,m}
\leq
\xi_m(n+1)^{2\tau_m}
\left[
|E_n|_{\rho_n,m}
+\alpha_0^{2^n}|\partial_\theta u_n|_{\pm,\rho_n,m}
+(n+1)^{p_{m-1}}\alpha_0^{2^{n-1}}
\right].
\end{equation}

Since $u_0$ is entire, an induction on $n \ge 0$  gives the stronger regularity
\begin{equation*}
u_n=u_{n-1}+h_{n-1}\in C^\pm_{\widetilde\rho_{n-1}} . 
\end{equation*}
We now bound the updated error $E_{n+1}$. By Theorem~\ref{thm:error-order-m-new}  with $\rho'=\rho_n$, $ \rho''=\rho_{n+1} $ it holds
\begin{equation*}
\begin{aligned}
|E_{n+1}|_{\rho_{n+1},m}
\leq {}&
\frac{\xi^{\mathrm{err}}_{m,\gamma,\tau,\bar d}
\left(1+|V|_{\rho^\sharp}\right)}
{d_n^{2\tau_m+m+3}}
\Bigl[
  B_{m-1} |E_n|_{\rho_n,0}|E_n|_{\rho_n,m}\\
&\quad
  + B_{m-1} |E_n|_{\rho_n,0}^2
\left(
|\partial_\theta u_n|_{\pm,\rho_n,m}
+|u_n|_{\pm,\rho_n,m}
\right)\\
&\quad
  + B_{m-1}^{2m^2+5m+6} 
\left(
|E_n|_{\rho_n,m-1}^2
+|h_n|_{\pm,\rho_{n+1},m-1}^2
\right)
\Bigr].
\end{aligned}
\end{equation*}

The induction hypothesis then provides
\begin{equation*}
\begin{aligned}
|E_n|_{\rho_n,m-1}^2
+|h_n|_{\pm,\rho_{n+1},m-1}^2
&\leq
|E_n|_{\rho_n,m-1}^2
+|h_n|_{\pm,\widetilde\rho_n,m-1}^2\\
&\leq
\left(K_{m-1}^2+L_{m-1}^2\right)
(n+1)^{2(p_{m-1}+2\tau_{m-1})}
\alpha_0^{2^n}.
\end{aligned}
\end{equation*}

Substituting these estimates into the preceding bound yields
\begin{equation*}
\begin{aligned}
|E_{n+1}|_{\rho_{n+1},m}
\leq {}&
\frac{\xi^{\mathrm{err}}_{m,\gamma,\tau,\bar d}
\left(1+|V|_{\rho^\sharp}\right)}
{d_n^{2\tau_m+m+3}}
\Bigl[
B_{m-1}\mu_0\alpha_0^{2^n}|E_n|_{\rho_n,m}\\
&\quad
+B_{m-1}\mu_0^2\alpha_0^{2^{n+1}}
\left(
|\partial_\theta u_n|_{\pm,\rho_n,m}
+|u_n|_{\pm,\rho_n,m}
\right)\\
&\quad
+B_{m-1}^{2m^2+5m+6}
\left(K_{m-1}^2+L_{m-1}^2\right)
(n+1)^{2(p_{m-1}+2\tau_{m-1})}
\alpha_0^{2^n}
\Bigr].
\end{aligned}
\end{equation*}
Set
\begin{equation*}
p_m
:=
2\left(p_{m-1}+2\tau_{m-1}\right)
+2\left(2\tau_m+m+3\right).
\end{equation*}
Using
\begin{equation*}
d_n^{-1}
=
\frac{\pi^2(n+1)^2}
{6(\rho_0-\rho_\infty)}
\end{equation*}
and the definition of $\xi_m$, we finally obtain
\begin{equation}
\label{eq:order-m-error-bootstrap}
|E_{n+1}|_{\rho_{n+1},m}
\leq
\xi_m\alpha_0^{2^n}
\Bigl[
(n+1)^{2(2\tau_m+m+3)} 
\Bigl(
|E_n|_{\rho_n,m}
+\alpha_0^{2^n}
\left(
|u_n|_{\pm,\rho_n,m}
+|\partial_\theta u_n|_{\pm,\rho_n,m}
\right)
\Bigr)
+(n+1)^{p_m}
\Bigr].
\end{equation}

Note that we end up with a coupled system \eqref{eq:order-m-error-bootstrap} and \eqref{eq:order-m-increment-bootstrap}. We will solve it using another induction on $n$. 

First observe that the series
\begin{equation*}
 \sum_{n\ge0}(n+1)^{p_m+2\tau_m+2}
 \varepsilon_{\mathcal U}^{2^{n-1}}
\end{equation*}
converges since $\varepsilon_{\mathcal U}\in(0,1)$. Choose an integer
$N\geq1$, depending on $m$, large
enough such that
\begin{equation}
 \label{eq:order-m-tail-choice}
 \frac{6\pi^2\xi_m}{\rho_0-\rho_\infty}
 \sum_{n\ge N}(n+1)^{p_m+2\tau_m+2}
 \varepsilon_{\mathcal U}^{2^{n-1}}\le 1 
\end{equation}
and, for every $n\ge N$,
\begin{equation}
 \label{eq:order-m-small-coefficients}
 \xi_m\left[
 (n+1)^{2(2\tau_m+m+3)+p_m}
 \varepsilon_{\mathcal U}^{2^{n-1}}
 +2(n+1)^{2(2\tau_m+m+3)}
 \varepsilon_{\mathcal U}^{2^n}\right]\le\frac14.
\end{equation}
Both requirements are possible by the superexponential decay of
$(\varepsilon_{\mathcal U}^{2^{n}})_n$.

To define $ K_m $ and $ L_m$, for every $0\leq n\leq N$  we first show recursively  that
\begin{equation}
\label{eq:finite-initial-sup}
 \mathscr{K}_n := \sup_{V\in\mathcal U \setminus  \lbrace 0 \rbrace  }
\left(
|u_n|_{\pm,\rho_n,m}
+|\partial_\theta u_n|_{\pm,\rho_n,m}
+\frac{|E_n|_{\rho_n,m}}
{(n+1)^{p_m}\alpha_0^{2^{n-1}}}
+\frac{|h_n|_{\pm,\widetilde\rho_n,m}}
{(n+1)^{p_m+2\tau_m}\alpha_0^{2^{n-1}}}
\right)<\infty.
\end{equation}
Let us start with the base case $ n =0$. 
Note that  $E_0=V'$ does not depend on $q$, and, since
$0<\alpha_0<1$,
\begin{equation*}
\frac{|E_0|_{\rho_0,m}}{\alpha_0^{1/2}}
=
\frac{\delta_0}{\alpha_0^{1/2}}
\leq
\mu_0\alpha_0^{1/2}
\leq
\mu_0.
\end{equation*}
As $u_0=0$, equation \eqref{eq:order-m-increment-bootstrap}  gives
\begin{equation*}
\frac{|h_0|_{\pm,\widetilde\rho_0,m}}{\alpha_0^{1/2}}
\leq 
 \frac{\xi_m}{ \alpha_0^{1/2} } \left( |E_0 |_{\rho_0, m} + \alpha_0^{1/2} \right)
\leq
\xi_m(\mu_0 +1 )  < \infty , 
\end{equation*}
and proves consequently $\mathscr{K}_0  < \infty$.  

Assume that $\mathscr{K}_n < \infty $ for some  $n<N$ and let us show that $\mathscr{K}_{n+1} < \infty $.  
For $V\neq0$, the identity $u_{n+1}=u_n+h_n$ and
Lemma~\ref{lem:der-hol} applied with gap $\widetilde\rho_n-\rho_{n+1}=d_n/2$, give 
\begin{equation*}
\begin{split}
|u_{n+1}|_{\pm,\rho_{n+1},m}
+|\partial_\theta u_{n+1}|_{\pm,\rho_{n+1},m}
& \leq
|u_n|_{\pm,\rho_n,m}
+|\partial_\theta u_n|_{\pm,\rho_n,m} +
\left(1+\frac{2}{d_n}\right)
|h_n|_{\pm,\widetilde\rho_n,m} \\ 
&\leq
\mathscr{K}_n\left[
1+
\left(1+\frac{2}{d_n}\right)
(n+1)^{p_m+2\tau_m}
\right] < \infty  \;. 
\end{split}
\end{equation*}
Dividing \eqref{eq:order-m-error-bootstrap}  by
$(n+2)^{p_m}\alpha_0^{2^n}$ gives
\begin{equation*}
\begin{aligned}
\frac{|E_{n+1}|_{\rho_{n+1},m}}
{(n+2)^{p_m}\alpha_0^{2^n}}
\leq
\frac{\xi_m}{(n+2)^{p_m}}
\left[
(n+1)^{2(2\tau_m+m+3)}
\mathscr{K}_n\left((n+1)^{p_m}+1\right)
+(n+1)^{p_m}
\right] < \infty .
\end{aligned}
\end{equation*} 

Finally, \eqref{eq:order-m-increment-bootstrap} at index $n+1$ gives
\begin{equation*}
\begin{aligned}
\frac{|h_{n+1}|_{\pm,\widetilde\rho_{n+1},m}}
{(n+2)^{p_m+2\tau_m}\alpha_0^{2^n}}
&\leq
\xi_m\left[
\frac{|E_{n+1}|_{\rho_{n+1},m}}
{(n+2)^{p_m}\alpha_0^{2^n}}
+
\frac{\alpha_0^{2^n}
|\partial_\theta u_{n+1}|_{\pm,\rho_{n+1},m}}
{(n+2)^{p_m}}
+
(n+2)^{p_{m-1}-p_m}
\right]\\
&\leq
\xi_m\left\{
\frac{\xi_m}{(n+2)^{p_m}}
\left[
(n+1)^{2(2\tau_m+m+3)}
\mathscr{K}_n\left((n+1)^{p_m}+1\right)
+(n+1)^{p_m}
\right]\right.\\
&\qquad\left.
+\mathscr{K}_n\left[
1+
\left(1+\frac{2}{d_n}\right)
(n+1)^{p_m+2\tau_m}
\right]
+1
\right\}
<\infty.
\end{aligned}
\end{equation*}
Taking the supremum over $V\in\mathcal U \setminus \{ 0 \}$ in the three preceding
estimates and adding them gives
\begin{equation*}
\begin{aligned}
\mathscr{K}_{n+1}
\leq {}&
(1+\xi_m)\mathscr{K}_n
\left[
1+
\left(1+\frac{2}{d_n}\right)
(n+1)^{p_m+2\tau_m}
\right]\\
&+
\frac{(1+\xi_m)\xi_m}{(n+2)^{p_m}}
\left[
(n+1)^{2(2\tau_m+m+3)}
\mathscr{K}_n\left((n+1)^{p_m}+1\right)
+(n+1)^{p_m}
\right]\\
&+
\xi_m(n+2)^{p_{m-1}-p_m}
<\infty.
\end{aligned}
\end{equation*}
We then conclude by induction that 
\begin{equation*}
K_m^0 := \sup_{0 \le n \le N}  \mathscr{K}_n < \infty. 
\end{equation*}
We are now ready to define the constants
\begin{equation*}
K_m:=\max\left\{4\xi_m,K_m^0\right\}
\qand
L_m:=3\xi_m(3K_m+1).
\end{equation*}


For every $V\in\mathcal U$, we prove by induction on $n\ge N$ the property 
\begin{equation*}
 (\mathcal{P}_n)\qquad
 \left\{
 \begin{aligned}
 |E_j|_{\rho_j,m}
 &\le K_m(j+1)^{p_m}\alpha_0^{2^{j-1}}
 && \forall \,0\le j\le n	,\\
 |h_j|_{\pm,\widetilde\rho_j,m}
 &\le L_m(j+1)^{p_m+2\tau_m}\alpha_0^{2^{j-1}}
 && \forall \, 0\le j<n.
 \end{aligned}
 \right.
\end{equation*}
The initial case $(\mathcal{P}_N)$ follows from the definition of $K_m$ and $ L_m \ge K_m$.

Assume now that $(\mathcal P_n)$ holds for some $n\geq N$ and let us
prove $(\mathcal P_{n+1})$. Since
\begin{equation*}
u_n=u_N+\sum_{i=N}^{n-1}h_i,
\end{equation*}
the definition of $K_m^0$ and the inequality $K_m\geq K_m^0$ give
\begin{equation*}
|u_N|_{\pm,\rho_N,m}
+
|\partial_\theta u_N|_{\pm,\rho_N,m}
\leq K_m.
\end{equation*}
Moreover, for every $i<n$,
\begin{equation*}
\widetilde\rho_i-\rho_n
\geq
\widetilde\rho_i-\rho_{i+1}
=
\frac{d_i}{2}.
\end{equation*}
Hence, by Lemma~\ref{lem:der-hol},
\begin{equation*}
|h_i|_{\pm,\rho_n,m}
+
|\partial_\theta h_i|_{\pm,\rho_n,m}
\leq
\left(1+\frac{2}{d_i}\right)
|h_i|_{\pm,\widetilde\rho_i,m}.
\end{equation*}
Using $(\mathcal P_n)$,  and since
\begin{equation*}
1+\frac2{d_i}
\leq
\frac{\pi^2(i+1)^2}{2(\rho_0-\rho_\infty)},
\qquad
\alpha_0\leq\varepsilon_{\mathcal U},
\qand
L_m\leq12\xi_mK_m,
\end{equation*}
we therefore obtain
\begin{equation*}
\begin{split}
|u_n|_{\pm,\rho_n,m}
+|\partial_\theta u_n|_{\pm,\rho_n,m}
&\leq
K_m+
\sum_{i=N}^{n-1}
\left(1+\frac2{d_i}\right)
|h_i|_{\pm,\widetilde\rho_i,m}\\
&\leq
K_m+
\frac{\pi^2L_m}{2(\rho_0-\rho_\infty)}
\sum_{i\geq N}
(i+1)^{p_m+2\tau_m+2}
\alpha_0^{2^{i-1}}\\
&\leq
K_m+
\frac{6\pi^2\xi_mK_m}{\rho_0-\rho_\infty}
\sum_{i\geq N}
(i+1)^{p_m+2\tau_m+2}
\varepsilon_{\mathcal U}^{2^{i-1}}. 
\end{split}
\end{equation*}
By ~\eqref{eq:order-m-tail-choice} it then comes 
\begin{equation}\label{eq:bound-2Km}
|u_n|_{\pm,\rho_n,m}
+|\partial_\theta u_n|_{\pm,\rho_n,m} \leq 2K_m. 
\end{equation}
Note that this is actually the last inequality of $ (\mathcal{Q}_m)$ at fixed $n$. We will use this to prove $ (\mathcal{P}_{n+1}) $. 
We inject the latter bound in ~\eqref{eq:order-m-error-bootstrap}, and then using
$(\mathcal P_n)$ to obtain
\begin{align*}
|E_{n+1}|_{\rho_{n+1},m}
&\leq
\xi_m\alpha_0^{2^n}
\Bigl[
(n+1)^{2(2\tau_m+m+3)}
\Bigl(
K_m(n+1)^{p_m}\alpha_0^{2^{n-1}}
+2K_m\alpha_0^{2^n}
\Bigr)
+(n+1)^{p_m}
\Bigr]\\
&=
K_m\alpha_0^{2^n}\xi_m
\left[
(n+1)^{2(2\tau_m+m+3)+p_m}
\alpha_0^{2^{n-1}}
+
2(n+1)^{2(2\tau_m+m+3)}
\alpha_0^{2^n}
\right] 
+\xi_m(n+1)^{p_m}\alpha_0^{2^n}.
\end{align*}
Since $\alpha_0\leq\varepsilon_{\mathcal U}$,
\eqref{eq:order-m-small-coefficients} bounds the first term on the
last line by $\frac{K_m}{4}\alpha_0^{2^n}$.
Moreover, since $K_m\geq4\xi_m$, the second term is bounded by $ \frac{K_m}{4}(n+1)^{p_m}\alpha_0^{2^n}$.

Consequently, it follows
\[
|E_{n+1}|_{\rho_{n+1},m}
\leq
\frac{K_m}{4}
\left(1+(n+1)^{p_m}\right)
\alpha_0^{2^n}
\leq
K_m(n+2)^{p_m}\alpha_0^{2^n}, 
\]
which is the first inequality in $ ( \mathcal{P}_{n+1} )$. 

Likewise, \eqref{eq:order-m-increment-bootstrap}, $(\mathcal P_n)$, and
the latter bound give the second inequality
\begin{align*}
|h_n|_{\pm,\widetilde\rho_n,m}
&\leq
\xi_m(n+1)^{2\tau_m}
\left[
K_m(n+1)^{p_m}\alpha_0^{2^{n-1}}
+2K_m\alpha_0^{2^n}
+(n+1)^{p_{m-1}}\alpha_0^{2^{n-1}}
\right]\\
&\leq
\xi_m(3K_m+1)
(n+1)^{p_m+2\tau_m}\alpha_0^{2^{n-1}}\\
&\leq
L_m(n+1)^{p_m+2\tau_m}\alpha_0^{2^{n-1}} . 
\end{align*}
Thus $(\mathcal P_{n+1})$ holds. Induction then gives $(\mathcal P_n)$ for every $n\geq N$.

We now resume the proof of $(\mathcal{Q}_{m})$.
For $0\leq n \leq N$, the definition of $\mathscr K_n$, together with
$K_m\geq K_m^0$ and $L_m\geq K_m$, gives
\begin{equation*}
|E_n|_{\rho_n,m}
\leq
K_m(n+1)^{p_m}\alpha_0^{2^{n-1}}
\qand
|h_n|_{\pm,\widetilde\rho_n,m}
\leq
L_m(n+1)^{p_m+2\tau_m}\alpha_0^{2^{n-1}}.
\end{equation*}
For every $n>N$, the first estimate follows from $(\mathcal P_n)$ and
the second from $(\mathcal P_{n+1})$. Hence the first two inequalities
in $(\mathcal Q_m)$ hold for every $n\geq0$.

Moreover, for $0\leq n\leq N$,
\begin{equation*}
|u_n|_{\pm,\rho_n,m}
+|\partial_\theta u_n|_{\pm,\rho_n,m}
\leq
\mathscr K_n
\leq
K_m^0
\leq
K_m,
\end{equation*}
while for $n\geq N$ the estimate~\eqref{eq:bound-2Km}  obtained above gives
\begin{equation*}
|u_n|_{\pm,\rho_n,m}
+|\partial_\theta u_n|_{\pm,\rho_n,m}
\leq2K_m.
\end{equation*}
Thus the third inequality in $(\mathcal Q_m)$ also holds for every
$n\geq0$. This proves $(\mathcal Q_m)$ and completes the induction on
$m$. Since $\rho_{n+1}<\widetilde\rho_n$, the first two inequalities in
$(\mathcal Q_m)$ imply \eqref{eq:Hm-new}.

We now show that the limit $ U $ exists and lies in $ \Hol^\infty( \mathcal{U} , C^\pm_{\rho_\infty} )  $.
The property  $(\mathcal Q_m)$ gives for every $ m \ge 0$ that
\begin{equation*}
\sum_{n\geq0}
\sup_{V\in\mathcal U}|h_n(V)|_{\pm,\rho_\infty,m}
\leq
L_m\sum_{n\geq0}
(n+1)^{p_m+2\tau_m}
\varepsilon_{\mathcal U}^{2^{n-1}} <\infty.
\end{equation*}
Thus the series
\begin{equation*}
\sum_{n\geq0}h_n
\end{equation*}
converges normally in every defining seminorm of
$C^\pm_{\rho_\infty}$ and uniformly for $V\in\mathcal U$.

By induction, every map
\begin{equation*}
V\in\mathcal U\longmapsto u_n(V)\in C^\pm_{\rho_\infty}
\end{equation*}
is holomorphic, so is $ U$ by uniform convergence of the series in a complete space. Moreover, $(\mathcal Q_m)$ gives
\begin{equation*}
\sup_{V\in\mathcal U}
|U(V)|_{\pm,\rho_\infty,m}
\leq 2K_m
\qquad\text{for every }m\geq0.
\end{equation*}
Thus $U$ is uniformly bounded on $\mathcal U$ in every defining
seminorm of $C_{\rho_\infty}^{\pm}$. By the definition of tameness and since $\mathcal{U}$ is open, this
implies that $U$ is tame.

It remains to show that $U$ satisfies the Euler--Lagrange equation.
Normal convergence gives
$u_n(V)\to U(V)$ in every seminorm
$|\cdot|_{\pm,\rho_\infty,m}$. The identity
\begin{equation*}
\begin{aligned}
&V'\circ(\id_\T+u_n)-V'\circ(\id_\T+U(V))\\
&\qquad=
(u_n-U(V))
\int_0^1
V''\circ\bigl(\id_\T+U(V)+t(u_n-U(V))\bigr)\,dt
\end{aligned}
\end{equation*}
together with~\eqref{eq:comp-bound-order-zero} and
\eqref{eq:leibniz-step} therefore gives
\begin{equation*}
V'\circ(\id_\T+u_n)
\longrightarrow
V'\circ(\id_\T+U(V))
\qquad\text{in }\mathscr C_{\rho_\infty}.
\end{equation*}
Since $\Delta$ is continuous, it follows that the sequence  $ 
(E_n)_{ n \ge 0} $ converges to $  \mathcal{E}(U(V),V) $ in  $ \mathscr{C}_{\rho_\infty}$. 

On the other hand, since $\rho_\infty<\rho_n$, $(\mathcal Q_m)$ gives
\begin{equation*}
|E_n|_{\rho_\infty,m}
\leq
|E_n|_{\rho_n,m}
\leq
K_m(n+1)^{p_m}\alpha_0^{2^{n-1}}
\longrightarrow0 \quad \text{ as } n \to \infty , 
\end{equation*}
for every $m\geq0$. Therefore
$\mathcal E(U(V),V)=0$ in $\mathscr C_{\rho_\infty}$.

\end{proof}

We conclude this section by completing the proof of the KAM theorem.

\begin{proof}[Proof of Theorem~\ref{thm:hol-KAM}]
In virtue of Theorem~\ref{thm:convergence-order-m-new}, we only need to
show that, for every $V\in\mathcal U^\R$ and $\omega\in\DC$, the map
$\id_\T+U(V,\e(\omega))$ is a degree-one orientation-preserving analytic
circle diffeomorphism and that the corresponding invariant graph is a
KAM curve of rotation number $\omega$.

Let us fix $V\in\mathcal U^\R$. For $\omega\in\DC$ with 
$q:=\e(\omega)$ the shifts are the real
translations
\begin{equation*}
f^\pm(\theta,q)=f(\theta\pm\omega,q).
\end{equation*}
Consequently, the solution $S$ of $[A,AS]_\Delta=AE$ is real-valued whenever $A$ and $E$ are.
Since $u_0=0$, the iteration then shows inductively that every $u_n(V,q,\cdot)$ is
real valued and this obviously passes to the limit $U(V)$. 
Set
\begin{equation*}
\psi :=\id_\T+U(V,q).
\end{equation*}
Proposition~\ref{prop:induction} and passage to the limit give
\begin{equation*}
\sup_{\theta\in\T}|\partial_\theta U(V,q,\theta)|\leq\frac{1}{8} .
\end{equation*}
Thus $\psi$ is the lift of a degree-one orientation-preserving analytic circle diffeomorphism. Let $\psi^\pm(\theta):=\psi(\theta\pm\omega)$ and define 
\begin{equation*}
\Psi 
:= \left( \psi, \psi - \psi^{-}   \right).
\end{equation*}
Since $ V' ( \psi) =  \psi^+ + \psi^- - 2 \psi $, it follows that
\begin{equation*}
F_V \circ \Psi
= \left( \psi + (\psi - \psi^{-}) +   V' (\psi) ,  (\psi - \psi^{-}) +   V' (\psi) \right) 
= \left( \psi^+ , \psi^+ - \psi \right) = \Psi^+ \;.
\end{equation*}
Hence the image of $\Psi$ is an analytic invariant graph, and the
dynamics induced on it is conjugate by $\psi$ to the rigid rotation of
angle $\omega$. Namely, it is a KAM curve of rotation number
$\omega$.
\end{proof}

\subsection{Holomorphy of the
\texorpdfstring{$\beta$}{beta}-function}
\label{subsec:beta-holo}
We conclude the section by deducing the holomorphy of Mather's
$\beta$-function from Theorem~\ref{thm:hol-KAM}.

\begin{proof}[Proof of Theorem~\ref{thm:kam-main-beta}]
For $V\in\mathcal U$, set
\begin{equation*}
h_V:=\id_\T+U(V)
\end{equation*}
and define
\begin{equation}
\label{eq:Phi-formula}
\Phi(V)
:=
\frac12\int_\T\left(U(V)^+-U(V)\right)^2
+
\int_\T V\circ h_V.
\end{equation}

Since $V$ has zero mean, it admits a periodic holomorphic primitive
$P_V$ satisfying $P_V'=V$. Hence
\begin{equation*}
0
=
\int_\T\partial_\theta(P_V\circ h_V)
=
\int_\T(V\circ h_V)(1+\partial_\theta U(V)).
\end{equation*}
Integration by parts and the Euler--Lagrange
equation~\eqref{eq:EL-U} therefore give
\begin{align*}
\int_\T V\circ h_V
&=
-\int_\T(V\circ h_V)\,\partial_\theta U(V)\\
&=
\int_\T U(V)\,\partial_\theta(V\circ h_V)\\
&=
\int_\T
U(V)(1+\partial_\theta U(V))(V'\circ h_V)\\
&=
\int_\T
U(V)(1+\partial_\theta U(V))\Delta U(V).
\end{align*}
Consequently,
\begin{equation}
\label{eq:Phi-U-formula}
\Phi(V)
=
\frac12\int_\T\left(U(V)^+-U(V)\right)^2
+
\int_\T
U(V)(1+\partial_\theta U(V))\Delta U(V).
\end{equation}

By Theorem~\ref{thm:hol-KAM},
Proposition~\ref{prop:Cpm-properties}(ii),
Lemma~\ref{lem:der-hol}, the product
estimate~\eqref{eq:leibniz-step}, and
Lemma~\ref{lem:linear-postcomposition}, the right-hand side
of~\eqref{eq:Phi-U-formula} is holomorphic.
Moreover, passage to the limit in the third estimate of
$(\mathcal Q_m)$ bounds all its factors uniformly at every order.
Thus $\Phi$ is bounded in every target seminorm and is tame without any shift of seminorms. 

Finally, let $V\in\mathcal U\cap\Hol_\rho^\R(\T)$ and
$\omega\in\DC$. Theorem~\ref{thm:hol-KAM} provides the KAM curve
parametrized by
$\id_\T+U(V,e^{2\pi i\omega})$. Formula~\eqref{eq:beta-general}
and~\eqref{eq:Phi-formula} give
\begin{equation*}
\beta_V(\omega)
=
\frac{\omega^2}{2}
+
\Phi(V)(e^{2\pi i\omega}),
\end{equation*}
which proves the theorem.
\end{proof}

\section{Perturbative theory and rigidity of the
\texorpdfstring{$\beta$}{beta}-functional}
\label{sec:perturbative}

Fix $\gamma,\tau,\rho > 0$ and let $\mathcal{U}$ be the neighbourhood
of the zero potential provided by Theorem~\ref{thm:hol-KAM}, $\Phi$ the
corresponding tame holomorphic extension of
$\beta_V(\omega)-\omega^2/2$ given by
Theorem~\ref{thm:kam-main-beta}, and
$U$ the complex KAM solution on $\mathcal{U}\times\KK$. Its restriction to real potentials and real Diophantine frequencies parametrizes the KAM
curves constructed above.
Recall also that for a real potential $ V$, the map 
$\Phi(V)$ is real on $ \DC$. 

This section consists of the proof of all rigidity results. The main tool is the expansion at the zero potential of the map $V\mapsto D\Phi(V)$.

\subsection{Fréchet derivatives of the \texorpdfstring{$\beta$}{beta}-function}
\label{subsec:frechet-deriv}

We compute the first and second Fréchet derivatives of $\Phi$
in the potential and evaluate them at $V=0$.

\begin{lemma}
\label{lem:beta-potential}
Let $V \in \mathcal{U}$. The Fréchet derivative of $\Phi$ at $V$
in any direction $W \in \Hol_\rho(\T)$ is
\begin{equation}
    \label{eq:D1Phi}
    D\Phi(V)[W] \;=\; \int_\T W\circ(\id_\T+U(V)).
\end{equation}
\end{lemma}

\begin{proof}
Theorems~\ref{thm:hol-KAM}
and~\ref{thm:kam-main-beta} justify differentiation.
Differentiating the discrete Lagrangian 
$\frac12(U^+-U)^2+V(\id_\T+U)$
with respect to $V$ in the direction $W$ gives
\begin{equation}
    \label{eq:cal-1}
    \begin{aligned}
    D_V\left[\frac12(U^+-U)^2+V(\id_\T+U)\right][W]
    &= (U^+-U)\cdot D_V(U^+-U)[W] \\
    &\quad{}+ V'(\id_\T+U)\cdot D_VU[W]
    + W(\id_\T+U).
    \end{aligned}
\end{equation}
Integrating~\eqref{eq:cal-1} over $\T$ and using the Euler--Lagrange
equation $V'(\id_\T+U)=\Delta U$ to cancel the terms containing
$D_VU[W]$ gives
\begin{equation*}
    D\Phi(V)[W]
    = \int_\T W(\id_\T+U)
    + \int_\T\left(V'(\id_\T+U)-\Delta U\right)\cdot D_VU[W]
    = \int_\T W\circ(\id_\T+U).
\end{equation*}
This ends the proof.
\end{proof}

At the zero potential, this vanishes.

\begin{corollary}
\label{cor:D1-at-zero}
At the zero potential, $$ D\Phi(0)[W] = 0$$ 
for every $W \in \Hol_\rho (\T)$. 
\end{corollary}

\begin{proof}
At $V=0$, $U=0$, so $D\Phi(0)[W]
= \int_\T W(\theta)\,d\theta = 0$ for zero-mean $W$.
\end{proof} 

We then compute the second derivative at $V=0$.
\begin{lemma}\label{lem:D2-at-zero}
The second Fréchet derivative of $\Phi$ at the zero potential in the directions $Z,W \in \Hol_\rho(\T)$ is given by
\begin{equation*}
    D^2\Phi|_{V=0}[W,Z] \;=\; \int_\T W\cdot\Green Z,
\end{equation*}
where  
\begin{equation}\label{eq:def-green}
\Green := -\Delta^{-1}\partial_\theta^2 \;. 
\end{equation}
\end{lemma}
From now on, we identify $\Hol_\rho(\T) \subset C_\rho$
via the embedding $W \mapsto ((q,\theta)\mapsto W(\theta))$. On zero-mean functions, $\Delta=\nabla\nabla_-$, so its inverse is obtained by applying the two cohomological inverses successively. Lemma~\ref{lem:der-hol} and Proposition~\ref{prop:Cpm-properties}(iii) therefore show that
$\Green=-\Delta^{-1}\partial_\theta^2$ is a tame continuous linear map.

\begin{proof}
Again, the holomorphy of $U$ and $\Phi$ makes the following differentiations
legitimate. 
Differentiating
$$ V'(\id_\T+U(V))=\Delta U(V) $$ 
at $V=0$ gives
$$ \Delta\,D_VU(0)[Z]=\partial_\theta Z .$$
 Hence we identify 
$$ D_VU(0)[Z]=\Delta^{-1}\partial_\theta Z . $$  Differentiating~\eqref{eq:D1Phi} at $V=0$ therefore gives
\begin{equation*}
    D^2\Phi|_{V=0}[W,Z]
    = \int_\T \partial_\theta W\cdot\Delta^{-1}\partial_\theta Z
    = \int_\T W\cdot(-\Delta^{-1}\partial_\theta^2)Z
    = \int_\T W\cdot\Green Z,
\end{equation*}
using integration by parts and the fact that $\partial_\theta$
and $\Delta^{-1}$ commute.
This ends the proof. 
\end{proof}

\subsection{The operator \texorpdfstring{$\Green$}{G} }
\label{subsec:linearization}

We now show that $\Green$ acts diagonally on Fourier modes with holomorphic multipliers on $\KK$. This will then be used to build the Wronskian criterion for the proof of the rigidity statements. 

For $n\in\N^*$, consider the rational function on the Riemann sphere $ \hat{\C}$ 
away from the $n$th roots of unity,
\begin{equation}
    \label{eq:sigma-n-q}
    \sigma_n(q) := \frac{4\pi^2n^2\,q^n}{(q^n-1)^2},
\end{equation}
and set $$ \sigma_{-n}:=\sigma_n . $$
The displayed formula gives $0$ at
$q=0$ and $0$ at $q=\infty$. 
Its restriction to $\KK$ is therefore well defined and holomorphic in
the interior of $\KK$.
Writing
$q=\e(\omega)$ in~\eqref{eq:sigma-n-q} yields
\begin{equation*}
    \sigma_n(\e(\omega)) = -\frac{\pi^2n^2}{\sin^2(\pi n\omega)}.
\end{equation*}
Equivalently, the identity
$q^n\lambda_n(q)=1+\lambda_n(q)$ gives
$\sigma_n(q)=4\pi^2n^2\left(\lambda_n(q)+\lambda_n(q)^2\right)$.
Corollary~\ref{cor:der-lambda-seminorm} therefore implies that
$\sigma_n\in\Hol^\infty(\KK,\C)$. 
The next fact records the multipliers of $\Delta$ and of
$\Green=-\Delta^{-1}\partial_\theta^2$.
\begin{fact} \label{fact:eigen}
For every $n\in\Z^*$,
\begin{equation*}
\Delta e_n=(q^n-2+q^{-n})e_n,
\qquad \partial_\theta^2e_n=-4\pi^2n^2e_n,
\qquad \Green e_n=\sigma_n(q)e_n.
\end{equation*}
\end{fact} 
The next lemma justifies termwise use of the identity
$\Green(\sum_n\widehat W_ne_n)=\sum_n\widehat W_n\sigma_ne_n$ for
$W\in\Hol_\rho(\T)$.

\begin{lemma}
\label{lem:sum-sin}
For every $0<\rho'<\rho$, the operator $\Green$ is a tame continuous
linear map from $\Hol_\rho(\T)$ to
$\Hol^\infty(\KK,\Hol_{\rho'}(\T))$, and it coincides with the series
\begin{equation}
    \label{eq:series-hess}
    \Green W(q) =  \sum_{n\in\Z^*}\sigma_n(q) \hat{W}_n e_n ,
\end{equation}
the series being normally convergent in
$\Hol^\infty(\KK,\Hol_{\rho'}(\T))$.
In particular, for every $W,Z\in\Hol_\rho(\T)$,
\begin{equation} \label{eq:repr-D2phi}
D^2 \Phi \vert_{V =0}[W,Z](q)
=\sum_{n\in\Z^*}\sigma_n(q)\,\hat W_{-n}\hat Z_n.
\end{equation}
\end{lemma}

\begin{proof}
By Fact~\ref{fact:eigen}, equation~\eqref{eq:series-hess} is formally
the Fourier expression of $-\Delta^{-1}\partial_\theta^2$. It remains to
prove normal convergence. Fix $0<\rho'<\rho''<\rho$.
Since
$\sigma_n=4\pi^{2}n^{2}\left(\lambda_n+\lambda_n^{2}\right)$, the
Leibniz rule together with Proposition~\ref{prop:der-lambda-bound} and
$ \sum_{i=0}^{k}\binom{k}{i} \,i! \,(k-i)! = (k+1)!$ gives, for every
$k\geq0$ and every $q\in\KK^0$,
\begin{equation}
    \label{eq:der-sigma-bound}
    \left|\partial_q^{k} \sigma_n(q) \right|
    \le  8 \pi^{2} \,(k+1)! \,c_\gamma^{\,k+2} \,
    |n|^{(k+2)(\tau+1)+k+2} \;.
\end{equation}
The bound~\eqref{eq:der-sigma-bound} also holds in the inversion chart by
the symmetry $\sigma_n(1/q)=\sigma_{-n}(q)$. Hence, for
$W\in\Hol_\rho(\T)$, Lemma~\ref{lem:fourier-extraction} at the intermediate strip $\rho''$ gives
$|\hat W_n|\le |W|_{\rho''}e^{-2\pi\rho''|n|}$, while
$|e_n|_{\rho'}=e^{2\pi\rho'|n|}$. It follows that
\begin{equation*}
    \sum_{n \in \Z^*} \left|\sigma_n \, \hat W_n \, e_n \right|_{\rho',m}
   \le  8 \pi^{2} \, |W|_{\rho''} (m+1)! \,c_\gamma^{m+2}
    \sum_{n \neq 0} |n|^{(m+2)(\tau+1)+m+2} 
    e^{-2\pi(\rho''-\rho') |n|} < \infty \;. 
\end{equation*}
Thus the series~\eqref{eq:series-hess} converges normally in every Whitney seminorm. The displayed bounds are linear in $|W|_{\rho''}$ and therefore also show that $\Green$ is tame. Finally, the identity~\eqref{eq:repr-D2phi} is directly obtained from Lemma~\ref{lem:D2-at-zero}. 
\end{proof}

\subsection{The linearized jet map and the Wronskian}
\label{subsec:wronskian}

Since $D\Phi(0)=0$, the linear term in the expansion of $D\Phi(V)$ at
the zero potential is $D^2\Phi(0)[V,\cdot]$. Accordingly, for
$V\in\Hol_\rho(\T)$, define
\begin{equation*}
\begin{array}{rcl}
\Green_V\colon\Hol_\rho(\T)
&\longrightarrow&
\Hol^\infty(\KK)\\
h
&\longmapsto&
D^2\Phi(0)[V,h].
\end{array}
\end{equation*}
Set the \emph{remainder}
\begin{equation*}
\begin{array}{rcl}
R\colon\mathcal U\times\Hol_\rho(\T)
&\longrightarrow&
\Hol^\infty(\KK)\\
(V,h)
&\longmapsto&
D\Phi(V)[h]-\Green_V(h).
\end{array}
\end{equation*}
We will use the following fact for the proof of rigidity theorems. 

\begin{fact} \label{fact:remainder}
The bilinear map $(V,h)\mapsto\Green_V(h)$ is continuous.
For a fixed $V\in\Hol_\rho(\T)$, the map $\Green_V$ is linear and continuous. The map $R$ is holomorphic and linear in the second variable. 
Furthermore, we have the identities
\begin{equation*}
 R(0,h)=0,
 \qquad D_V R(0,h)[Z]=0 \;
\end{equation*}
for every $ h,Z\in\Hol_\rho(\T)$.
 \end{fact}

\begin{proof}
Fix $0<\sigma<\rho_\infty$. By~\eqref{eq:D1Phi},
Lemma~\ref{lem:composition-q-deriv} using the order-zero
bound above and applied with
$\rho'=\rho^\star$, $\rho''=\sigma$, $\widetilde\rho=\rho_\infty$ and
$\kappa=1/9$, Lemma~\ref{lem:linear-postcomposition}
and Theorem~\ref{thm:hol-KAM}, the map
\begin{equation*}
(V,h)\longmapsto D\Phi(V)[h]
=\int_\T h\circ(\id_\T+U(V))
\end{equation*}
is holomorphic and linear in $h$.  
On the other hand, Lemma~\ref{lem:D2-at-zero} gives
\begin{equation*}
\Green_V(h)=D^2\Phi(0)[V,h]=\int_\T V\cdot\Green h.
\end{equation*}
By Lemma~\ref{lem:sum-sin}, $\Green$ is continuous and coincides with
its Fourier representation. Hence $(V,h)\mapsto\Green_V(h)$ is
continuous bilinear, and it follows that $R$ is holomorphic and linear
in its second variable. Finally,
\begin{equation*}
R(0,h)=D\Phi(0)[h]=0
\end{equation*}
by Corollary~\ref{cor:D1-at-zero}, while
\begin{equation*}
D_VR(0,h)[Z]=D^2\Phi(0)[Z,h]-\Green_Z(h)=0
\end{equation*}
by the definition of $\Green_Z(h)$.
\end{proof}

\smallskip 
For the rest of this section, fix $\omega_0\in\DC$ and set
$q_0:=\e(\omega_0)$. 
Let $ \mathscr{V}$ be a  Fourier-separated
subspace of $\Hol_\rho^\R(\T) $. Let $ \V  \subset \mathscr{V}$ be a subspace of dimension $d \ge 1$ with a basis
$(h^{(1)},\ldots,h^{(d)})$. Choose pairwise distinct positive integers
$n_1,\ldots,n_d$ and phases $\xi_1,\ldots,\xi_d\in\mathbb S^1$
as in Definition~\ref{def:fourier-separated}, that is such that the real matrix 
\begin{equation} \label{def:H}
H := \left[
\Re \!\left(
\xi_k\,\widehat{h^{(j)}}_{n_k}
\right)
\right]_{1\le k,j\le d}
\end{equation}
is invertible. 

The Wronskian of the linearized jet operator in the $q$-variable is denoted by
\begin{equation}
    \label{eq:wronskian-q}
    \Pi_V^{\mathrm{lin}} := \det\left[\partial_q^{k-1}\Green_V(h^{(j)})(q_0)\right]_{1\le k,j\le d} \;. 
\end{equation}  

Since $\Green_V$ is linear in $V$ and the determinant in
\eqref{eq:wronskian-q} is multilinear in its columns, $V \mapsto
\Pi_V^{\mathrm{lin}}$ is a homogeneous holomorphic polynomial of degree
$d$.

We now choose the particular potential
\begin{equation} \label{eq:part-potential}
V^* := \frac{1}{2} \sum_{k=1}^d
\left( \bar{\xi}_k e_{n_k}+ \xi_k e_{-n_k}\right) \;. 
\end{equation}
It is obviously not small, so we will also denote $ V_t^* := t V^*$ 
where $t\in\C$ is a free parameter. 
 In particular for real values of $t$, this potential   
$V_t^*$ belongs to $\Hol_\rho^\R(\T)$ for every
$\rho>0$ and moreover to $ \mathcal{U}^\R$ for small enough values of $t \in \R$. Under stability of the Fourier truncations, we can actually tune the phases to obtain that $V^*$ genuinely lies in $\mathscr{V}$. 
\begin{proposition}
\label{prop:phase-adapted-potential}
Assume that $\mathscr V$ is stable under Fourier truncations. Then the phases $ (\xi_k )_{ 1 \le k \le d} $ can be chosen so that, for every $ h \in \mathscr{V}$, we have $ \xi_k \hat{h}_{n_k}  \in \R $. 
In particular, the corresponding matrix is simply $$ H := \left[ \xi_k \widehat{h^{(j)}}_{n_k} \right]_{1 \le j,k \le d} . $$ 
Moreover with the above choice we obtain that $ V^* \in \mathscr{V}$. 
\end{proposition} 
\begin{proof}
Let $\xi_1^0,\ldots,\xi_d^0$ be phases initially supplied by
Definition~\ref{def:fourier-separated}, so that
\begin{equation*}
H^0
:=
\left[
\Re \left(
\xi_k^0\widehat{h^{(j)}}_{n_k}
\right)
\right]_{1\leq k,j\leq d}
\end{equation*}
is invertible.  For $ n \ge 1$, set 
\begin{equation*}
\begin{array}{r@{\;}c@{\quad}c@{\quad}c}
P_n:=\mathcal T_n-\mathcal T_{n-1}\colon
& \Hol_\rho^\R(\T)
& \longrightarrow
& \Hol_\rho^\R(\T)\\
& h
& \longmapsto
& \widehat h_n e_n+\overline{\widehat h_n}\,e_{-n}.
\end{array}
\end{equation*}
Stability under Fourier truncations gives
$P_n(\mathscr V)\subset\mathscr V$ for every $ n \ge 1$. 
Since $\mathscr{V}$ is Fourier separated, $P_n(\mathscr V)$ has
dimension at most one. Since $H^0$ is chosen to be invertible, each of its rows is nonzero. Consequently, for every $1\le k\le d$, it follows that 
$P_{n_k}(\mathscr V)\neq\{0\}$ and thus has dimension exactly $1$. Choose $ u^{(k)} $ to be a vector in $P_{n_k}(\mathscr V)$  such that 
\begin{equation}
\xi_k :=   \overline{ \mathcal{F}_{n_k}  ( u^{(k)} )  }  \in \S^1 \;. 
\end{equation}
Note that $$ u^{(k)} = \bar{\xi}
_k e_{n_k} + \xi_k e_{-n_k} .$$
Then, for every $h\in\mathscr V$, we have $P_{n_k} (h)  \in \R \, u^{(k)}$, and consequently $ \xi_k \hat{h}_{n_k} \in \R$.  

It remains to verify that the new phases still satisfy the
Fourier-separation condition. Write the new matrix $ H :=\left[  \Re \left( \xi_k \widehat{h^{(j)}}_{n_k}  \right) \right]_{1 \le j,k \le d} $ which is obviously equal to $\left[ \xi_k \widehat{h^{(j)}}_{n_k} \right]_{1 \le j,k \le d}$. 
Observe that, for every $ 1 \le k, j \le d$, we have 
\[    \Re \left( \xi_k^0  \widehat{  h^{(j)}}_{n_k}  \right) = \Re \left( \xi_k^0   \bar{\xi}_k  \right) \cdot  \xi_k  \widehat{  h^{(j)}}_{n_k} \;.    
\]
Thus, one has 
\begin{equation*}
H^0
=
\operatorname{diag}\left(
\Re \left(\xi_1^0\overline{\xi_1}\right),
\ldots,
\Re \left(\xi_d^0\overline{\xi_d}\right)
\right) H \; . 
\end{equation*}
The invertibility of $H^0$ implies that of $H$, since the diagonal terms are all nonzero. 
To end the proof, note that since every $u^{(k)} \in \mathscr{V}$ and $ V^* = \frac{1}{2} \sum_{k=1}^d u^{(k)}$ we obtain that $ V^* $ lies in $\mathscr{V}$ as well.
\end{proof}

Define the matrix
\begin{equation} \label{eq:def-Sigma}
\Sigma (q) := \left[\partial_q^{k-1} \sigma_{n_j} (q) \right]_{1\le k,j\le d}   \;. 
\end{equation}
We now provide identities on the Jacobian and Wronskian of the operator $ \Green$ at this potential. 
 
\begin{lemma} \label{lem:factorization-det}
For every $ t \in \C$, we have 
\begin{equation*}
\left[\partial_q^{k-1}\Green_{V_t^*}(h^{(j)})(q_0)\right]_{1\le k,j\le d}
=t\,\Sigma(q_0)H.
\end{equation*}
In particular, the linearized Wronskian is given by
\begin{equation*}
\Pi_{V_t^*}^{\mathrm{lin}} = t^d \det H\,\det \Sigma(q_0) \;. 
\end{equation*}
\end{lemma}

\begin{proof}
By Lemma~\ref{lem:sum-sin}, for every $1\leq j\leq d$,
\begin{equation*}
\begin{aligned}
\Green_{V_t^*}(h^{(j)})
&=
\frac{t}{2}
\sum_{\ell=1}^d
\left(
\sigma_{n_\ell}\xi_\ell
\widehat{h^{(j)}}_{n_\ell}
+
\sigma_{-n_\ell}\overline{\xi_\ell}
\widehat{h^{(j)}}_{-n_\ell}
\right)\\
&=
t\sum_{\ell=1}^d
\sigma_{n_\ell}
\operatorname{Re}\left(
\xi_\ell\widehat{h^{(j)}}_{n_\ell}
\right).
\end{aligned}
\end{equation*}
Here we used $\sigma_{-n_\ell}=\sigma_{n_\ell}$ and
$ \widehat{h^{(j)}}_{-n_\ell} = \overline{\widehat{h^{(j)}}_{n_\ell}}$ since $h^{(j)}$ is real-valued. 
Consequently, for every
$1\leq k,j\leq d$,
\begin{equation*}
\partial_q^{k-1}
\Green_{V_t^*}(h^{(j)})(q_0)
=
t\sum_{\ell=1}^d
\partial_q^{k-1}\sigma_{n_\ell}(q_0)
\operatorname{Re}\left(
\xi_\ell\widehat{h^{(j)}}_{n_\ell}
\right).
\end{equation*}
By the definitions of $\Sigma$ and $H$, this is precisely the
$(k,j)$-entry of $t\,\Sigma(q_0)H$. Hence
\begin{equation*}
\left[
\partial_q^{k-1}
\Green_{V_t^*}(h^{(j)})(q_0)
\right]_{1\leq k,j\leq d}
=
t\,\Sigma(q_0)H.
\end{equation*}
Taking determinants ends the proof. 
\end{proof}

Consequently, to show that 
$\Pi_{V_t^*}^{\mathrm{lin}}\neq0$ for nonzero $t$, it is sufficient to prove that $\Sigma(q_0)$ is invertible. 
To that end, we derive a workable expression for the derivatives of the
"eigenvalue functions" $\sigma_n$.

\begin{lemma}
\label{lem:der-sig-n}
For all integers $n,k\geq1$, there is a polynomial
$\sig_{n,k}\in\pi^2\Q[q]$ such that
\begin{equation*}
    \partial_q^{k-1}\sigma_n(q)
    =\frac{\sig_{n,k}(q)}{(q^n-1)^{k+1}}.
\end{equation*}
Moreover, $\deg\sig_{n,k}=n+(k-1)(n-1)$, and its leading coefficient is
\begin{equation*}
    s_{n,k}:=4\pi^2n^2(-1)^{k-1}
    \frac{(n+k-2)!}{(n-1)!}.
\end{equation*}
\end{lemma}

\begin{proof}
Fix $n\geq1$. We prove the statements by induction on $k\geq1$.
For $k=1$, set $\sig_{n,1}:=4\pi^2n^2q^n$.

Suppose that the statements hold for some $k\geq1$. Differentiating
$\partial_q^{k-1}\sigma_n=\sig_{n,k}(q^n-1)^{-(k+1)}$ gives
\begin{equation*}
    \sig_{n,k+1}
    =(q^n-1)\sig_{n,k}'-(k+1)nq^{n-1}\sig_{n,k}.
\end{equation*}
The coefficient of $q^{\deg\sig_{n,k}+n-1}$ in the right-hand side is
\begin{equation*}
    s_{n,k}\left[n+(k-1)(n-1)-(k+1)n\right]
    =-(n+k-1)s_{n,k}=s_{n,k+1},
\end{equation*}
which is nonzero. Therefore
$\deg\sig_{n,k+1}=n+k(n-1)$, and the asserted formula for the leading
coefficient follows as well. We complete the proof by running the induction.
\end{proof}

We are now ready to prove that $\det\Sigma(q_0)\neq0$  whenever $\omega_0$ is algebraic.

\begin{proposition}
\label{prop:S-invertible}
If $\omega_0 \in \DC $ is algebraic  then
$\Sigma(q_0)  $ is invertible.
\end{proposition}

\begin{proof}
By Lemma~\ref{lem:der-sig-n}, multiplying the $\ell$-th column of
$\Sigma(q)$ by $(q^{n_\ell}-1)^{d+1}$ clears the denominators of all
its entries, hence
\begin{equation*}
\begin{aligned}
P(q)
&:=\prod_{\ell=1}^d(q^{n_\ell}-1)^{d+1}
\det\left[\partial_q^{k-1}\sigma_{n_\ell}(q)\right]_{1\le  k,\ell\le  d}\\
&=\det\left[\sig_{n_\ell,k}(q)
(q^{n_\ell}-1)^{d-k}\right]_{1\le  k,\ell\le  d}
\end{aligned}
\end{equation*}
is a polynomial in $\pi^{2d}\Q[q]$.

We compute its degree and leading coefficient. The $(k,\ell)$-entry in
the last determinant has degree
\begin{equation*}
    n_\ell+(k-1)(n_\ell-1)+n_\ell(d-k)
    =n_\ell d-k+1.
\end{equation*}
For any permutation $\varpi$ of $\{1,\ldots,d\}$, the corresponding
term in the determinant expansion is
\begin{equation*}
    \operatorname{sgn}(\varpi)\prod_{\ell=1}^d
    \sig_{n_\ell,\varpi(\ell)}(q)
    (q^{n_\ell}-1)^{d-\varpi(\ell)},
\end{equation*}
which has degree
\begin{equation*}
\begin{aligned}
    \sum_{\ell=1}^d\left[n_\ell d-\varpi(\ell)+1\right]
    &=d\sum_{\ell=1}^d n_\ell-\frac{d(d-1)}2
    =:p.
\end{aligned}
\end{equation*}
This value $ p$ is independent of the choice of the permutation. 
Thus $\deg P\le  p$, and the coefficient of $q^p$ is the determinant of the leading coefficients of the matrix with entries $\sig_{n_\ell,k}$ for $ 1 \le k, \ell \le d$,  that is 
\begin{equation*}
\det\left[s_{n_\ell,k}\right]_{1\le  k,\ell\le  d} =(-1)^{d(d-1)/2}(2\pi)^{2d} \prod_{\ell=1}^d n_\ell^2  \det D 
\end{equation*}
where 
\begin{equation*}
\begin{aligned}
D
&:=
\left[
n_\ell(n_\ell+1)\cdots(n_\ell+k-2)
\right]_{1\le  k,\ell\le  d}.
\end{aligned}
\end{equation*}
We claim that $ \det D = \prod_{i<j}(n_j-n_i)$.
Indeed, let us denote 
\begin{equation*}
d_1(x):=1,\qquad d_k(x):=x(x+1)\cdots(x+k-2),\quad k\ge2.
\end{equation*}
Each $d_k$ is a monic polynomial of degree $k-1$, and
$D=[d_k(n_\ell)]_{1\le  k,\ell\le  d}$. Therefore elementary row
operations transform $D$ into the Vandermonde matrix
\begin{equation*}
\begin{pmatrix}
1&1&\cdots&1\\
n_1&n_2&\cdots&n_d\\
n_1^2&n_2^2&\cdots&n_d^2\\
\vdots&\vdots&&\vdots\\
n_1^{d-1}&n_2^{d-1}&\cdots&n_d^{d-1}
\end{pmatrix} 
\end{equation*}
without modifying its determinant. 

Hence
$\det D=\prod_{1\le  i<j\le  d}(n_j-n_i)$ and the leading
coefficient of $P$ is
\begin{equation*}
    (-1)^{d(d-1)/2}(2\pi)^{2d}
    \prod_{\ell=1}^d n_\ell^2
    \prod_{i<j}(n_j-n_i)\neq0.
\end{equation*}

In particular, $P$ is a nonzero polynomial of degree $p$.
By Lemma~\ref{lem:der-sig-n}, each
$\sig_{n,k}$ lies in $\pi^2\Q[q]$, so $\pi^{-2d}P\in\Q[q]$.
Consequently, every zero of $P$ is algebraic.
Since $\omega_0$ is algebraic and irrational, the Gelfond--Schneider
theorem (see e.g. ~\cite{Baker1975}) therefore shows
that $q_0$ is transcendental. Hence $P(q_0)\neq0$. Moreover,
$q_0^{n_\ell}\neq1$ for every $\ell$ because $\omega_0$ is irrational,
so the definition of $P$ gives $\det\Sigma(q_0)\neq0$.
\end{proof}

\subsection{From the Wronskian to rigidity} \label{subsec:wronsk-to-rig}
We keep all notations and the setting of the previous subsection. If $\mathscr{V}$ is assumed to be stable under Fourier truncation, it will be specified. 
Before providing the proofs of the rigidity results we 
first show how the linearized Wronskian being nonzero allows to obtain that the real nonlinear Wronskian  is as well.
Recall that
\begin{equation*}
\Pi_V := \det \left[
\partial_\omega^{k-1}D\beta(V)[h^{(j)}](\omega_0)
\right]_{1\le k,j\le d}.
\end{equation*}     
\begin{theorem}
\label{thm:wronskian-small}
Assume that $\omega_0$ is algebraic.
There exists $t_0>0$ such that
\begin{equation*}
    \Pi_{V^*_{t}}\neq 0
\end{equation*}
for every $t\in\R$ such that $0<|t|<t_0$.
\end{theorem}

\begin{proof}
We first shift from the $\omega$ variable to the $q$ variable. 
Since $\partial_\omega=2\pi i q\partial_q$, for every integer $k\geq0$ one has
\begin{equation*}
  \partial_\omega^k
  = (2\pi i)^k \sum_{\ell=0}^k \stirling{k}{\ell}\,q^\ell\partial_q^\ell.
\end{equation*}
Elementary row operations on the jet matrices yield, for every potential $V$,
\begin{equation}
    \label{eq:wronskian-change}
    \Pi_V   = (2\pi i q_0)^{d(d-1)/2}\,\tilde{\Pi}_V,
    \qquad \text{where}  \quad 
    \tilde{\Pi}_V := \det\left[
    \partial_q^{k-1}D\Phi(V)[h^{(j)}](q_0)
    \right]_{1\le k,j\le d}.
\end{equation}
Consequently we reduce to showing that
$\tilde{\Pi}_{V_t^*}\neq0$ for small $t \neq 0$.

Since $V_t^*=tV_1^*$, Fact~\ref{fact:remainder} shows that the map $t\mapsto R(V_t^*,h)$ is holomorphic and has a zero of
order at least two at the origin. Therefore, there exists a matrix-valued map $\check{R} : \C \to M_d(\C)$ that is holomorphic near $0$ and such that
\begin{equation*}
\left[
\partial_q^{k-1}D\Phi(V_t^*)[h^{(j)}](q_0)
\right]_{1\leq k,j\leq d}
=
t\,\Sigma(q_0)H+t^2 \check{R}(t).
\end{equation*}
Consequently, by Lemma~\ref{lem:factorization-det} one has
\begin{equation*}
\widetilde\Pi_{V_t^*}
=t^d\det\left(\Sigma(q_0)H+t\check R(t)\right) =t^d\det\Sigma(q_0)\det H+t^{d+1}\mathcal R(t).
\end{equation*}
where $\mathcal{R} : \C \to \C $ is a holomorphic function on a complex disk $\{\,|t|<r\,\}$ for some $r>0$. Up to reducing $r$, there exists $\kappa>0$ such that
\begin{equation*}
|\mathcal R(t)|\leq\kappa
\qquad\text{for every }|t|\leq r/2.
\end{equation*}
Proposition~\ref{prop:S-invertible} gives
$\det\Sigma(q_0)\neq0$, while $\det H\neq0$ by the
Fourier-separation assumption. Set
\begin{equation*}
t_0
:=
\min\left\{
\frac r2,
\frac{\left|\det\Sigma(q_0)\det H\right|}{2\kappa}
\right\} > 0 . 
\end{equation*}
Then, for every $t$ such that  $0<|t|<t_0$, we obtain
\begin{equation*}
\left|\widetilde\Pi_{V_t^*}\right|
\geq
|t|^d
\left(
\left|\det\Sigma(q_0)\det H\right|-\kappa|t|
\right)
>0.
\end{equation*}
\end{proof}

Shyness of the set involved in the rigidity result is actually  provided by Theorem~\ref{thm:shyness}. It states that, with the present definition of real analyticity, the zero set of a nonzero real-analytic function is shy. Since the proof relies on different techniques, the statement and its proof are provided in the last subsection. 
We are now ready to prove the main result of this paper.

 \begin{proof}[Proof of Theorem~\ref{thm:main-rig}]
Fix a linear subspace
$\V\subset\mathscr{V}$ of finite dimension $d \ge 1  $.  Pick a basis $(h^{(1)},\ldots,h^{(d)})$ of $ \V$ and 
$(n_k,\xi_k)_{1\le  k\le  d}$ as in
Definition~\ref{def:fourier-separated}, that is such that $ H$, as defined in \eqref{def:H}, is invertible. 
Write the Wronskian as a map
\begin{equation*}
\begin{array}{r@{\;}c@{\quad}c@{\quad}c@{\;}l}
\Pi\colon
& \mathcal U^\R
& \longrightarrow
& \R
& \\ 
& V
& \longmapsto
& \Pi_V
& \displaystyle :=
\det\left[
\partial_\omega^{k-1}
D\beta(V)[h^{(j)}](\omega_0)
\right]_{1\leq k,j\leq d}.
\end{array}
\end{equation*}
By Theorem~\ref{thm:kam-main-beta}, this map $\Pi$ is real analytic on
$\mathcal U^\R$. Then Theorem~\ref{thm:wronskian-small} provides a real
potential $V_t^*\in\mathcal U^\R$, for $t\neq0$ sufficiently small,
such that $\Pi(V_t^*)\neq0$. Hence
\begin{equation*}
    \mathcal Z_{\V}
    :=
    \left\{V\in\mathcal U^\R:\Pi(V)=0\right\}
\end{equation*}
is a proper real analytic subset of $\mathcal U^\R$. It is relatively
closed in $ \mathcal{U}^\R$ by continuity of $\Pi$. 

The real Fréchet space $\Hol_\rho^\R(\T)$ is separable, and
$\mathcal U^\R$ is open and connected. Therefore, by Theorem~\ref{thm:shyness}, the zero set 
$\mathcal Z_{\V}$ is shy in $\Hol_\rho^\R(\T)$.

For $V\in\mathcal U^\R$, consider the jet-map
\begin{equation*}
\begin{array}{r@{\;}c@{\;}c@{\;}c}
\mathcal J_V\colon
& \V\cap\left(-V+\mathcal U^\R\right)
& \longrightarrow
& \R^d\\
& W
& \longmapsto
& J_{\omega_0}^{d-1}\left(\beta_{V+W}\right).
\end{array}
\end{equation*}
In the basis $(h^{(1)},\ldots,h^{(d)})$, its Jacobian at $ 0$ is 
\begin{equation*}
    D \mathcal{J}_V(0)  =  \left[
    \partial_\omega^{k-1}
    D\beta(V)[h^{(j)}](\omega_0)
    \right]_{1\le  k,j\le  d}.
\end{equation*}
Thus, whenever $V \notin \mathcal Z_{\V}$, by the inverse function
theorem, $\mathcal J_V$ is a local diffeomorphism at $0$.

It remains to prove that whenever  $\mathscr{V}$ is
stable under Fourier truncations then $\mathcal Z_{\V} \cap \mathscr{V} $ is shy in $\mathscr{V}$. By
Proposition~\ref{prop:phase-adapted-potential}, the potential $V_t^*$ belongs to
$\mathscr V$ for every real $t$. 
Thus the restriction of  $\Pi$ to 
$\mathcal U^\R\cap\mathscr V$ does not vanish identically. Hence
Theorem~\ref{thm:shyness}, applied to the restriction of $\Pi$,
shows that its zero set $\mathcal Z_\V\cap\mathscr V$ is shy in $\mathscr V$.
\end{proof}

We now provide the proof of the more general version of Corollary~\ref{cor:def-rig-gen}, which gives in particular Example~\ref{ex:def-rig}.

\begin{proof}[ Proof of Corollary~\ref{cor:def-rig-gen}]
For every $n\in\N^*$, let $\mathcal Z_n:=\mathcal Z_{\V_n}\cap\mathscr V$,
where $\mathcal Z_{\V_n}$ is the exceptional set provided by
Theorem~\ref{thm:main-rig}, and set
\begin{equation*}
    \mathcal R
    :=
    \left(\mathcal U^\R\cap\mathscr V\right)
    \setminus\bigcup_{n\in\N^*}\mathcal Z_n.
\end{equation*}
Actually $\mathcal Z_{\V_n}$ is only provided by the above Theorem when $ \dim  \V_n \ge 1$. If $\V_n = \lbrace 0 \rbrace$, we choose $ \mathcal{Z}_n = \emptyset $ as a convention.
By Theorem~\ref{thm:main-rig}, every $\mathcal Z_n$ is relatively closed in $\mathcal U^\R\cap\mathscr V$ and shy in $\mathscr V$. Hence $\mathcal R$ is residual and prevalent in $\mathcal U^\R\cap\mathscr V$.

Let $V_0\in\mathcal R$ and suppose that the deformation is not constant.
Since it is real analytic and $I$ is connected, there exists a smallest
integer $\ell\geq1$ such that
\begin{equation*}
    W_\ell
    :=
    \left.\partial_t^\ell V_t\right|_{t=0}
    \neq0.
\end{equation*}
Choose $n_\ell$ such that $W_\ell\in\V_{n_\ell}$ and set
$d:=\dim\V_{n_\ell}$. The hypothesis on the infinite jets implies that
\begin{equation*}
    J_{\omega_0}^{d-1}\left(\beta_{V_t}\right)
\end{equation*}
is independent of $t$. Differentiating $\ell$ times at $t=0$, all terms
involving derivatives of $V_t$ of order smaller than $\ell$ vanish, and
therefore
\begin{equation*}   
    D  J_{\omega_0}^{d-1}  \beta  (V_0)[W_\ell] = 0  \;.
\end{equation*}
Since $V_0\notin\mathcal Z_{n_\ell}$, this leaves no choice but $W_\ell=0$, yielding a contradiction. Thus the deformation is constant in
a neighbourhood of $0$, and consequently on $I$ by real analyticity and connectedness. 
\end{proof}

We next let both the base potential and the deformation directions vary and
deduce prevalence of the tuples that define an affine deformation space.

\begin{proof}[Proof of Theorem~\ref{thm:prevalent-affine}]
Consider  the function
\begin{equation}
\label{eq:tuple-wronskian}
\begin{array}{rcl}
\mathfrak P\colon \mathscr P_d
&\longrightarrow&
\R
\\[4pt]
\bigl(V,h^{(1)},\ldots,h^{(d)}\bigr)
&\longmapsto&
\displaystyle
\det\left[
\partial_\omega^{k-1}
D\beta(V)[h^{(j)}](\omega_0)
\right]_{1\leq k,j\leq d}.
\end{array}
\end{equation}
By 
Theorem~\ref{thm:kam-main-beta} this function is real analytic on $ \mathscr{P}_d$. We now show that it is not identically $ 0$. Consider the subspace 
\begin{equation*}
 E_d:=\Vect_{\R}\{\cos(2\pi j\,\cdot):1\le  j\le  d\},
\end{equation*}
with its displayed basis, $n_k=k$, and $\xi_k=1$. The corresponding
Fourier-separation matrix is $\frac12\operatorname{Id}_d$.
Theorem~\ref{thm:wronskian-small} then provides a real even potential
$V_t^*\in\mathcal U^\ev$ for which the
determinant in~\eqref{eq:tuple-wronskian} is nonzero.

The space $\Hol_\rho^\R(\T)$ is separable, and $\mathscr P_d$ is connected
because $\mathcal U^\R$ is star-shaped. Applying
Theorem~\ref{thm:shyness}  to $\mathfrak P$ shows that
\begin{equation*}
    \mathscr N_d:=\{P\in\mathscr P_d:\mathfrak P(P)=0\}
\end{equation*}
is shy in the ambient space
$\left(\Hol_\rho^\R(\T)\right)^{d+1}$. Hence
$\mathscr R_d:=\mathscr P_d\setminus\mathscr N_d$ is prevalent in
$\mathscr P_d$. It is open by continuity of $\mathfrak P$ and
dense by the empty-interior property of shy sets proved
in~\cite{hunt1992prevalence}.

For $P=(V,h^{(1)},\ldots,h^{(d)})\in\mathscr R_d$, the matrix in
\eqref{eq:tuple-wronskian} is the differential at the origin of the locally defined map
\begin{equation*}
    a\longmapsto
    J_{\omega_0}^{d-1}\!\left(
      \beta_{V+\sum_{j=1}^d a_jh^{(j)}}
    \right).
\end{equation*}
 The inverse function theorem gives the asserted local diffeomorphism.
Its invertible differential also implies that the directions $h^{(j)}$
are linearly independent, so they define a $d$-dimensional affine space.
If two potentials in a sufficiently small neighbourhood of the base
potential have the same $\beta$-function, their displayed jets are equal. Local injectivity then makes the
two parameters, and hence the two potentials, equal. This proves the
stated local $\beta$-rigidity.
\end{proof}

The open and prevalent sets just obtained can be chosen simultaneously in
every finite dimension.

\begin{proof}[Proof of Corollary~\ref{cor:simultaneous-affine}]
Set $H:=\Hol_\rho^\R(\T)$, equip $H^\N=H\times H^{\N^*}$ with its product Fréchet topology, and let
\begin{equation*}
    \pi_d:
    H^\N
    \longrightarrow H^{d+1}
\end{equation*}
be the projection onto the first $d+1$ components, that is the base potential and the first $d$ directions of perturbation. Let $\mathscr R_d$ be the set supplied by Theorem~\ref{thm:prevalent-affine}. Then
\begin{equation*}
    \mathscr{R} :=\bigcap_{d\geq1}\pi_d^{-1}(\mathscr R_d) \;.
\end{equation*}
Each $\pi_d^{-1}(\mathscr N_d)$ is a shy cylinder in $H^\N$, as one may embed a compactly supported transverse measure for $\mathscr N_d$ in the first $d+1$ coordinates. Since shy Borel sets form a $\sigma$-ideal, the complement of $\mathscr R$ in $\mathcal U^\R\times H^{\N^*}$ is shy in $H^\N$. Moreover, every $\pi_d^{-1}(\mathscr R_d)$ is open and dense in this KAM product domain. Thus $\mathscr R$ is prevalent and residual in $\mathcal U^\R\times H^{\N^*}$.  Then $ \mathscr{R}$ obviously satisfies the conclusion of the corollary. 
\end{proof}

We end this subsection with the one-frequency rigidity argument.

\begin{proof}[Proof of Theorem~\ref{thm:one-frequency-rigidity}]
Set $Z:=V_2-V_1$ and $q_0:=\e(\omega_0)$. Since $\mathcal B$ is balanced, $Z\in\mathcal{B}+\mathcal{B} \subset \mathcal{U}^\R$. On the open convex set
\begin{equation*}
    \mathcal{B} \cap(\mathcal{B} + \mathcal{B}-Z),
\end{equation*}
which contains both $0$ and $V_1$, the function
\begin{equation*}
    F(V):=\Phi(V+Z)(q_0)-\Phi(V)(q_0)
\end{equation*}
is real analytic by Theorem~\ref{thm:kam-main-beta}. Indeed, on this set
$V\in\mathcal B\subset\mathcal U^\R$ and
$V+Z\in\mathcal B+\mathcal B\subset\mathcal U^\R$.

After shrinking $\mathcal W$, assume that
$V_i+\mathcal W\subset\mathcal B$ for $i=1,2$. The hypothesis and the formula for $\beta$ in Theorem~\ref{thm:kam-main-beta} give
\begin{equation*}
    F(V_1+W)=0
    \qquad \text{for every } W \in \mathcal{W} .
\end{equation*}
For any point in the displayed convex domain, restrict $F$ to the real line segment joining that point to $V_1$. Consequently, by analyticity, $F$ vanishes on the whole domain. Differentiating at $0$ in a direction $W\in\Hol_\rho^\R(\T)$ therefore yields
\begin{equation*}
    D\Phi(Z)[W](q_0)=D\Phi(0)[W](q_0).
\end{equation*}
Set
\begin{equation*}
    \phi:=\id_\T+U(Z,q_0).
\end{equation*}
Since $U(0)=0$, the first-variation formula~\eqref{eq:D1Phi} and the
zero-mean condition on $W$ give
\begin{equation*}
    \int_\T W\circ\phi
    = \int_\T W =0  
\end{equation*}
 for every $W \in \Hol_\rho^\R(\T)$.
By Theorem~\ref{thm:hol-KAM}, $\phi$ is an orientation-preserving
analytic circle diffeomorphism. A change of variables therefore gives
\begin{equation*}
    \int_\T (\phi^{-1})' \cdot W =0
    \qquad\text{for every }W\in\Hol_\rho^\R(\T).
\end{equation*}
Testing against the real zero-mean trigonometric polynomials shows that
$(\phi^{-1})'$ is constant. Its integral is $1$, so
$(\phi^{-1})'=1$. Thus $\phi$ is a rotation and by the normalization assumptions it must be the identity. Finally,
\eqref{eq:EL-U} gives $Z'=0$. Since $Z$ has zero mean over $\T$, it follows
that $Z=0$, and hence $V_1=V_2$.
\end{proof}

\subsection{Shyness of zero sets of holomorphic maps in general Fréchet spaces}

We conclude with the proof of shyness of the zero sets of nonzero analytic maps.
\begin{theorem}\label{thm:shyness}
Let $X$ be a separable real Fréchet space, and let $ Z \subset X$ be a proper real analytic subset. Then $Z$ is shy in $X$. 
\end{theorem}

Some similar results can be found, for instance, in~\cite[Corollary~2]{BogachevMalofeev2014} and the result in finite dimension is well known. 
The proof relies on the construction of a  Gaussian measure  whose Cameron--Martin space is dense in the ambient Fréchet space. 
General references on the topic are found in the book of Bogachev
\cite[Chapters~2 and~3]{BogachevGaussian}.

\medskip
Let us recall the setting and provide preliminaries for the proof. 
Let $ X $ be a real Fréchet space and consider $(p_k)_{k\ge 0 }$ to be an increasing sequence of seminorms defining the Fréchet topology on $X$. 
We first construct a centred Radon Gaussian probability measure $\gamma$
on $X$ whose Cameron--Martin space is dense in $X$ for the Fréchet topology. This is well known in the case of Banach spaces.

 Since $X$ is separable, there exists a sequence $ (x_n)_{n \in \N}$ of points in $X$ with dense span in $X$ and such that for every $n\geq 0 $, the seminorms of $x_n$ satisfy
\begin{equation*}
    p_k(x_n)\leq2^{-n} , \quad \forall  \, 0 \le k \le n \;.
\end{equation*}
In particular, for every $ k \ge 0$ note that 
\[ \sum_{ n\ge 0} p_k(x_n) \le \sum_{ n \le k}  p_k(x_n)  +  2  < \infty \;.  \]
Let $(g_n)_{n \in\N}$ be a sequence of independent standard Gaussian variables and consider the series
\[ G := \sum_{n \ge 0} g_n \cdot x_n \;. \]
 For every fixed $k\ge 0 $, we have
\begin{equation*}
\E (p_k(G)) \le  \E \left[ \sum_{n \ge 0 }|g_n| \cdot p_k(x_n) \right]
\le   \E   ( |g_0|) \cdot \sum_{n \ge 0 } p_k(x_n)  
<\infty \;.
\end{equation*}
Consequently, the series $G$ converges almost surely in every seminorm, and hence by completeness and countability of the family of the seminorms, almost surely in the Fréchet topology. We then denote by $ \gamma$ the law of $ G$ which is a finite Radon measure.
Let $ X' $ denote the continuous dual of $X$.
 By definition, for every  $\ell\in X'$, one has
\begin{equation} \label{eq:sum-l-G}
    \ell(G)=\sum_{n\ge 0}g_n\ell(x_n),
\end{equation} 
We now consider the Cameron--Martin space of $\gamma$ and show that it is dense in $X$. Recall that this space is classically defined as 
\begin{equation*}
H
:=
\left\{
h\in X:
\|h\|_{H}<\infty
\right\},
\end{equation*}
where
\begin{equation*}
\|h\|_{H}
:=
\sup\left\{
|\ell(h)|:
\ell\in X',
\quad
\int_X\ell(x)^2\,d\gamma(x)\leq1
\right\}, 
\end{equation*}
see \cite[Definition~2.2.7]{BogachevGaussian}. 
 For every $\ell\in X'$, since the  $g_n$ are independent, it holds 
\begin{equation}
\label{eq:gaussian-series-variance}
    \int_X\ell(x)^2\,d\gamma(x)
    =
    \mathbb E\left[\ell(G)^2\right]
    =
    \sum_{j\ge 0}\ell(x_j)^2.
\end{equation}
In particular, for every $n\geq0$, one has
\begin{equation*}
    |\ell(x_n)|
    \leq
    \left(
        \int_X\ell(x)^2\,d\gamma(x)
    \right)^{1/2}.
\end{equation*}
In other words, we have   $x_n \in H$ with $ \|x_n\|_H \le 1$ and the vector space spanned by these points is, by construction, dense in $X$. 

We are now ready to prove shyness of proper analytic subsets. 

\begin{proof}[Proof of Theorem~\ref{thm:shyness}]
Consider a connected open subset $\mathcal O\subset X$ and a nonzero
real-analytic function $f:\mathcal O\to\R$. We must show that
\begin{equation*}
Z:=\{x\in\mathcal O:f(x)=0\}
\end{equation*}
is shy in $X$.
That is, there exists a compactly supported probability measure $\mu$ on
$X$ such that $\mu(x+Z)=0$ for every $x\in X$. Since $\gamma$ is finite
and Radon, there exists a compact set $K\subset X$ such that
$0<\gamma(K)<+\infty$. We then choose
\begin{equation*}
\mu:=\frac{\gamma ( \cdot \cap K ) }{\gamma(K)}.
\end{equation*}
Actually, it is sufficient to prove that $\gamma(Z)=0$. Indeed, once
this is established for any real-analytic function
$f:\mathcal O\to\R$, for $x\in X$ we apply it to the shifted map
\begin{equation*}
y\in x+\mathcal O\longmapsto f(y-x),
\end{equation*}
whose zero set is exactly $x+Z$. Consequently, we would have that
$$\mu(x+Z)  \le \frac{\gamma(x+Z) }{\gamma(K) } =  0 .$$

We exclude the trivial cases where $f$ is constant or $Z$ is empty and
provide the proof that $\gamma(Z)=0$. Fix $z\in Z$.

Since $\mathcal O$ is connected, the jet of $f$ is nowhere zero. In
particular, there exist a neighbourhood of $z$ in $\mathcal O$ and an
integer $n\geq1$ such that $D^nf\neq0$ pointwise on this neighbourhood.
Since the Cameron--Martin space $H$ is dense in $X$, there exist a vector
$h\in H$ and a convex neighbourhood $\mathcal O_z\subset\mathcal O$ of
$z$ such that
\begin{equation}
\label{eq:dnf}
D^n f(y) [h,\dots,h ]\ne 0
\qquad\text{for every } y \in \mathcal{O}_z.
\end{equation}
	For every fixed $y\in  X$, set 
\begin{equation*}
I_y:=\{t\in\R:y+th\in\mathcal O_z\} \qand  Z_y:=\{t\in I_y:y+th\in Z\} \;. 
\end{equation*}
If $ I_y$ is not empty, then the function
\begin{equation*}
t\in I_y\longmapsto f(y+th)\in\R
\end{equation*}
is real analytic and, by~\eqref{eq:dnf}, is nonconstant. Its zero set is $Z_y$ which consequently has one-dimensional Lebesgue measure zero. This holds obviously in the case of $ I_y = \emptyset $ since $Z_y = \emptyset$ as well. 

Integrating over $y \in X $ and applying Tonelli's version of
Fubini's theorem gives
\begin{equation}
0
=
\int_{X}\Leb(Z_y)\,d\gamma(y)
=
\int_{\R}
\gamma\left((Z\cap\mathcal O_z)-th\right)\,dt.
\end{equation}
In particular, there exists $t\in\R$ such that
\begin{equation*}
\gamma\left((Z\cap\mathcal O_z)-th\right)=0.
\end{equation*}
By the Cameron--Martin formula, see
\cite[Corollary~2.4.3]{BogachevGaussian}, the measure $\gamma$ is
absolutely continuous with respect to the shifted measure
$\gamma(\,\cdot-th)$. Thus $ \gamma\left((Z\cap\mathcal O_z)-th\right) = 0$ implies
\begin{equation*}
\gamma(Z\cap\mathcal O_z) = 0 \;. 
\end{equation*}
Since $Z$ can be covered by countably many such open neighbourhoods
$\mathcal O_z$, it follows that $\gamma(Z)=0$. This ends the proof.
\end{proof}

                    
\bibliographystyle{alpha}
\small
\bibliography{references}

@article{Aubry,
  author  = {Serge Aubry},
  title   = {The Twist Map, the Extended {F}renkel--{K}ontorova Model and the Devil's Staircase},
  journal = {Physica D: Nonlinear Phenomena},
  volume  = {7},
  number  = {1--3},
  pages   = {240--258},
  year    = {1983},
  doi     = {10.1016/0167-2789(83)90129-X}
}

@book{Baker1975,
  author    = {Alan Baker},
  title     = {Transcendental Number Theory},
  publisher = {Cambridge University Press},
  address   = {Cambridge},
  year      = {1975},
  doi       = {10.1017/CBO9780511565977}
}

@incollection{Bangert,
  author    = {Victor Bangert},
  title     = {Mather Sets for Twist Maps and Geodesics on Tori},
  booktitle = {Dynamics Reported},
  editor    = {Urs Kirchgraber and Hans-Otto Walther},
  volume    = {1},
  pages     = {1--56},
  publisher = {B. G. Teubner and John Wiley \& Sons},
  address   = {Stuttgart and Chichester},
  year      = {1988},
  doi       = {10.1007/978-3-322-96656-8_1}
}

@article{Bastiani1964,
  author  = {Andr{\'e}e Bastiani},
  title   = {Applications diff{\'e}rentiables et vari{\'e}t{\'e}s diff{\'e}rentiables de dimension infinie},
  journal = {Journal d'Analyse Math{\'e}matique},
  volume  = {13},
  pages   = {1--114},
  year    = {1964},
  doi     = {10.1007/BF02786619}
}

@article{Bialy2022,
  author  = {Michael Bialy},
  title   = {Mather {$\beta$}-Function for Ellipses and Rigidity},
  journal = {Entropy},
  volume  = {24},
  number  = {11},
  pages   = {1600},
  year    = {2022},
  doi     = {10.3390/e24111600}
}

@article{BogachevMalofeev2014,
  author  = {Vladimir I. Bogachev and Ilya I. Malofeev},
  title   = {On the Distributions of Smooth Functions on Infinite-Dimensional Spaces with Measures},
  journal = {Doklady Mathematics},
  volume  = {89},
  number  = {1},
  pages   = {5--7},
  year    = {2014},
  doi     = {10.1134/S1064562414010025}
}

@article{CMS,
  author  = {Carlo Carminati and Stefano Marmi and David Sauzin},
  title   = {There Is Only One {KAM} Curve},
  journal = {Nonlinearity},
  volume  = {27},
  number  = {9},
  pages   = {2035--2062},
  year    = {2014},
  doi     = {10.1088/0951-7715/27/9/2035}
}

@article{CMSS,
  author  = {Carlo Carminati and Stefano Marmi and David Sauzin and Alfonso Sorrentino},
  title   = {On the Regularity of {Mather's} {$\beta$}-Function for Standard-Like Twist Maps},
  journal = {Advances in Mathematics},
  volume  = {377},
  pages   = {107460},
  year    = {2021},
  doi     = {10.1016/j.aim.2020.107460}
}

@article{DKW,
  author  = {De Simoi, Jacopo and Vadim Kaloshin and Qiaoling Wei},
  title   = {Dynamical Spectral Rigidity among {$\mathbb Z_2$}-Symmetric Strictly Convex Domains Close to a Circle},
  journal = {Annals of Mathematics},
  volume  = {186},
  number  = {1},
  pages   = {277--314},
  year    = {2017},
  doi     = {10.4007/annals.2017.186.1.7},
  note    = {Appendix B coauthored with Hamid Hezari}
}

@book{Dineen1999,
  author    = {Se{\'a}n Dineen},
  title     = {Complex Analysis on Infinite Dimensional Spaces},
  series    = {Springer Monographs in Mathematics},
  publisher = {Springer},
  address   = {London},
  year      = {1999},
  doi       = {10.1007/978-1-4471-0869-6}
}

@misc{FKT,
  author        = {Corentin Fierobe and Vadim Kaloshin and Frank Trujillo},
  title         = {Rigidity of {Mather's} {$\beta$}-Function on a {KAM} Set for Analytic Billiards-Like Maps and Unique Quasi-Analytic Continuation},
  year          = {2026},
  eprint        = {2608.15401},
  archiveprefix = {arXiv},
  primaryclass  = {math.DS},
  url           = {https://arxiv.org/abs/2608.15401},
  note          = {arXiv:2608.15401}
}

@article{GWW,
  author  = {Carolyn Gordon and David L. Webb and Scott Wolpert},
  title   = {One Cannot Hear the Shape of a Drum},
  journal = {Bulletin of the American Mathematical Society},
  volume  = {27},
  number  = {1},
  pages   = {134--138},
  year    = {1992},
  doi     = {10.1090/S0273-0979-1992-00289-6}
}

@article{Hamilton1982,
  author  = {Richard S. Hamilton},
  title   = {The Inverse Function Theorem of {Nash} and {Moser}},
  journal = {Bulletin of the American Mathematical Society (New Series)},
  volume  = {7},
  number  = {1},
  pages   = {65--222},
  year    = {1982},
  doi     = {10.1090/S0273-0979-1982-15004-2}
}

@article{HuangKaloshinSorrentino2018,
  author  = {Guan Huang and Vadim Kaloshin and Alfonso Sorrentino},
  title   = {On the Marked Length Spectrum of Generic Strictly Convex Billiard Tables},
  journal = {Duke Mathematical Journal},
  volume  = {167},
  number  = {1},
  pages   = {175--209},
  year    = {2018},
  doi     = {10.1215/00127094-2017-0038}
}

@article{MS,
  author  = {Daniel Massart and Alfonso Sorrentino},
  title   = {Differentiability of {Mather's} Average Action and Integrability on Closed Surfaces},
  journal = {Nonlinearity},
  volume  = {24},
  number  = {6},
  pages   = {1777--1793},
  year    = {2011},
  doi     = {10.1088/0951-7715/24/6/005}
}

@article{Mather82,
  author  = {John N. Mather},
  title   = {Existence of Quasi-Periodic Orbits for Twist Homeomorphisms of the Annulus},
  journal = {Topology},
  volume  = {21},
  number  = {4},
  pages   = {457--467},
  year    = {1982},
  doi     = {10.1016/0040-9383(82)90023-4}
}

@article{mather1990differentiability,
  author  = {John N. Mather},
  title   = {Differentiability of the Minimal Average Action as a Function of the Rotation Number},
  journal = {Boletim da Sociedade Brasileira de Matem\'atica},
  volume  = {21},
  number  = {1},
  pages   = {59--70},
  year    = {1990},
  doi     = {10.1007/BF01236280}
}

@book{Mujica1986,
  author    = {Jorge Mujica},
  title     = {Complex Analysis in Banach Spaces: Holomorphic Functions and Domains of Holomorphy in Finite and Infinite Dimensions},
  series    = {North-Holland Mathematics Studies},
  volume    = {120},
  publisher = {North-Holland},
  address   = {Amsterdam},
  year      = {1986}
}

@article{Poschel,
  author  = {J{\"u}rgen P{\"o}schel},
  title   = {Integrability of {Hamiltonian} Systems on {Cantor} Sets},
  journal = {Communications on Pure and Applied Mathematics},
  volume  = {35},
  number  = {5},
  pages   = {653--696},
  year    = {1982},
  doi     = {10.1002/cpa.3160350504}
}

@book{Schmeding2023,
  author    = {Alexander Schmeding},
  title     = {An Introduction to Infinite-Dimensional Differential Geometry},
  series    = {Cambridge Studies in Advanced Mathematics},
  volume    = {202},
  publisher = {Cambridge University Press},
  address   = {Cambridge},
  year      = {2023},
  doi       = {10.1017/9781009091251}
}

@incollection{Poschel2001,
  author    = {J{\"u}rgen P{\"o}schel},
  title     = {A Lecture on the Classical {KAM} Theorem},
  booktitle = {Smooth Ergodic Theory and Its Applications},
  editor    = {Anatole Katok and Rafael de la Llave and Yakov Pesin and Howard Weiss},
  series    = {Proceedings of Symposia in Pure Mathematics},
  volume    = {69},
  pages     = {707--732},
  publisher = {American Mathematical Society},
  address   = {Providence, RI},
  year      = {2001},
  doi       = {10.1090/pspum/069/1858551}
}

@article{SVeselov,
  author  = {Alfonso Sorrentino and Alexander P. Veselov},
  title   = {Markov Numbers, {Mather's} {$\beta$} Function and Stable Norm},
  journal = {Nonlinearity},
  volume  = {32},
  number  = {6},
  pages   = {2147--2156},
  year    = {2019},
  doi     = {10.1088/1361-6544/ab047d}
}

@article{Sor14,
  author  = {Alfonso Sorrentino},
  title   = {Computing {Mather's} {$\beta$}-Function for {Birkhoff} Billiards},
  journal = {Discrete and Continuous Dynamical Systems},
  volume  = {35},
  number  = {10},
  pages   = {5055--5082},
  year    = {2015},
  doi     = {10.3934/dcds.2015.35.5055}
}

@article{hunt1992prevalence,
  author  = {Brian R. Hunt and Tim Sauer and James A. Yorke},
  title   = {Prevalence: A Translation-Invariant ``Almost Every'' on Infinite-Dimensional Spaces},
  journal = {Bulletin of the American Mathematical Society},
  volume  = {27},
  number  = {2},
  pages   = {217--238},
  year    = {1992},
  doi     = {10.1090/S0273-0979-1992-00328-2},
  note    = {Addendum: Bulletin of the American Mathematical Society 28 (1993), 306--307}
}

@article{koval2026local,
  author  = {Illya Koval},
  title   = {Local Strong {Birkhoff} Conjecture and Local Spectral Uniqueness of Almost Every Ellipse},
  journal = {Inventiones Mathematicae},
  volume  = {244},
  number  = {1},
  pages   = {221--298},
  year    = {2026},
  doi     = {10.1007/s00222-025-01397-y}
}

@article{lazutkin1973existence,
  author  = {Vladimir F. Lazutkin},
  title   = {The Existence of Caustics for a Billiard Problem in a Convex Domain},
  journal = {Mathematics of the USSR-Izvestiya},
  volume  = {7},
  number  = {1},
  pages   = {185--214},
  year    = {1973},
  doi     = {10.1070/IM1973v007n01ABEH001932}
}

@incollection{levi2001lagrangian,
  author    = {Mark Levi and J{\"u}rgen Moser},
  title     = {A Lagrangian Proof of the Invariant Curve Theorem for Twist Mappings},
  booktitle = {Smooth Ergodic Theory and Its Applications},
  editor    = {Anatole Katok and Rafael de la Llave and Yakov Pesin and Howard Weiss},
  series    = {Proceedings of Symposia in Pure Mathematics},
  volume    = {69},
  pages     = {733--746},
  publisher = {American Mathematical Society},
  address   = {Providence, RI},
  year      = {2001},
  doi       = {10.1090/pspum/069/1858552}
}

@misc{li2025deformations,
  author        = {Yunzhe Li},
  title         = {Deformations of the {Standard} Map with Prescribed Actions and {Lyapunov} Exponents},
  year          = {2025},
  eprint        = {2512.03865},
  archiveprefix = {arXiv},
  primaryclass  = {math.DS},
  url           = {https://arxiv.org/abs/2512.03865},
  note          = {arXiv:2512.03865}
}

@incollection{mather2006action,
  author    = {John N. Mather and Giovanni Forni},
  title     = {Action Minimizing Orbits in {Hamiltonian} Systems},
  booktitle = {Transition to Chaos in Classical and Quantum Mechanics},
  editor    = {Sandro Graffi},
  series    = {Lecture Notes in Mathematics},
  volume    = {1589},
  pages     = {92--186},
  publisher = {Springer},
  address   = {Berlin},
  year      = {1994},
  doi       = {10.1007/BFb0074076}
}

@book{siburg2004principle,
  author    = {Karl Friedrich Siburg},
  title     = {The Principle of Least Action in Geometry and Dynamics},
  series    = {Lecture Notes in Mathematics},
  volume    = {1844},
  publisher = {Springer},
  address   = {Berlin},
  year      = {2004},
  doi       = {10.1007/978-3-540-40985-4}
}

@book{sorrentino2015action,
  author    = {Alfonso Sorrentino},
  title     = {Action-Minimizing Methods in Hamiltonian Dynamics: An Introduction to {A}ubry--{M}ather Theory},
  series    = {Mathematical Notes},
  volume    = {50},
  publisher = {Princeton University Press},
  address   = {Princeton, NJ},
  year      = {2015},
  doi       = {10.1515/9781400866618}
}

@techreport{chirikov1971research,
  title={Research concerning the theory of non-linear resonance and stochasticity},
  author={Chirikov, Boris Valerianovich},
  year={1971},
  institution={CM-P00100691}
}

@inproceedings{kolmogorov1954general,
  title={The general theory of dynamical systems and classical mechanics},
  author={Kolmogorov, AN},
  booktitle={Proceedings of the International Congress of Mathematicians},
  volume={1},
  pages={315--333},
  year={1954},
  organization={Amsterdam}
}

@article{whitney1934analytic,
  title={Analytic extensions of differentiable functions defined in closed sets},
  author={Whitney, Hassler},
  journal={Transactions of the American Mathematical Society},
  volume={36},
  number={1},
  pages={63--89},
  year={1934}
}

@article{moser1962invariant,
  title={On invariant curves of area-preserving mappings of an annulus},
  author={M{\"o}ser, Jurgen},
  journal={Nachr. Akad. Wiss. G{\"o}ttingen, II},
  pages={1--20},
  year={1962}
}

@article{mcmullen2009uniformly,
  title={Uniformly Diophantine numbers in a fixed real quadratic field},
  author={McMullen, Curtis T},
  journal={Compositio Mathematica},
  volume={145},
  number={4},
  pages={827--844},
  year={2009},
  publisher={London Mathematical Society}
}

@book{BogachevGaussian,
  title={Gaussian measures},
  author={Bogachev, Vladimir Igorevich},
  number={62},
  year={1998},
  publisher={American Mathematical Soc.}
}

@article{guillarmou2025marked,
  title={Marked length spectrum rigidity for Anosov surfaces},
  author={Guillarmou, Colin and Lefeuvre, Thibault and Paternain, Gabriel P},
  journal={Duke Mathematical Journal},
  volume={174},
  number={1},
  pages={131--157},
  year={2025},
  publisher={Duke University Press}
}

@article{HezariHamid,
	author = {Hezari, Hamid and Zelditch, Steve},
	year = {2010},
	journal = {Geometric and Functional Analysis},
	number = {1},
	pages = {160--191},
	title = {Inverse Spectral Problem for Analytic {\$}{\$}{\{}{\{}({$\backslash$}mathbb{\{}Z{\}}/2 {$\backslash$}mathbb{\{}Z{\}})\^{}{\{}n{\}}{\}}{\}}{\$}{\$}-Symmetric Domains in {\$}{\$}{\{}{\{}{$\backslash$}mathbb{\{}R{\}}\^{}{\{}n{\}}{\}}{\}}{\$}{\$}},
	volume = {20}
}

@article{fierobe2025billiard,
  title={A billiard table close to an ellipse is deformationally spectrally rigid among dihedrally symmetric domains},
  author={Fierobe, Corentin and Kaloshin, Vadim and Sorrentino, Alfonso},
  journal={arXiv preprint arXiv:2511.20062},
  year={2025}
}

@article{Z2,
  title={Spectral determination of analytic bi-axisymmetric plane domains},
  author={Zelditch, Steve},
  journal={Geometric \& Functional Analysis GAFA},
  volume={10},
  number={3},
  pages={628--677},
  year={2000},
  publisher={Springer}
}

@article{GM,
  title={An inverse spectral result for elliptical regions in R2},
  author={Guillemin, Victor and Melrose, Richard},
  journal={Advances in Mathematics},
  volume={32},
  number={2},
  pages={128--148},
  year={1979},
  publisher={Elsevier}
}

@article{FT,
  title={Rigidity of the Suris' potential in the Frenkel-Kontorova Model},
  author={Fierobe, Corentin and Tsodikovich, Daniel},
  journal={arXiv preprint arXiv:2601.15058},
  year={2026}
}

@article{suris,
  title={Integrable mappings of the standard type},
  author={Suris, Yu B},
  journal={Functional Analysis and Its Applications},
  volume={23},
  number={1},
  pages={74--76},
  year={1989},
  publisher={Springer}
}

@article{GL,
  title={The marked length spectrum of Anosov manifolds},
  author={Guillarmou, Colin and Lefeuvre, Thibault},
  journal={Annals of Mathematics},
  volume={190},
  number={1},
  pages={321--344},
  year={2019},
  publisher={JSTOR}
}

@article{Nardi,
  title={Higher order terms of Mather's $\beta$-function for symplectic and outer billiards},
  author={Baracco, Luca and Bernardi, Olga and Nardi, Alessandra},
  journal={Journal of Mathematical Analysis and Applications},
  volume={537},
  number={2},
  pages={128353},
  year={2024},
  publisher={Elsevier}
}

@article{AbbondandoloMazzucchelli,
  title={Length spectrum rigidity and flexibility of spheres of revolution with one equator},
  author={Abbondandolo, Alberto and Mazzucchelli, Marco},
  journal={arXiv preprint arXiv:2601.16804},
  year={2026}
}

@article{HZ22,
  title={One can hear the shape of ellipses of small eccentricity},
  author={Hezari, Hamid and Zelditch, Steve},
  journal={Annals of Mathematics},
  volume={196},
  number={3},
  pages={1083--1134},
  year={2022},
  publisher={Department of Mathematics, Princeton University Princeton, New Jersey, USA}
}

@article{Ruelle,
  title={Differentiation of SRB states},
  author={Ruelle, David},
  journal={Communications in Mathematical Physics},
  volume={187},
  number={1},
  pages={227--241},
  year={1997},
  publisher={Springer}
}

@article{Ruellereview,
  title={A review of linear response theory for general differentiable dynamical systems},
  author={Ruelle, David},
  journal={Nonlinearity},
  volume={22},
  number={4},
  pages={855--870},
  year={2009}
}

@inproceedings{Baladi,
  author    = {Baladi, Viviane},
  title     = {Linear Response, or Else},
  booktitle = {Proceedings of the International Congress of
               Mathematicians---Seoul 2014},
  volume    = {III},
  pages     = {525--545},
  publisher = {Kyung Moon Sa},
  address   = {Seoul},
  year      = {2014}
}

@article{GS,
  author  = {Galatolo, Stefano and Sorrentino, Alfonso},
  title   = {Quantitative Statistical Stability and Linear Response
             for Irrational Rotations and Diffeomorphisms of the Circle},
  journal = {Discrete and Continuous Dynamical Systems},
  volume  = {42},
  number  = {2},
  pages   = {815--839},
  year    = {2022},
}

@article{SKL,
  title={Marked Length Spectral determination of analytic chaotic billiards with axial symmetries},
  author={De Simoi, Jacopo and Kaloshin, Vadim and Leguil, Martin},
  journal={Inventiones mathematicae},
  volume={233},
  number={2},
  pages={829--901},
  year={2023},
  publisher={Springer}
}

@article{CD, 
  title={Sur les longueurs des trajectoires p{\'e}riodiques d'un billard},
  author={Colin de Verdi{\`e}re, Yves},
  journal={S{\'e}minaire de th{\'e}orie spectrale et g{\'e}om{\'e}trie},
  volume={1},
  pages={1--18},
  year={1984}
}
\end{document}